\documentclass[11pt,letterpaper]{amsart}

\usepackage{amssymb,amscd,amsmath,amsthm}
\usepackage[T1]{fontenc}
\usepackage{lmodern}
\usepackage{microtype}
\usepackage{graphicx}
\usepackage{mathtools}
\usepackage{hyperref}
\usepackage[shortlabels]{enumitem}
\usepackage{thm-restate}
\usepackage{mathrsfs}
\usepackage[british]{babel}
\usepackage{tikz-cd}
\usepackage{float}
\usepackage[margin=1.25in]{geometry}

\newtheorem{thm}{Theorem}[section]
\newtheorem{cor}[thm]{Corollary}
\newtheorem{lem}[thm]{Lemma}
\newtheorem{prop}[thm]{Proposition}
\newtheorem*{thm*}{Theorem}
\newtheorem*{prop*}{Proposition}

\newcounter{theoremalph}

\newtheorem{thmAlph}[theoremalph]{Theorem}
\newtheorem{corAlph}[theoremalph]{Corollary}
\newtheorem{propAlph}[theoremalph]{Proposition}
\theoremstyle{definition}
\newtheorem{defn}[thm]{Definition}
\newtheorem*{defn*}{Definition}
\theoremstyle{remark}

\newtheorem{rem}[thm]{Remark}
\newtheorem*{rem*}{Remark}
\newtheorem{exmp}[thm]{Example}

\newtheorem{quesIntro}{Question}

\newtheoremstyle{custom}{}{}{\itshape}{}{\bfseries}{.}{.5em}{#1 \thmnote{#3}}
\theoremstyle{custom}

\DeclareMathOperator{\Aut}{Aut}

\DeclareMathOperator{\Comm}{Comm}

\DeclareMathOperator{\id}{Id}

\DeclareMathOperator{\im}{Im}
\DeclareMathOperator{\lcm}{lcm}
\DeclareMathOperator{\Isom}{Isom}

\DeclareMathOperator{\stab}{stab}
\DeclareMathOperator{\CAT}{CAT}

\DeclareMathOperator{\Sym}{Sym}

\DeclareMathOperator{\GL}{GL}

\DeclareMathOperator{\Orth}{O}

\DeclareMathOperator{\Helly}{Helly}

\DeclareMathOperator{\PSL}{PSL}
\DeclareMathOperator{\QI}{QI}

\DeclareMathOperator{\qa}{{\overset{\scriptscriptstyle\text{q.a.}}{\curvearrowright}}}

\newcommand{\cc}[1]{\overset{#1}{\rightsquigarrow}}
\newcommand{\close}[1]{\overset{#1}{\sim}}
\newcommand{\nclose}[1]{\overset{#1}{\not\sim}}
\newcommand{\Haus}{{\text{Haus}}}

\DeclareMathOperator{\QZ}{QZ}

\newfont{\cusfont}{alnorm}
\newcommand{\alnorm}{\text{\;\cusfont Q\; }}

\newcommand{\bbH}{\mathbb{H}}
\newcommand{\bbN}{\mathbb{N}}

\newcommand{\bbQ}{\mathbb{Q}}
\newcommand{\bbR}{\mathbb{R}}
\newcommand{\bbZ}{\mathbb{Z}}

\newcommand{\cB}{\mathcal{B}}
\newcommand{\cC}{\mathcal{C}}

\newcommand{\cU}{\mathcal{U}}

\author{Alex Margolis}
\title[Discretisable quasi-actions I]{Discretisable quasi-actions I:\\ Topological completions and hyperbolicity}
\address{Alex Margolis, Department of Mathematics, Vanderbilt University, 1326 Stevenson Center, Nashville, TN 37240, USA}
\email{alexander.margolis@vanderbilt.edu}
\thanks{This research was partially supported by the Israel Science Foundation (grant No. \texttt{1562/19}).}

\begin{document}
\begin{abstract}
	We define and develop the notion  of a  discretisable quasi-action. It is shown that a cobounded quasi-action on a proper non-elementary hyperbolic space $X$ not fixing a point of $\partial X$ is   quasi-conjugate to an isometric action on either a rank one symmetric space or a locally finite graph.  Topological completions of quasi-actions are also introduced.
	Discretisable quasi-actions are used to give several instances where  commensurated subgroups  are  preserved by quasi-isometries. For example, the class of $\bbZ$-by-hyperbolic groups is shown to be quasi-isometrically rigid. We characterise  the class of finitely generated groups quasi-isometric to either $\bbZ^n\times \Gamma_1$ or $\Gamma_1\times \Gamma_2$, where $\Gamma_1$ and $\Gamma_2$ are  non-elementary hyperbolic groups.
\end{abstract}
\maketitle
\section{Introduction}
 One of the central problems of geometric group theory is to  understand to what extent the large-scale geometry of a finitely generated group  determines its  algebraic structure. 
 A \emph{$\bbZ$-by-hyperbolic group} is a finitely generated group $G$ containing  an infinite cyclic normal subgroup $\bbZ\cong H\vartriangleleft G$ such that the quotient $G/H$ is a non-elementary hyperbolic group. The following theorem is an application of the  methods developed in this article:
\begin{thmAlph}[Corollary \ref{cor:zbyhyp}]\label{thm:z-by-hyp}
	A finitely generated group that is quasi-isometric to a $\bbZ$-by-hyperbolic group is also $\bbZ$-by-hyperbolic.
\end{thmAlph}

\noindent Theorem \ref{thm:z-by-hyp} is complemented by earlier work of Gersten \cite{gersten92bounded} and Kapovich--Kleiner--Leeb  \cite{kapovich1998derham}, from which it can be deduced that two $\bbZ$-by-hyperbolic groups are quasi-isometric if and only if their hyperbolic quotients are quasi-isometric. 

Special cases of Theorem \ref{thm:z-by-hyp} were already known. Kapovich--Leeb and  Rieffel independently proved that a finitely generated group is quasi-isometric to $\bbR\times \bbH^2$ if and only if it is of the form $\bbZ$-by-$Q$, where $Q$ is a finite extension of a cocompact Fuchsian group \cite{kapovichleeb97qidecomp,rieffel01}. This was an important step in understanding  finitely generated groups  quasi-isometric to  the fundamental groups of  3-manifolds. Kleiner--Leeb later  proved a special case of Theorem \ref{thm:z-by-hyp} for groups of the form $\bbZ$-by-$Q$, where $Q$ is a finite extension of a uniform lattice in a symmetric space \cite{kleinerleeb2001symmetric}. 
Mosher--Sageev--Whyte  showed that a finitely generated group quasi-isometric to $\bbZ\times F_2$ is virtually $\bbZ\times F_m$ for some $m\geq 2$  \cite{whyte2001baumslag,mosher2003quasi}, proving another special case of Theorem \ref{thm:z-by-hyp}.

The common thread needed to prove these  special cases of Theorem \ref{thm:z-by-hyp} is the notion of  a quasi-action.  A \emph{quasi-action} of a group $G$ on a metric space $X$, denoted $G\qa X$, consists of a uniform collection of quasi-isometries $\{f_g\}_{g\in G}$ of $X$ such that $f_{\id}$ is uniformly close to the identity map, and the distance between $f_{gh}$ and $f_g\circ f_h$ in the sup metric is uniformly bounded for all $g,h\in G$. Two quasi-actions $G\qa X$ and $G\qa Y$ are   \emph{quasi-conjugate} if there is a coarsely $G$-equivariant quasi-isometry from $X$ to $Y$. 
To characterise groups quasi-isometric to $\bbR\times \bbH^2$, Kapovich--Leeb and  Rieffel use the fact that a quasi-action on $\bbH^2$ can be quasi-conjugated to an isometric action on $\bbH^2$.  Mosher--Sageev--Whyte prove that any cobounded quasi-action on a locally finite bushy tree can be quasi-conjugated to an isometric action on a possibly different locally finite bushy tree.

This  article introduces a framework for studying  quasi-actions on more general spaces, including all proper non-elementary hyperbolic spaces and all Cayley graphs of finitely presented groups of cohomological dimension two. We introduce the notion of a discretisable quasi-action and  give numerous applications, including Theorem \ref{thm:z-by-hyp}.
 The central definition of this article is the following:
\begin{defn*}%
		Let $X$ be a  metric space. A cobounded quasi-action $G\qa X$ is said to be \emph{discretisable} if  it is  quasi-conjugate to an action $G\curvearrowright Y$, where $Y$ is a connected locally finite graph and $G$ acts on $Y$ by graph automorphisms. We say the space $X$ is \emph{discretisable} if every cobounded quasi-action on $X$ is discretisable.
\end{defn*}

 As we  shortly see, the notion of a discretisable quasi-action is much more natural, robust and useful than the previous definition might initially suggest. To explain why this is the case, we first note that in general, there is nothing particularly remarkable about the  graph $Y$ in the preceding definition --- all we know about $Y$ is that it is a locally finite graph quasi-isometric to $X$.  The crux of  the matter  is not so much the intrinsic local geometry of $Y$, but rather what one can say about $G$ by knowing that  it acts on a locally finite graph.  
  Here are two such properties  that will motivate the subsequent discussion:
 \begin{enumerate}[(I)]
 	\item \label{item:topcomp} Corestricting $\phi:G\rightarrow \Aut(Y)$ to the closure of the image gives a homomorphism  $G\rightarrow \hat G$ with dense image, where $\hat G$ is a totally disconnected locally compact group.
 	\item \label{item:cstab} The stabiliser $H$ of a vertex of $Y$ is a commensurated subgroup of $G$, i.e.\ every conjugate of $H$ is commensurable to $H$.
 \end{enumerate}  The homomorphism $G\rightarrow \hat G$ in \ref{item:topcomp} is an example of what we call a   topological completion, whilst 
the subgroup $H$ in \ref{item:cstab} is an example of what we call a  coarse stabiliser.  
Discretisability can be reformulated  in terms of topological completions and coarse stabilisers, which we formally define shortly.
\newcommand{\discequivs}{	
	Let $X$ be a quasi-geodesic metric space and let $G\qa X$ be a cobounded quasi-action. The following are equivalent:
	\begin{enumerate}
		\item $G\qa X$ is discretisable;
		\item $G\qa X$ has a coarse stabiliser;
		\item the topological completion of $G\qa X$ exists and  is compact-by-(totally disconnected).
	\end{enumerate}
}
\begin{propAlph}[Proposition \ref{prop:disc_equiv}]
	\label{prop:disc_equiv_intro}
	\discequivs
\end{propAlph}
Switching between these equivalent formulations  is essential for both proving and applying discretisability. We now define and discuss the terms in Proposition \ref{prop:disc_equiv_intro}, indicating how  discretisability can be used to  prove Theorem \ref{thm:z-by-hyp} and similar theorems.

\subsection*{Topological completions of quasi-actions}
 Given a quasi-action $G\qa X$, we say a subset $T\subseteq G$ is \emph{bounded} if for some (hence every)  $x\in X$, the quasi-orbit $T\cdot x\coloneqq \{t\cdot x\mid t\in T\}$  is bounded. If a topological group $G$ acts continuously, properly and isometrically on a proper metric space $X$, then $T\subseteq G$ has compact closure if and only if it has bounded orbits.  This motivates the following definition:
 
\begin{defn*}
	A \emph{topological completion of a quasi-action} $G\qa X$ is a pair $(\hat G , \rho)$ where:
	\begin{enumerate}
		\item $\hat G$ is a  locally compact  group; 
		\item  $\rho:G\rightarrow \hat G$  is a homomorphism with dense image;
		\item for every $T\subseteq G$,  $T$ is bounded if and only if $\overline{\rho(T)}$ is compact.
	\end{enumerate}
\end{defn*}
 The following theorem shows that topological completions are unique. We refer  the reader  to Theorem \ref{thm:topcomp unique} for  a more precise statement.
\begin{thmAlph}[Theorem \ref{thm:topcomp unique}]\label{thm:topcomp unique_intro}
	Topological completions are unique modulo a compact normal subgroup.
\end{thmAlph}
It thus makes sense to talk about \emph{the} topological completion of a quasi-action. Topological completions need not exist. Indeed, in Section \ref{sec:qmorph} we show that non-trivial quasi-morphisms  give rise to  cobounded quasi-actions on $\bbR$ that do not have topological completions. Nonetheless, we demonstrate that for a very large class of spaces, including proper non-elementary hyperbolic spaces,  cobounded quasi-actions \emph{always} have topological completions. The following theorem is inspired by work of Furman \cite{furman2001mostowmargulis} and Whyte \cite{whyte01coarse}.
\begin{thmAlph}[Corollary \ref{cor:morse_topcomp_exist}]\label{thm:topcomp_acyl_intro}
	If $X$ is a cocompact proper geodesic metric space whose Morse boundary contains at least three points,   then any cobounded quasi-action on $X$ has a topological completion.
\end{thmAlph}	
\noindent More generally, cobounded quasi-actions on tame metric spaces always have topological completions; see Section \ref{sec:tame} for details. Armed with the structure theory of locally compact groups,  topological completions are a powerful new tool in geometric group theory.

 \subsection*{Coarse stabilisers of quasi-actions}
Discretisability can be characterised  in terms of coarse stabilisers, which are coarse analogues of compact open subgroups. 
 \begin{defn*}
 	A subgroup  $H\leq G$ is a \emph{coarse stabiliser} of $G\qa X$ if:
 \begin{enumerate}
 	\item $H$ is bounded;
 	\item every bounded subset of $G$ is contained in finitely many left $H$-cosets.
 \end{enumerate}
 \end{defn*}
 If $G\qa X$ has a coarse stabiliser, then it is  unique up to commensurability (Proposition \ref{prop:coarsestabsarecomm}).  
 A subgroup $H\leq G$ is said to be \emph{commensurated}, denoted $H\alnorm G$,  if every conjugate of $H$ is commensurable to $H$. Commensurated subgroups are also known as almost normal, inert  or near normal subgroups. If $G$ is  finitely generated  relative to a commensurated subgroup $H\alnorm G$,  the \emph{quotient space} is the set of left cosets $G/H$ equipped with the relative word metric; see Section \ref{sec:alnorm} for details. The quotient space is a proper discrete metric space that is  well-defined up to quasi-isometry. Moreover, if $H$ is normal, then the quotient space coincides with the quotient group when equipped with the word metric.
 
 	A basic fact of Lie theory is that if a Lie group $G$ acts smoothly and transitively on a smooth manifold $M$ with isotropy subgroup $H$, then the homogeneous space $G/H$   is  diffeomorphic to $M$ via a diffeomorphism that conjugates the action of $G$ on $G/H$ to the action of $G$ on $M$.
The following proposition is a coarse geometric analogue of this fact in the setting of cobounded quasi-actions admitting a  coarse stabiliser, where the associated homogeneous space $G/H$ is discrete. This proposition also generalises  the  Milnor--Schwarz lemma, which can be recovered in the case  $H$ is trivial.
 \begin{propAlph}[Proposition \ref{prop:cstab homspace}]\label{prop:relms_intro}
 	Let $X$ be a quasi-geodesic metric space and let $G\qa X$ be a cobounded quasi-action with coarse stabiliser $H$. Then:
 	\begin{enumerate}
 		\item $G$ is finitely generated relative to $H$;
 		\item $H\alnorm G$;
 		\item $G\qa X$ is quasi-conjugate to the natural  left action $G\curvearrowright G/H$. In particular, $G/H$ is quasi-isometric to $X$.
 	\end{enumerate}	
 \end{propAlph}
 Proposition \ref{prop:relms_intro} furnishes us with a blueprint for proving Theorem \ref{thm:z-by-hyp} and similar theorems. Indeed, suppose $G$ is a finitely generated group quasi-isometric to a group of the form $\bbZ$-by-$Q$, where $Q$ is a non-elementary hyperbolic group. A theorem of Kapovich--Kleiner--Leeb ensures $G$ admits a cobounded  quasi-action on $Q$ \cite{kapovich1998derham}. If this quasi-action  were discretisable (see Theorem \ref{thm:trichotomyhyp_intro}), then Propositions \ref{prop:disc_equiv_intro} and \ref{prop:relms_intro} would  imply that $G\qa Q$ has a coarse stabiliser $H\alnorm G$ such that the associated quotient space $G/H$ is quasi-isometric to $Q$. All that remains is to deduce that  $H$ --- which is necessarily a  two-ended  commensurated subgroup --- is commensurable to a normal subgroup. 
 
 \subsection*{Discretisable quasi-actions and spaces}
 A hyperbolic metric space is said to be \emph{non-elementary} if its Gromov boundary contains at least three points.  We  exhibit the following trichotomy for 	quasi-actions on proper hyperbolic spaces,  partially  resolving  a conjecture of  Gromov \cite[0.3.C]{gromov1987hyperbolic}.
 \newcommand{\trichotomyhyp}{	Let $X$ be a proper non-elementary  hyperbolic metric space. If $G\qa X$ is a cobounded quasi-action, then one of the following holds:
 	\begin{enumerate}
 		\item \textbf{(connected case)} $G\qa X$ is quasi-conjugate to a cocompact isometric action on a homogeneous
 		simply-connected negatively curved manifold of dimension at least two.
 		\item \textbf{(mixed case)} $G\qa X$  is quasi-conjugate to a cocompact isometric action on a pure millefeuille space.
 		\item \textbf{(totally-disconnected case)} $G\qa X$  is discretisable. 
 \end{enumerate}}
 
 \begin{thmAlph}[Theorem \ref{thm:trichotomyhyp_main}]\label{thm:trichotomyhyp_intro}
 \trichotomyhyp
 \end{thmAlph}

 We briefly elaborate on the continuous and mixed cases of Theorem \ref{thm:trichotomyhyp_intro}. Homogeneous negatively curved manifolds were classified by Heintze as a particular family of connected solvable Lie groups equipped with left-invariant metrics \cite{heintze74homogeneous}.   Caprace--Cornulier--Monod--Tessera    define a class of $\CAT(-1)$ spaces known as \emph{millefeuille spaces}, which are a certain sort of warped product of a tree with a homogeneous negatively curved manifold \cite{caprace2015amenable,cornulier2018quasi}. A millefeuille space is said to be \emph{pure} if it is neither  a tree nor a homogeneous negatively curved manifold. Such spaces exhibit a mixture of continuous behaviour coming from the transitive action on the  manifold, and discrete behaviour coming from the action on the tree with discrete orbits.

A quasi-action $G\qa X$ on a proper hyperbolic space $X$ naturally extends to an action of $G$  on the Gromov boundary $\partial X$. Using work of Caprace--Cornulier--Monod--Tessera \cite{caprace2015amenable}, we have a complete understanding of the case where $G$ fixes a point of $\partial X$, which can occur only when $X$  is quasi-isometric to a millefeuille space (see Corollary \ref{cor:trichot_hyp_fixpt}).  Under the additional hypothesis that $G$ does not fix a point of $\partial X$, we can sharpen the conclusion of  Theorem \ref{thm:trichotomyhyp_intro}:
 
 \newcommand{\dichothyp}{Let $X$ be a proper non-elementary  hyperbolic metric space. If $G\qa X$ is a cobounded quasi-action that does not fix a point of $\partial X$, then one of the following holds:
 	\begin{enumerate}
 		\item \textbf{(connected case)} $G\qa X$ is quasi-conjugate to a cocompact isometric action on a rank one symmetric space of non-compact type.
 		\item \textbf{(totally-disconnected case)} $G\qa X$  is discretisable. 
 \end{enumerate}	}
 
 \begin{corAlph}[Corollary \ref{cor:dichothyp_main}]\label{cor:dichothyp_intro}
 	\dichothyp
\end{corAlph}

An iconic theorem of Mosher--Sageev--Whyte  asserts  one can quasi-conjugate a cobounded  quasi-action on a finite valance bushy tree to an isometric action on such a tree \cite[Theorem 1]{mosher2003quasi}.  
Mosher--Sageev--Whyte's proof proceeds in two stages. In the first stage, which is the heart of their argument,  they use \emph{quasi-edges} to quasi-conjugate  the quasi-action to an isometric   action on a locally finite graph. In other words, they show precisely that cobounded quasi-actions on finite valance bushy trees are discretisable.  For the second stage, which is more classical,  they add 2-cells to this graph  and then apply Dunwoody's tracks technology   to obtain an action on a tree \cite{dunwoody1985accessibility}. 
Theorem \ref{thm:trichotomyhyp_intro} thus provides a far-reaching generalisation  of the crucial first step of Mosher--Sageev--Whyte's argument.

 Discretisability is not a phenomenon particular to groups exhibiting coarse negative curvature, and can often be deduced via  group cohomology. We will not pursue this line of investigation in this article, but  state a result, to appear in future work, that further illustrates the ubiquity of discretisable groups. A  \emph{generalised Baumslag--Solitar  group} is the fundamental group of a finite graph of groups in which every vertex and edge group is infinite cyclic, and the associated Bass--Serre tree is infinite-ended.
 
 \begin{thmAlph}[\cite{margolis2022discretisableII}]\label{thm:trichcodim2}
 	Let  $\Gamma$ be a  finitely presented group  of cohomological dimension two. Then exactly one of the following holds:
 	\begin{enumerate}
 		\item \textbf{(connected case)} $\Gamma$ is the fundamental group of a closed surface of non-positive Euler characteristic.
 		\item \textbf{(mixed case)} $\Gamma$ is a generalised Baumslag--Solitar group.
 		\item \textbf{(totally-disconnected case)} every cobounded quasi-action $G\qa \Gamma$ is discretisable.
 	\end{enumerate}
 \end{thmAlph}

A natural problem in geometric group theory, raised by the author in \cite{margolisxer2021geometry}, is to find examples of groups $G$ and subgroups $H$ such that the following question has a positive solution:
\begin{quesIntro}\label{ques:comm}
Suppose a finitely generated group $G$ contains an infinite normal (or commensurated) subgroup $H$. If $G'$ is a finitely generated group quasi-isometric to $G$, does $G'$ contain a normal (or commensurated) subgroup $H'$ such that $H'$ is quasi-isometric to $H$ and $G/H$ is quasi-isometric to $G'/H'$? 
\end{quesIntro}
Normal and commensurated subgroups give rise to a \emph{coarse bundle} structure on the group. A key step in finding positive solutions to Question  \ref{ques:comm} is showing that fibres of this coarse bundle structure are preserved by quasi-isometries.
Positive answers to special cases of Question \ref{ques:comm}  formed a key step in several quasi-isometric rigidity results, including work of Farb--Mosher, Mosher--Sageev--Whyte, Eskin--Fisher--Whyte, Peng and the author \cite{farbmosher1998bs1,farbmosher1999bs2,farbmosher2000abelianbycyclic,farbmosher2002surfacebyfree,mosher2003quasi,whyte2001baumslag,whyte2010coarse,eskinfisherwhyte12coarse,eskinfisherwhyte2013coarse,peng2011coarse,peng2011coarse2,margolisxer2021geometry}. 
Although Question \ref{ques:comm}  is of course false in full generality, discretisability provides a framework to prove there are many further instances where this question has a positive answer. In particular, we use discretisability to prove Theorems  \ref{thm:qialcent_intro} and \ref{thm:hypprod_comm_intro}, providing a positive solution to Question \ref{ques:comm} for groups quasi-isometric to central extensions of hyperbolic groups and products of hyperbolic groups respectively.

\subsection*{Groups quasi-isometric to central extensions of hyperbolic groups}
By a \emph{central extension of a hyperbolic group} we mean a finitely generated group $G$ that is a central extension of the form $H$-by-$Q$, where $H$ is a  finitely generated abelian central subgroup of $G$, and $Q=G/H$ is a non-elementary hyperbolic group. Without loss of generality, we can always assume   $H$ is free abelian. These groups are a well-studied class of groups that were shown to be biautomatic by Neumann--Reeves \cite{neumannreeves97central}. 

The ``internal'' quasi-isometry classification of central extensions of hyperbolic groups is already well-understood. Indeed,  work of Neumann--Reeves  \cite{neumannreeves97central} shows a central extension of a hyperbolic group  of the form $\bbZ^n$-by-$Q$ is quasi-isometric to $\bbR^n\times Q$. Combined with a result of   Kapovich--Kleiner--Leeb \cite{kapovich1998derham}, this implies  that if $Q$ and $Q'$ are non-elementary hyperbolic groups, then the central extensions  $\bbZ^n$-by-$Q$ and $\bbZ^m$-by-$Q'$ are quasi-isometric if and only if $n=m$ and $Q$ is quasi-isometric to $Q'$. We thus investigate the  class of  finitely generated  groups quasi-isometric to central extensions of hyperbolic groups, or equivalently, the class of finitely generated groups quasi-isometric to $\bbR^n\times Q$.

We first focus on  $\bbZ$-by-hyperbolic groups, which --- up to passing to an index two subgroup if necessary --- are central extensions. The  class is a rich source of groups  that are quasi-isometric but have different algebraic properties. For instance, Gersten showed that uniform lattices in $\bbR\times \bbH^2$ and $\widetilde{\PSL}(2,\bbR)$ are quasi-isometric but not commensurable, demonstrating that two of Thurston's eight model geometries are quasi-isometric \cite{gersten92bounded}. The class of  $\bbZ$-by-hyperbolic groups also contains  two quasi-isometric groups $G$ and $G'$  such that $G$ has Kazhdan's property (T) but $G'$ does not, thus demonstrating property (T) is not a quasi-isometry invariant \cite[\S 19.9]{drutu2018geometric}. 

Theorem \ref{thm:z-by-hyp} says the class of $\bbZ$-by-hyperbolic groups is  quasi-isometrically rigid, giving many instances in which Question \ref{ques:comm} has a positive solution.
Somewhat surprisingly, the technology developed in this article ensures the proof of Theorem \ref{thm:z-by-hyp}  is actually easier in the generic case when the  hyperbolic quotient is \textbf{not} virtually isomorphic to a  uniform lattice in a rank one symmetric space. This is because in the discretisable case, after applying \cite{kapovich1998derham}, one gets a two-ended commensurated subgroup for free via  Propositions \ref{prop:disc_equiv_intro}, \ref{prop:relms_intro} and Corollary \ref{thm:trichotomyhyp_intro}. Such a subgroup can then be shown to be commensurable to a  normal subgroup. For the case in which the quotient is not discretisable, one has to work much harder to construct a two-ended normal subgroup via a growth estimate \cite[\S 5--6]{kleinerleeb2001symmetric}.

For higher rank free abelian central subgroups, we obtain an analogous result if we  assume residual finiteness. We recall that a group is \emph{residually finite} if the intersection of all finite index subgroups is trivial.
\begin{thmAlph}[Theorem \ref{thm:resfinite_cent}]\label{thm:qicent_resfin_intro}
A  finitely generated residually finite group that is quasi-isometric to a central extension of a hyperbolic group of the form $\bbZ^n$-by-$Q$ is virtually a central extension of a hyperbolic group $\bbZ^n$-by-$Q'$, where $Q$ is quasi-isometric to $Q'$.
\end{thmAlph}
Leary--Minasyan recently described a family of   CAT(0) groups that are not biautomatic \cite{learyminasyan2021commensurating}. Such groups are quasi-isometric to $\bbZ^n\times F_2$,   but are not virtually central extensions of hyperbolic groups and are not residually finite. These groups demonstrate that  the residual finiteness hypothesis in Theorem \ref{thm:qicent_resfin_intro} is necessary.  The groups constructed in \cite{learyminasyan2021commensurating} are  HNN extensions of $\bbZ^n$ such that the stable letter acts on a finite index subgroup of $\bbZ^n$ as  an orthogonal matrix in $\Orth_n(\bbQ)$. Theorem \ref{thm:qialcent_intro} demonstrates that similar behaviour occurs for all groups quasi-isometric to central extensions of hyperbolic groups, thus providing a complete algebraic characterisation of such groups.

A subgroup $H\leq G$ is said to be \emph{weakly separable} if for every $g\in G\setminus H$, there exists a homomorphism $\phi:G\rightarrow Q$ such that $\phi(H)$ is finite and $\phi(g)\notin \phi(H)$;  we note  $Q$ is not assumed to be finite. A separable subgroup is necessarily weakly separable. The notion of a  weakly separable subgroup is due to Caprace--Kropholler--Reid--Wesolek \cite{caprace_kropholler_reid_wesolek_2020}.
If $H\vartriangleleft G$ is a normal subgroup, the action of $G$ by conjugation on $H$ induces a homomorphism $G\rightarrow \Aut(H)$. More generally, if $H\alnorm G$ is a commensurated subgroup, then conjugation  induces the \emph{modular homomorphism}  $G\rightarrow \Comm(H)$, where $\Comm(H)$ is the abstract commensurator  of $H$. The abstract commensurator of $\bbZ^n$ is $\GL_n(\bbQ)$.
\begin{thmAlph}[Theorem \ref{thm:qitocent_main} and Proposition \ref{prop:qitocent_main}]\label{thm:qialcent_intro}
	Let $Q$ be a non-elementary locally finite vertex-transitive   hyperbolic graph and let  $n\geq 1$. A finitely generated group $G$ is quasi-isometric to $\bbR^n\times Q$ if and only if both the following hold:
	\begin{enumerate}
		\item there is a subgroup $\bbZ^n\cong H\alnorm G$ such that the quotient space $G/H$ is quasi-isometric to $Q$;
		\item the image of the modular homomorphism $G\rightarrow \Comm(H)\cong \GL_n(\bbQ)$ is conjugate to a subgroup of $\Orth_n(\bbR)$.
	\end{enumerate} 
Moreover, the following are equivalent:
\begin{itemize}
	\item $H$  is weakly separable;
	\item $H$ is commensurable to a normal subgroup of $G$;
	\item a finite index subgroup of $G$ centralises a finite index subgroup of $H$;
	\item $G$ is biautomatic.
\end{itemize}
\end{thmAlph} 
\begin{rem*}
	It follows from Lemma \ref{lem:weaksepvssep} that weak separability of $H$ in Theorem \ref{thm:qialcent_intro}  is equivalent to separability of $H$ under the additional hypothesis that   all finitely generated  groups quasi-isometric to $Q$ are residually finite.  It is a well-known open question whether all hyperbolic groups are residually finite.
\end{rem*}
We briefly  outline the proof of Theorem \ref{thm:qialcent_intro}. If $G$ is a finitely generated group quasi-isometric to $\bbR^n\times Q$, then a result of Kapovich--Kleiner--Leeb implies the quasi-action of $G$ on $\bbR^n\times Q$  descends to a quasi-action of $G$ on $Q$. We show that this quasi-action is discretisable, hence has some coarse stabiliser $H$. We now apply the formalism of \emph{coarse bundles} as described in Section \ref{sec:cbundles}. In particular,  Proposition \ref{prop:quasi-conj_coarse bundle} can be applied to show the coarse bundle $p:G\to G/H$ is fibre-preserving quasi-isometric to $\bbR^n\times Q\to Q$. This implies $H$ is quasi-isometric to $\bbR^n$, hence has a finite index subgroup isomorphic to $\bbZ^n$. Properties  $(1)$ and $(2)$ readily follow.   
Conversely, if $(1)$ and $(2)$ hold, we can apply a result of Mineyev concerning $\ell_\infty$-cohomology to deduce the coarse bundle $p:G\to G/H$ has a coarse Lipschitz section \cite{mineyev00isoperimetric}. Together with  $(1)$ and $(2)$, this can be used to show  $G$ is quasi-isometric to  $H\times G/H$, hence to $\bbR^n\times Q$.

\subsection*{Groups quasi-isometric to products of hyperbolic spaces}
We characterise finitely generated groups quasi-isometric to the product of finitely many non-elementary hyperbolic spaces. A group acting isometrically on a product $\Pi_{i=1}^nY_i$ is said to   \emph{preserve the product structure} if every isometry splits as a product of isometries, possibly permuting the factors. The following theorem generalises results of Kleiner--Leeb \cite{kleinerleeb97rigidity,kleinerleeb09quasiactions} and Ahlin \cite{ahlin02producttrees}, who prove  special cases where the hyperbolic factors are quasi-isometric to symmetric spaces or trees.
  \begin{thmAlph}[Theorem \ref{thm:hyp_productmain}]\label{thm:hyp_productintro}
 	Let $\Gamma$ be a finitely generated group quasi-isometric to $\Pi_{i=1}^nX_i$, where each $X_i$ is a cocompact proper non-elementary hyperbolic metric space. Then $\Gamma$ acts geometrically on $\Pi_{i=1}^nY_i$, preserving the product structure, where each $Y_i$ is quasi-isometric to $X_i$ and is either a rank one symmetric space of non-compact type or a locally finite graph.
 \end{thmAlph}

It is well-known that there are finitely generated groups quasi-isometric to the direct product of two infinite hyperbolic  groups that do not virtually split as a direct product. Examples include uniform irreducible lattices in $\bbH^2\times \bbH^2$ and  $\Aut(T)\times\Aut(T)$ for $T$ a finite valence regular tree. Such examples demonstrate that the conclusion of Theorem \ref{thm:hyp_productintro} cannot in general be strengthened to conclude $\Gamma$ virtually splits as a direct product.  The reason why such lattices exist is that both $\bbH^2$ and $T$  are spaces with a large amount of symmetry. We expect most hyperbolic groups and spaces do not possess such symmetry, and in such situations, Theorem \ref{thm:hyp_productintro} would imply  direct products of hyperbolic groups are virtually preserved by quasi-isometries; see also Theorem \ref{thm:hypprod_comm_intro}.

A graph is \emph{Helly} if any family of pairwise intersecting combinatorial balls have non-trivial intersection. A group is \emph{Helly} if it acts geometrically on a Helly graph \cite{chalopin2020helly}. Helly groups have a robust nonpositive-curvature-like  structure, possessing  many properties also possessed by $\CAT(0)$ groups. Moreover, Helly groups are biautomatic even though $\CAT(0)$ groups might not be \cite{learyminasyan2021commensurating}.  We refer the reader to \cite{chalopin2020helly} for additional  properties of Helly groups. An isometric action on  a locally finite hyperbolic graph can be conjugated to an action on a Helly graph; see Proposition \ref{prop:helly}. Combining this observation with  Theorem \ref{thm:hyp_productintro}, we  show the following:
\begin{thmAlph}[Theorem \ref{thm:hypprod_helly}]\label{thm:hypprod_helly_intro}
If $\Gamma$ is a finitely generated group quasi-isometric to $\Pi_{i=1}^nX_i$, where each $X_i$ is a cocompact proper non-elementary hyperbolic metric space \textbf{not} quasi-isometric to a rank one symmetric space,  then $\Gamma$ is Helly. In particular, $\Gamma$ is  biautomatic. 
\end{thmAlph}

	Intriguingly, Theorem \ref{thm:hypprod_helly_intro} is false if we drop  the hypothesis that no $X_i$ is quasi-isometric to a symmetric space. Hughes and Valiunas have an  example of a group that is  not biautomatic   acting geometrically on $\bbH^2\times T$, where $T$ is a finite valence regular tree  \cite{hughes2022commensurating}. This example illustrates an essential difference between discretisable and non-discretisable hyperbolic groups. In addition to the existence of coarse stabilisers, Theorem \ref{thm:hypprod_helly_intro} gives compelling further  evidence why  in some circumstances,  the discrete and  combinatorial nature of    locally finite graphs --- a priori  with rather  unexciting  local geometry ---  has concrete advantages over spaces with  richer local geometry such as symmetric spaces.

We now refine Theorem \ref{thm:hyp_productintro} by deducing the  existence of natural  commensurated subgroups and quotient spaces that  capture  the algebraic structure of  groups quasi-isometric to products of hyperbolic spaces, providing further instances in which Question \ref{ques:comm} (virtually)  has a positive solution.   For simplicity, we restrict to products of two factors. The case where both factors are quasi-isometric to symmetric spaces is well-understood by work of Kleiner--Leeb \cite{kleinerleeb97rigidity}, so we  assume at least one factor is not quasi-isometric a symmetric space.

\begin{thmAlph}[Theorem \ref{thm:qiprod_comm}]\label{thm:hypprod_comm_intro}
	Let $X$ and $Y$ be cocompact  non-elementary proper hyperbolic metric spaces such that  $Y$ is \textbf{not} quasi-isometric to a rank one symmetric space. Suppose $\Gamma$ is a finitely generated group quasi-isometric to $X\times Y$. 
		Then there exists a finitely generated subgroup $\Lambda\leq \Gamma$,  commensurated by a finite index subgroup  $\Gamma^*\leq \Gamma$ such that:
	\begin{enumerate}
		\item $\Lambda$ is quasi-isometric to $X$;
		\item $\Gamma^*/\Lambda$ is quasi-isometric to $Y.$
	\end{enumerate} Moreover,  $\Lambda$ is weakly separable in $\Gamma^*$ if and only if  $\Gamma$ is virtually isomorphic to the direct product $\Gamma_X\times \Gamma_Y$, where $\Gamma_X$ and $\Gamma_Y$ are finitely generated groups quasi-isometric to $X$ and $Y$ respectively.
\end{thmAlph}
\begin{rem*}
	By Lemma \ref{lem:weaksepvssep}, the weak separability of $\Lambda$ in Theorem \ref{thm:hypprod_comm_intro}  is equivalent to separability of $\Lambda$ under the additional hypothesis that  all finitely generated  groups quasi-isometric to $Y$ are residually finite. 
\end{rem*}
We recall that subgroup separability was used by  Wise to study irreducible lattices in products of trees. Wise showed that such lattices are virtually products if and only if certain ``horizontal'' and ``vertical'' subgroups are separable  \cite[Theorem 5.5]{wise96thesis}; see also \cite[Corollary 32]{caprace_kropholler_reid_wesolek_2020}.  The subgroup $\Lambda$ in the preceding theorem plays an analogous role. 
We note that in  the case $X$  is also not quasi-isometric to a rank one symmetric space, we can reverse the roles of $X$ and $Y$ in Theorem \ref{thm:hypprod_comm_intro} to obtain a commensurated subgroup quasi-isometric to $Y$.

We state two theorems that allow the conclusion of Theorem \ref{thm:hyp_productintro} to be strengthened under additional algebraic hypotheses. These results rely on ``abstract arithmeticity'' theorems of Caprace--Monod \cite[\S 5]{capracemonod2009isometry} which assert that under suitable hypotheses, irreducible lattices in products of locally compact groups are  arithmetic lattices in semisimple algebraic groups.   
In order to rule out irreducible arithmetic lattices, we say a hyperbolic space is \emph{of coarse algebraic type} if it is quasi-isometric to either  a rank one symmetric space of non-compact type, or to a rank one Bruhat--Tits building, i.e. to a $k$-regular tree for $k\geq 3$.
\begin{thmAlph}[Theorem \ref{thm:mixedlattice_main}]\label{thm:mixedlattice_intro}
	Let $\Gamma$ be a  finitely generated residually finite group quasi-isometric to $X\times Y$, where  $X$ and $Y$ are cocompact  non-elementary proper  hyperbolic metric spaces. If $X$ is quasi-isometric to an irreducible rank one symmetric space and $Y$ is \textbf{not} of coarse algebraic type, then  $\Gamma$ virtually splits as a product $\Gamma_X\times \Gamma_Y$, where $\Gamma_X$ is quasi-isometric to $X$ and $\Gamma_Y$ is quasi-isometric to $Y$.
\end{thmAlph}
We can remove the assumption that $X$ is quasi-isometric to a symmetric space if we strengthen the residual finiteness hypothesis. For the purposes of the following theorem, a group is \emph{linear}  if it has a faithful finite-dimensional linear representation over a field of characteristic $\neq 2,3$. This hypothesis on the characteristic is needed to induce a representation that is admissible in the sense of Margulis \cite[\S IX]{margulis1991discrete}.
 \begin{thmAlph}[Theorem \ref{thm:splitlinear_main}]\label{thm:splitlinear_intro}
 		Let $\Gamma$ be a  finitely generated linear group quasi-isometric to $X\times Y$, where  $X$ and $Y$ are cocompact non-elementary proper  hyperbolic metric spaces. If at least one of $X$ or  $Y$ is \emph{not} of coarse algebraic type, then  $\Gamma$ virtually splits as a product $\Gamma_X\times \Gamma_Y$, where $\Gamma_X$ is quasi-isometric to $X$ and $\Gamma_Y$ is quasi-isometric to $Y$.
 \end{thmAlph}
\subsection*{Boundary rigidity of groups quasi-isometric to products of hyperbolic spaces}
A $\CAT(0)$ group $\Gamma$ is said to be \emph{boundary rigid} if the boundaries of any two proper $\CAT(0)$ spaces admitting a geometric $\Gamma$-action are homeomorphic.  Croke--Kleiner show there exists a finitely generated group $\Gamma$ acting geometrically on proper $\CAT(0)$ spaces $Y_1$ and $Y_2$ such that the visual boundaries $\partial Y_1$ and $\partial Y_2$ are not homeomorphic \cite{crokekleiner2000spaces}. In particular, such a group $\Gamma$ is not boundary rigid. 

 Ruane showed that direct products of hyperbolic groups are boundary rigid \cite{ruane1999boundaries}. Recent work of Jankiewicz--Karrer--Ruane--Sathaye shows that a group acting freely and  vertex-transitively on the product of two trees  is boundary rigid \cite{jankiewicz2021boundary}.  Using Theorem \ref{thm:hyp_productintro} and work of Monod and Caprace--Monod \cite{monod2006superrigidity,capracemonod2009strucure}, we generalise these results by showing boundary rigidity for groups quasi-isometric to products of hyperbolic spaces.

\begin{thmAlph}[Theorem \ref{thm:boundary_rigidity}]\label{thm:boundary_rigidity_intro}
	Let $\Gamma$ be a finitely generated quasi-isometric to $\Pi_{i=1}^nX_i$, where each $X_i$ is a cocompact proper  non-elementary hyperbolic metric space. Suppose $\Gamma$ acts geometrically on a proper $\CAT(0)$ space $Y$. Then the visual boundary $\partial Y$ is homeomorphic to the join  $\partial X_1*\dots*\partial X_n$, where $\partial X_i$ is the Gromov boundary of $X_i$. In particular, $\partial Y$ depends only on the quasi-isometry type of $\Gamma$ and not the choice of  $\CAT(0)$ space $Y$.
\end{thmAlph}
\subsection*{Outline}
Section \ref{sec:preliminaries} contains preliminary results. Section \ref{sec:coarse geom}  consists of elementary definitions and properties concerning coarse geometry and quasi-actions. In Section \ref{sec:alnorm}, we prove some basic facts about commensurated subgroups and the geometry of pairs of groups. In Section \ref{sec:cbundles}, we define coarse bundles as used in \cite{whyte2010coarse,margolisxer2021geometry}, and reformulate results of Kapovich--Kleiner--Leeb and Kleiner--Leeb in the language of coarse bundles \cite{kapovichleeb97qidecomp,kleinerleeb09quasiactions}; this material will not be used till Sections \ref{sec:hyp_prod} and \ref{sec:qicent}.  In Section \ref{sec:topgps},  we summarise  some  facts  about the metric geometry of locally compact groups.

In Sections \ref{sec:disc_quasi_actions} and \ref{sec:topcomp} we develop the general theory of discretisable quasi-actions. In Section \ref{sec:disc_quasi_actions} we prove some basic results about  coarse stabilisers, proving Proposition \ref{prop:relms_intro} and showing the equivalence of the first two items of Proposition \ref{prop:disc_equiv_intro}. In Section \ref{sec:topcomp}, we throw topological groups into the mix,  defining and developing the notion of a topological completion of a quasi-action. The basic theory of topological completions is developed in Section \ref{sec:geom_topcomp}, whilst Section \ref{sec:topcomp_unique} shows topological completions are essentially unique. In Section \ref{sec:topcomp_disc} we give a  topological characterisation of discretisable quasi-actions, thus completing our proof of Proposition \ref{prop:disc_equiv_intro}. In Section \ref{sec:tame}, we  prove Theorem \ref{thm:topcomp_acyl_intro}. Section \ref{sec:qmorph}  demonstrates that non-trivial quasi-morphisms give rise to  quasi-actions on $\bbR$ that do not have topological completions. 

	In Sections \ref{sec:qaction_hyp} -- \ref{sec:qicent} we specialize to  quasi-actions on hyperbolic spaces. In Section \ref{sec:qaction_hyp} we prove Theorem \ref{thm:trichotomyhyp_intro} and its corollaries. We make use of results of Caprace--Cornulier--Monod--Tessera on the structure of locally compact hyperbolic groups \cite{caprace2015amenable}. In Section \ref{sec:hyp_prod} we investigate the structure of groups quasi-isometric to the product of hyperbolic spaces, proving Theorems \ref{thm:hyp_productintro}--\ref{thm:boundary_rigidity_intro}. In Section \ref{sec:qicent} we characterise groups quasi-isometric to the products of a hyperbolic group with $\bbR^n$, proving Theorems \ref{thm:z-by-hyp}, \ref{thm:qicent_resfin_intro} and \ref{thm:qialcent_intro}.

\subsection*{Acknowledgements} The author would like to thank Michah Sageev,  Sam Shepherd, Emily Stark and Daniel Woodhouse for  helpful   discussions relating to this work.

\section{Preliminaries}\label{sec:preliminaries}

\subsection{Coarse geometric preliminaries}\label{sec:coarse geom}
In this subsection we give some basic definitions of coarse metric geometry. We develop  terminology and notation relating to quasi-actions, and prove  a generalisation of the Milnor--Schwarz lemma for non-proper quasi-actions.

\begin{defn}
	Let $(X,d)$ be a metric space. For any $r\geq 0$ and $A\subseteq X$, let \[N_r(A)\coloneqq \{x\in X\mid d(x,y)\leq r \text{ for some $y\in A$}\}.\] The \emph{Hausdorff distance} between any two subsets $A,B\subseteq X$ is defined to be \[d_\Haus(A,B)\coloneqq\inf\{r\mid A\subseteq N_r(B)\text{ and } B\subseteq N_r(A)\}.\]
\end{defn}

There are various notions of coarse maps between metric spaces that we use. 
\begin{defn}
	Let $f:X\rightarrow Y$ be a map between metric spaces.
\begin{itemize}
	\item Given constants $K\geq 1$ and $A\geq 0$, we say that $f$ is \emph{$(K,A)$-coarse Lipschitz} if 
	\[ d(f(x),f(x'))\leq Kd(x,x')+A\] 
	for all $x,x'\in X$. We say that $f$ is  \emph{coarse Lipschitz} if there exist constants $K\geq 1$ and $A\geq 0$ such that $f$ is  $(K,A)$-coarse Lipschitz.
	\item We say that $f$ is a \emph{$(K,A)$-quasi-isometric-embedding} if 
			\[\frac{1}{K}d(x,x')-A\leq d(f(x),f(x'))\leq Kd(x,x')+A\] 
		for all $x,x'\in X$.
	We say that $f$ is a \emph{quasi-isometric embedding} if there exist constants $K\geq 1$ and $A\geq 0$ such that $f$ is a $(K,A)$-quasi-isometric embedding.
	\item We say that $f$ is a \emph{$(K,A)$-quasi-isometry} if $f$ is a $(K,A)$-quasi-isometric embedding such that $Y=N_A(f(X))$. We say $f$ is a \emph{quasi-isometry} if there exist constants $K\geq 1$ and $A\geq 0$ such that $f$ is a $(K,A)$-quasi-isometry.
	\item We say that $f$ is a \emph{$K$-bi-Lipschitz equivalence}, if it is a $(K,0)$-quasi-isometry. We say that $f$ is a \emph{bi-Lipschitz equivalence} if it is a $K$-bi-Lipschitz equivalence for some $K\geq 1$.
	\item Given a proper non-decreasing function $\eta:\bbR_{\geq 0}\rightarrow \bbR_{\geq 0}$ and constants $K\geq 1$ and $A\geq 0$, we  say that $f$ is an \emph{$(\eta,K,A)$-coarse embedding} if  
		\[\eta(d(x,x'))\leq d(f(x),f(x'))\leq Kd(x,x')+A \] for all $x,x'\in X$.  We say that $f$ is a \emph{coarse embedding} if there exist $\eta$, $K$ and $A$  as above such that $f$ is an $(\eta,K,A)$-coarse embedding.
\end{itemize}
\end{defn}

\begin{rem}\label{rem:inv_control}
 For each proper non-decreasing function $\eta : \bbR_{\geq 0}\to \bbR_{\geq 0}$, we define
 $\tilde{\eta} : \bbR_{\geq 0}\to \bbR_{\geq 0}$ by $\tilde \eta(S) :=\sup(\eta^{-1} ([0, S]))$. We observe that
whenever $\eta(S) \leq R$, then $S \leq \tilde \eta(R)$. Conversely, if $S < \eta (R)$, then $\tilde \eta(S) \leq R$.
\end{rem}

Being quasi-isometric is an equivalence relation among metric spaces. If $G$ is a finitely generated group, then any two word metrics on $G$ with respect to arbitrary finite generating sets are bi-Lipschitz equivalent. Moreover, any two Cayley graphs of $G$ with respect to finite generating sets are quasi-isometric. Unless explicitly  stated otherwise, we assume finitely generated groups are equipped with a word metric with respect to a  finite generating set. If $G$ is finitely generated and $H\leq G$ is finitely generated, then the inclusion $H\to G$ is a coarse embedding when $H$ and $G$ are equipped with word metrics with respect to finite generating sets.

If $X$ and $Y$ are metric spaces, two functions $f,g:X\to Y$  are said to be \emph{$A$-close} if $\sup_{x\in X}d(f(x),g(x))\leq A$, and are \emph{close} if they are $A$-close for some $A$. An \emph{$A$-coarse inverse} to $f:X\to Y$ is a function $h:Y\to X$ such that $h\circ f$ and $f\circ h$ are $A$-close to $\id_X$ and $\id_Y$ respectively. We say that $h$ is a \emph{coarse inverse} to $f$ if it is an $A$-coarse inverse for some $A$.
If $X$ is a metric space, then its \emph{quasi-isometry group} $\QI(X)$ consists of equivalence classes of quasi-isometries from $X$ to $X$, where $f,g:X\to X$ are equivalent if they are close. The quasi-isometry group $\QI(X)$ forms a group with multiplication defined by $[f][g]=[f\circ g]$. Moreover, if $h$ is a coarse inverse  to $f$, then $[f]^{-1}=[h]$.

 In a metric space $X$, an \emph{$A$-chain of length $n$} between $x,x'\in X$ consists of a sequence $x=x_0, x_1, \dots, x_n=x'$ in $X$ such that $d(x_{i-1},x_{i})\leq A$ for $i=1,\dots, n$. The following classes of  metric spaces  will be used throughout this article.
\begin{defn}
	Let $X$ be a metric space. 
	\begin{itemize}
		\item We say that $X$ is \emph{proper} if closed balls in $X$ are compact.
		\item We say that $X$ is \emph{locally finite} if for every $r\geq 0$ and $x\in X$, $\lvert N_r(x)\rvert<\infty$.
		\item   We $X$ is \emph{coarsely connected} if there exists a constant $A$ such that every $x,x'\in X$ can be joined by an $A$-chain.
		\item We say that $X$ is  \emph{$(K,A)$-quasi-geodesic} if any $x,x'\in X$ can be joined by a $A$-chain of length at most $Kd(x,x')+A$.   We say that $X$ is   \emph{quasi-geodesic} if there exists  $K,A\geq 1$ such that $X$ is  $(K,A)$-quasi-geodesic.
	\end{itemize}
\end{defn}
A metric space is  quasi-geodesic if and only if it is quasi-isometric to a geodesic metric space, i.e.\ to a space in which any pair of points can be joined by an isometrically embedded arc. The following lemma is well-known; see e.g.\ \cite[Proposition 1.A.1.]{cornulierdlH2016metric}.
\begin{lem}\label{lem:fg<->cc}
	Let $G$ be a finitely generated group equipped with the word metric with respect to a finite generating set. Let $H\leq G$ be a subgroup equipped with the subspace metric. Then $H$ is finitely generated if and only if $H$ is coarsely connected.
\end{lem}

A connected graph can be considered a metric space when equipped  with the induced path metric in which each edge has length one. A graph is \emph{locally finite} if each vertex is incident to only finitely many edges. A connected graph is a locally finite graph if and only if it is a proper metric space. We caution that a locally finite graph with at least one edge is not locally finite as a metric space. 
\begin{prop}[{\cite[Proposition 3.D.11]{cornulierdlH2016metric}}]\label{prop:coarse_proper}
	Let $X$ be a metric space. The following are equivalent:
	\begin{enumerate}
		\item $X$ is quasi-isometric to a proper  quasi-geodesic metric space;
		\item $X$ is quasi-isometric to a quasi-geodesic locally finite metric space;
		\item $X$ is quasi-isometric to a locally finite connected graph.
	\end{enumerate}
\end{prop}

\begin{lem}\label{lem:coarse_lip}
	Let $X,Y$ be metric spaces such that $X$ is $(K,A)$-geodesic. Suppose $f:X\to Y$ is a map and $B$ is a constant  such that  $d_Y(f(x),f(x'))\leq B$ whenever $d_X(x,x')\leq A$. Then $f$ is coarse Lipschitz.
\end{lem}
\begin{proof}
	Let $x,x'\in X$. Let $x=x_0,x_1,\dots, x_n=x'$ be an $A$-chain from $x$ to $x'$ with $n\leq Kd_X(x,x')+A$. Then  $d_Y(f(x),f(x'))\leq n B\leq BKd_X(x,x')+AB$. Thus $f$ is coarse Lipschitz.
\end{proof}

The following lemma is  a straightforward generalisation of \cite[Lemma 2.1]{farbmosher2000abelianbycyclic} from geodesic to quasi-geodesic metric spaces:
\begin{lem}[{ \cite[Lemma 2.1]{farbmosher2000abelianbycyclic}}]\label{lem:qi_subspace}
	Let $X$, $Y$, $Z$ and $W$ be proper quasi-geodesic metric spaces and let $s:Z\to X$ and $s':W\to Y$ be coarse embeddings. If $f:X\to Y$ is a quasi-isometry such that $d_\Haus(f(\im(s)),\im(s'))<\infty$, then there exists a quasi-isometry $h:Z\to W$ such that \[\sup_{z\in Z}d(f(s(z)),s'(h(z)))<\infty.\]
\end{lem}

We now prove  the  following coarse analogue of the Arzel\`a--Ascoli Theorem.

\begin{lem}[Coarse Arzel\`a--Ascoli Theorem]\label{lem:coarse AA}
	Let $X$ be a locally finite metric space. Suppose that for some $K\geq 1$ and $A\geq 0$, there exist a sequence of $(K,A)$-quasi-isometries $(f_i)$
	of $X$  such that $\{f_i(x_0)\mid i\in \bbN\}$ is bounded for some $x_0\in X$.
	Then there exists a $(K,A)$-quasi-isometry $f:X\rightarrow X$ and a subsequence $(f_{n_i})$  such that for all $x\in X$, there exists an $N_x$ such that $f_{n_i}(x)=f(x)$ for all $i\geq N_x$. In particular, $(f_{n_i})$ converges pointwise to $f$.
\end{lem}

\begin{proof}
	We first show that for every $x\in X$,  $T_x\coloneqq \{f_i(x)\mid i\in \bbN\}$ is finite. By hypothesis $T_{x_0}\subseteq B_R(f_0(x_0))$ for some sufficiently large $R$. For each $x\in X$, we thus see that for all $i\in \bbN$ \[d(f_0(x_0),f_i(x))\leq d(f_0(x_0),f_i(x_0))+d(f_i(x_0),f_i(x))\leq R+Kd(x,x_0)+A,\] and so   $T_x$ is bounded. Since $X$ is locally finite, each $T_x$ is  finite.
	
	As $X$ is locally finite and hence is countable,  we can pass to a subsequence $(f_{n_i})$ such that for each $x\in X$, the sequence $(f_{n_i}(x))$ is eventually constant and so eventually equal to some $f_x\in X$.  We define $f:X\to X$ by $f(x)=f_x$ and  claim that $f$ is a $(K,A)$-quasi-isometry. Indeed, for any $x,y\in X$,  $f(x)$ and $f(y)$ are equal to $f_{n_i}(x)$ and $f_{n_i}(y)$ respectively for $i$ sufficiently large. Since  each $f_{n_i}$ is a $(K,A)$-quasi-isometry, it follows that $f$ is a $(K,A)$-quasi-isometric embedding.
	
	Suppose $y\in X$ and let $y_0\coloneqq f(x_0)$.  Since each $f_{n_i}$ is a $(K,A)$-quasi-isometry, there exists a $z_i\in X$ such that $d(f_{n_i}(z_i),y)\leq A$. Thus for $i$ sufficiently large, $f_{n_i}(x_0)=y_0$ and so
	\begin{align*}d(z_i,x_0)\leq 
		K\Big( d(f_{n_i}(z_i),f_{n_i}(x_0))+A\Big) 
		\leq K\Big( d(y,y_0)+2A\Big)\eqqcolon R_y. 			
	\end{align*}
	Since $N_{R_y}(x_0)$ is finite, there exists an $x\in N_{R_y}(x_0)$ such that $d(f_{n_i}(x),y)\leq A$ for all $i$ sufficiently large.  Thus $d(f(x),y)\leq A$ and so $f$ is a $(K,A)$-quasi-isometry.
\end{proof}

\begin{defn}
	Let $G$ be a group and $X$ be a metric space. For $K\geq 1$ and $A\geq 0$, a  \emph{$(K,A)$-quasi-action} of $G$ on $X$ is a function $\phi:G\rightarrow X^X$ such that 
	\begin{enumerate}
		\item for all $g\in G$,  $\phi(g)$ is a $(K,A)$-quasi-isometry;
		\item for all $g,h\in G$ and $x\in X$, $d\big(\phi(h)(\phi(g)(x)),\phi(hg)(x)\big)\leq A$;
		\item for all  $x\in X$, $d(\phi(e)(x),x)\leq A$.
	\end{enumerate}
	We say that $\phi:G\rightarrow X^X$ is a \emph{quasi-action}  if it is a $(K,A)$-quasi-action for some $K$ and $A$.
	
	We say that a quasi-action $\phi:G\rightarrow  X^X$ is
	\emph{$r$-cobounded} if for every $x,y\in X$, there exists a $g\in G$ such that $d(g\cdot x,y)\leq r$. We say a quasi-action is \emph{cobounded} if it is $r$-cobounded for some $r\geq 0$. We say $\phi$ is \emph{proper} if for every $x\in X$ and $r\geq 0$, the set $\{g\in G\mid d(g\cdot x,x)\leq r\}$ is finite.
\end{defn}

To limit cumbersome notation, we will frequently suppress the actual quasi-action $\phi:G\rightarrow X^X$, simply denoting $\phi(g)(x)$ as $g\cdot x$.  Since function composition is associative, there is no need to use  brackets with this notation, but we emphasise that $h\cdot g\cdot x=\phi(h)(\phi(g)(x))$ should not be confused with $hg\cdot x=\phi(hg)(x)$. When we do not need to specify the quasi-action $\phi:G\rightarrow X^X$ explicitly, we simply write $G\qa X$ to denote a quasi-action of $G$ on $X$.

The following lemma follows easily from the definition of a quasi-action:
\begin{lem}\label{lem:qaction_gp}
	Suppose $G\qa X$ is a $(K,A)$-quasi-action. Then for any $g,g_1,g_2\in G$ and $x,y\in X$, 
	\[d(gg_1\cdot x,gg_2\cdot y)\leq d(g\cdot g_1\cdot x,g\cdot g_2\cdot y) +2A \leq Kd(g_1\cdot x, g_2\cdot y)+3A. \] 
\end{lem}

\begin{defn}
	A  \emph{quasi-conjugacy} between two quasi-actions $G\qa X$ and $G'\qa Y$ consists of a  homomorphism $\rho:G\rightarrow G'$ and  a $(K,A)$-quasi-isometry $f:X\rightarrow Y$  such that for every $x\in X$ and $g\in G$, we have \[d(\rho(g)\cdot f(x),f(g\cdot x))\leq A.\]
	\end{defn}

If we wish to specify the homomorphism $\rho$ and the  constants $K$ and $A$ as above, we say that $f$ is a   \emph{$(\rho,K,A)$-quasi-conjugacy}. If we wish to specify exactly one of either the function $\rho$ or  the constants $K$ and $A$,  we  use the terms $\rho$-quasi-conjugacy and $(K,A)$-quasi-conjugacy as appropriate. If  $G=G'$ and no homomorphism $\rho:G\to G$ is explicitly mentioned, a quasi-conjugacy is always assumed to be an $\id_G$-quasi-conjugacy.

Quasi-conjugacies arise naturally in the context of group actions. Suppose $\rho:G\rightarrow \Isom(X)$ is an isometric  action and $f:X\rightarrow Y$ is a quasi-isometry with coarse inverse $\overline f$. Then $f$ is easily seen to be a quasi-conjugacy from the action $\rho:G\rightarrow \Isom(X)$ to the quasi-action $g\mapsto f\circ \rho(g)\circ \overline f$ on $Y$. 
Moreover, quasi-conjugacy is an equivalence relation among the collection of quasi-actions of a fixed group $G$.

The following proposition is a variant of  the well-known Milnor--Schwarz lemma, which has been called the \emph{fundamental observation of geometric group theory}. Similar generalisations to  actions that are not proper are present in the literature ---  see \cite[Theorem 4.C.5]{cornulierdlH2016metric} and \cite[Lemma 3.11]{aboottetal2019hypstructures} --- although the following theorem is  more  general as it is phrased in terms of quasi-actions.
\begin{prop}\label{prop:generalms}
	Suppose $(X,d)$ is a quasi-geodesic metric space and $G\qa X$ is a cobounded quasi-action. Fix  a basepoint $b\in X$. Then there exists an $R_0$ such that the following holds: 
	\begin{enumerate}
		\item For any $R\geq R_0$,  $G$ is generated by any set $S\subseteq G$ such that   \[\{g\in G\mid d(g\cdot b,b)\leq R_0\}\subseteq S\subseteq \{g\in G\mid d(g\cdot b,b)\leq R\}.\]
		\item Let $d_S$ denote the word metric on $G$ with respect to some $S$ as above. 
		Then the isometric left action $G\curvearrowright (G,d_S)$ is quasi-conjugate to  $G\qa X$ via the \emph{quasi-orbit map} $g\mapsto g\cdot b$. 
		In particular, $(G,d_S)$ is quasi-isometric to $X$.
	\end{enumerate} 
\end{prop}
\begin{proof}
	We can choose constants $K\geq 1$ and $A\geq 0$ sufficiently large  such that the following hold:
	\begin{enumerate}
		\item $X$ is a $(K,A)$-quasi-geodesic metric space;
		\item $G\qa X$ is a $(K,A)$-quasi-action;
		\item for every $x\in X$, there exists a $g_x\in G$ such that $d(g_x\cdot b,x)\leq A$.
	\end{enumerate} 
	Set $R_0\coloneqq 3KA+3A$ and fix $R\geq R_0$ and $S\subseteq G$ as in the statement. Let $f:G\rightarrow X$ be the quasi-orbit map $g\mapsto g\cdot b$.
	
	Let $g,k\in G$. As $X$ is quasi-geodesic, there exists an $A$-chain $g\cdot b=x_0,x_1, \dots, x_m=k\cdot b$, where $m\leq Kd(g\cdot b,k\cdot b)+A$. 
	Let $g_0=g$, $g_m=k$ and $g_i\coloneqq g_{x_i}$ for $i\neq 0,m$. Thus for every $0\leq i\leq m$, we have $d(g_i\cdot b,x_i)\leq A$.  
	Using Lemma \ref{lem:qaction_gp}, we see that
	\begin{align*}
	d(b,g_{i-1}^{-1}g_i\cdot b)& \leq  Kd(g_{i-1}\cdot b, g_i\cdot b)+3A\\
	& \leq Kd(x_{i-1},x_i)+2KA+3A \leq 3KA+3A= R_0,
	\end{align*}
	and so $g_{i-1}^{-1}g_i\in S$. Since $g^{-1}k=(g_0^{-1}g_1)(g_1^{-1}g_2)\dots (g_{m-1}^{-1}g_m)$, we  deduce that $G$ is generated by $S$. 
	Moreover, since $m\leq Kd(g\cdot b,k\cdot b)+A$, we see  \begin{align}
	d_S(g,k)\leq K d(f(g),f(k))+A,\label{eqn:genms_qi1}
	\end{align} for all $g,k\in G$. 
	
	Now suppose $d_S(g,k)=n$, so we can write $g^{-1}k=s_1\dots s_n$ for some $s_1,\dots, s_n\in S$. Setting $t_i=s_1s_2\dots s_i$ for each $i$, we see by another application of Lemma \ref{lem:qaction_gp} that \[d(gt_{i-1}\cdot b, gt_{i}\cdot b) \leq Kd(b,s_i\cdot b)+3A\leq KR+3A\eqqcolon R'\] for every $1\leq i\leq n$. Thus  \begin{align}
		d(f(g),f(k))=d( g\cdot b,k\cdot b)\leq \sum_{i=1}^n d(gt_{i-1}\cdot b, gt_{i}\cdot b)\leq  R'n=R'd_S(g,k)\label{eqn:genms_qi2}
	\end{align} for all $g,k\in G$. Together, (\ref{eqn:genms_qi1}) and (\ref{eqn:genms_qi2}) imply that  $f:G\rightarrow X$ is a quasi-isometric embedding. 
	
	Since for every $x\in X$, $d(f(g_x),x)\leq A$, we see $f$ is a quasi-isometry.
	Finally, for all $g,k\in G$, we have $d(k\cdot f(g),f(kg))=d(k\cdot g\cdot b,kg\cdot b)\leq A.$ Therefore, $f$ is a quasi-conjugacy.
\end{proof}
By considering the action of $G$ on its Cayley graph with respect to a generating set $S$ as in Proposition \ref{prop:generalms}, we deduce:
\begin{cor}\label{cor:qa, qconj to isoma}
	An arbitrary cobounded quasi-action $G\qa X$ on a quasi-geodesic metric space is quasi-conjugate to an isometric and vertex-transitive action of $G$ on a connected (perhaps locally infinite) graph.
\end{cor}
Corollary \ref{cor:qa, qconj to isoma} is a general nonsense argument that cannot --- as far as the author is aware --- be used to say anything nontrivial about either the group $G$ or the space $X$. 
A more general statement of this sort was proven by Kleiner--Leeb \cite[Corollary 1.1]{kleinerleeb09quasiactions}. In contrast,  quasi-conjugating a quasi-action to an isometric action on a \emph{locally finite} graph is a   non-trivial matter, and as this article demonstrates,   has several non-trivial consequences.

\subsection{Commensurated subgroups and quotient spaces}\label{sec:alnorm}
In this subsection, we define and investigate the relative word metric and relative Cayley graphs of pairs $(G,H)$, where $G$ is a group and $H$ is a subgroup.   In the case $G$ is finitely generated relative to $H$, the relative Cayley graph  is well-defined up to quasi-isometry. The relative Cayley graph is shown to be locally finite if and only if $H$ is commensurated. We also characterise when a commensurated subgroup is commensurable to normal subgroup. 
The results and proofs in this subsection are minor variations of statements already in the literature, typically  stated under additional hypothesis that $G$ is  finitely generated or has the structure of a topological group; see e.g. \cite{kronmoller08roughcayley,connermihalik2013}. We do not make these assumptions here.

\begin{defn}
If $G$ is a group and $H\leq G$ is a subgroup, a  \emph{relative  generating set} of the pair $(G,H)$ is a set $S\subseteq G$ such that $S\cup H$ generates $G$. We say that \emph{$G$ is finitely generated relative to $H$} if $(G,H)$ has a finite relative generating set.
\end{defn}
\begin{rem}
If $G$ is  finitely generated, then it is finitely  generated relative to any subgroup $H\leq G$.
\end{rem}

Given a group $G$, a subgroup $H$ and a non-empty set $S\subseteq G$, we define   a simplicial  graph  $\Gamma_{G,H,S}$ with vertex  set
$\{gH\mid g\in H\}$ and edge set \[\{(gH,gsH)\mid g\in G, s\in S\}.\] In particular, $gH$ and $kH$ are joined by an edge if and only there exist $h,h'\in H$ and $s\in S$ such that either $ghs=kh'$ or $gh=kh's$. It is clear  $G$ acts on $\Gamma_{G,H,S}$ by graph automorphisms.

\begin{prop}
The graph $\Gamma=\Gamma_{G,H,S}$ is connected if and only if $S$ is a relative generating set of $(G,H)$.
\end{prop}
\begin{proof}
Without loss of generality, we may assume  that $S$ is symmetric, i.e.\  $S=S\cup S^{-1}$.
We first suppose that $\Gamma$ is connected. For every $g\in G$, $H$ and $gH$ can be joined by an edge path $g_0H,g_1H,\dots, g_nH$ in $\Gamma$ where $g_0=1_G$ and $g_n=g$. For each $i$, there exist $h_i\in H$ and $s\in S$ such that $g_ih_isH=g_{i+1}H$, and so $g_i^{-1}g_{i+1}$ is contained in the subgroup $\langle H,S\rangle_G$. Thus $g=(g_0^{-1}g_1)(g_1^{-1}g_2)\dots (g_{n-1}^{-1}g_n)$ is also contained in $\langle H,S\rangle_G$.

Conversely, suppose  $G=\langle H,S\rangle_G$. Then every $g\in G$ can be written in the form  $h_1s_1h_2s_2\dots h_ns_nh_{n+1}$, where for each $i$, $h_i\in H$ and $s_i\in S$. Setting $g_i=h_1s_1h_2s_2\dots h_is_i$, we see that \[H,g_1H, g_2H, \dots, g_nH=gH\] is a path in $\Gamma$ from $H$ to $gH$, so $\Gamma$ is connected.
\end{proof}
If $S$ is a relative generating set, we call $\Gamma_{G,H,S}$ the \emph{relative Cayley graph} of the pair $(G,H)$.
The \emph{relative word metric} on $G/H$ is the metric $d_S$ such that $d_S(gH,kH)$ is the minimal number of edges from  $gH$ to $kH$ in the graph $\Gamma_{G,H,S}$.
Although the relative Cayley graph is sometimes called a (left) Schreier graph,  we do not use this terminology here to avoid confusion with the graph whose edge set is $\{(gH,sgH)\mid s\in S, g\in G\}$, also called the Schreier graph.  

\begin{prop}\label{prop:word metric bilip}
If $S$ and $S'$ are two finite relative  generating sets of the pair $(G,H)$, then $(G/H,d_S)$ and $(G/H,d_{S'})$ are bi-Lipschitz equivalent. Moreover, the relative Cayley graphs $\Gamma_{G,H,S}$ and $\Gamma_{G,H,S'}$ are quasi-isometric.
\end{prop}
\begin{proof}
We pick $K\geq 1$ sufficiently large such that $d_S(H,s'H)\leq K$ for every $s'\in S'$ and  $d_{S'}(H,sH)\leq K$ for every $s\in S$. It follows from the triangle inequality that \[\frac{1}{K}d_{S'}(gH,kH)\leq d_S(gH,kH)\leq Kd_{S'}(gH,kH)\] for every $gH,kH\in G/H$.
Since $(G/H,d_S)$ and $(G/H,d_{S'})$ are bi-Lipschitz equivalent nets in  $\Gamma_{G,H,S}$ and $\Gamma_{G,H,S'}$ respectively, the relative Cayley graphs are quasi-isometric.
\end{proof}
One consequence of Proposition \ref{prop:word metric bilip} is that geometric properties of the relative word metric that are invariant under bi-Lipschitz equivalence depend only on the pair $(G,H)$ and not the choice of relative generating set.

We recall two subgroups $H,K\leq G$ are  \emph{commensurable} if the intersection $H\cap K$ has finite index in both $H$ and $K$. Commensurability is an equivalence relation among subgroups of $G$.  A subgroup $H\leq G$ is said to be \emph{commensurated}, denoted $H \alnorm G$, if every conjugate of $H$ is commensurable to $H$. A subgroup commensurable to a commensurated subgroup is also commensurated.
 Commensurated subgroups can be characterised as follows:
\begin{prop}[cf.\ {\cite[Theorem 4.5]{connermihalik2013}}]\label{prop:relcg_locfinite}
Suppose $S$ is a finite relative generating set of the pair $(G,H)$. Then $\Gamma=\Gamma_{G,H,S}$ is locally finite if and only if $H\alnorm G$. 
\end{prop}
\begin{proof}
We  assume without loss of generality that $S$ is symmetric. The set $\Lambda\coloneqq \{hsH\mid h\in H, s\in S\}$ is the set of  vertices of $\Gamma$ that are adjacent to $H$. Since $G$ acts transitively on $\Gamma$, we see that $\Gamma$ is locally finite if and only if  $\Lambda_s\coloneqq \{hsH\mid h\in H\}$ is finite for every $s\in S$.  Now $\Lambda_s$ is finite if and only if  $s^{-1}Hs$ is contained in finitely many left $H$-cosets. We thus see $\Gamma$ is locally finite if and only if $s^{-1}Hs$ and $H$ are commensurable for every $s\in S$. Since $H$ and $S$ generate $G$, $H \alnorm G$ if and only if $s^{-1}Hs$ is commensurable to $H$ for every $s\in S$. Thus $H \alnorm G$ if and only if $\Gamma$ is locally finite.
\end{proof}

\begin{rem}
	When $G$ is finitely generated, $\Gamma_{G,H}$ is quasi-isometric to the \emph{coned--off Cayley graph} as defined by Farb \cite{farb1998relhyp}.  The relative Cayley graph $\Gamma_{G,H}$ is thus hyperbolic precisely when $G$ is weakly hyperbolic relative to $H$, i.e.\ $G$ is  hyperbolic relative to $H$ in the sense of \cite{farb1998relhyp}. If $H\alnorm G$, then the bounded coset penetration property of \cite{farb1998relhyp} fails, so $G$ is never hyperbolic relative to $H$ in the sense of \cite{bowditch2012relhyp}.
\end{rem}

\begin{defn}
	If $H\alnorm G$ and $G$ is finitely generated relative to $H$, we call the left coset space $G/H$, equipped with the relative word metric, a \emph{quotient space}. We call the map $G\to G/H$ given by $g\mapsto gH$ the \emph{quotient map}.
\end{defn}
\noindent By Proposition \ref{prop:word metric bilip}, the quotient space is well-defined up to bi-Lipschitz equivalence.   We caution the reader that  the quotient map is not a homomorphism unless $H$ is  normal.

Given a commensurated subgroup $H$, a natural question to consider is how close $H$ is to a normal subgroup.  A subgroup $H\leq G$ is \emph{weakly separable} if it is the intersection of virtually normal subgroups of $G$ containing $H$.  Using a result of  Caprace et al.\ \cite{caprace_kropholler_reid_wesolek_2020}, we  show the following:
\begin{lem}\label{lem:weaksep}
	Let $G$ and $H\alnorm G$ be finitely generated groups. The following are equivalent:
	\begin{enumerate}
		\item $H$ contains a finite index subgroup that is normal in $G$;\label{item:weaksep_1}
		\item  $H$ is commensurable to a normal subgroup of $G$;\label{item:weaksep_2}
		\item $H$ is commensurable to a subgroup that is normal in a finite index subgroup of $G$;\label{item:weaksep_2.5}
		\item  $H$ is weakly separable in $G$. \label{item:weaksep_3}
	\end{enumerate}
\end{lem}
\begin{proof}
	(\ref{item:weaksep_1}) $\implies$ (\ref{item:weaksep_2}) and (\ref{item:weaksep_2}) $\implies$ (\ref{item:weaksep_2.5}) are obvious.
	
	(\ref{item:weaksep_2.5}) $\implies$ (\ref{item:weaksep_2}): Suppose $H$ is commensurable to a subgroup $H'$ that is normal in some finite index subgroup of $G$. Then $H'$ has finitely many conjugates in $G$, each of which is commensurable to $H'$ as $H'$ is commensurated in $G$. Therefore $H''\coloneqq \cap_{g\in G} gH'g^{-1}$ is a finite index subgroup of $H'$ that is normal in $G$, hence  $H$ is commensurable to a normal subgroup.
	
	 (\ref{item:weaksep_2}) $\implies$ (\ref{item:weaksep_3}):	Suppose $K$ is a normal subgroup of $G$ commensurable to $H$. As $H$ is finitely generated, so is $K$. Moreover, $H\cap K$ is a finite index subgroup of $K$. Let $K_0$ be the intersection  of all subgroups of $K$ of index  $[K:H\cap K]$.  Since $K$ is finitely generated, $K_0$ is a finite index characteristic subgroup of $K$, hence is normal  in $G$. Moreover, $H$ contains  $K_0$ as a finite index subgroup, so that $H$ is virtually normal. Thus $H$ is weakly separable.
	 
	 (\ref{item:weaksep_3}) $\implies$ (\ref{item:weaksep_1}): If $H$ is weakly separable, then \cite[Corollary 7]{caprace_kropholler_reid_wesolek_2020} ensures $H$ contains a finite index subgroup that is normal in $G$. 
\end{proof}
Recall that a subgroup  $H\leq G$ is \emph{separable} if it is equal to the intersection of finite index subgroups of $G$ containing $H$. Moreover, $H$ is separable if and only if it is closed in the profinite topology of $G$, i.e. the topology on $G$ with basis the cosets of finite index normal subgroups.
The relation between weak separability and separability in the present context  is encapsulated in the following lemma:
\begin{lem}\label{lem:weaksepvssep}
	Let $G$ and $H\alnorm G$ be finitely generated groups.
	\begin{enumerate}
		\item If $H$ is separable, then it is weakly separable.\label{item:wsepvssep1}
		\item Suppose $H$ is weakly separable. Then $H$ is separable if and only if $G/H_0$ is residually finite, where $H_0$ is the  normal core of $H$.\label{item:wsepvssep2}
	\end{enumerate}
\end{lem}
Note that by Lemma \ref{lem:weaksep}, the normal core $H_0$ in (\ref{item:wsepvssep2}) is a finite index subgroup of $H$.
\begin{proof}
(\ref{item:wsepvssep1}): This is immediate from the definitions.

(\ref{item:wsepvssep2}):  It follows from the definitions that  $H_0$ is separable if and only if $G/H_0$ is residually finite. Thus it is enough to show $H$ is separable if only if $H_0$ is. First suppose $H$ is separable.  Since $H_0$ is the intersection of conjugates of $H$, it follows $H_0$ is separable. Conversely, suppose $H_0$ is separable. As $H_0$ is a finite index subgroup of $H$ and $H_0$ is closed in the profinite topology, $H$ is also closed in the profinite topology, hence is separable. \end{proof}

If $H$ is a group then a \emph{commensuration} of $H$ is an isomorphism $\phi:H_1\rightarrow H_2$ between finite index subgroups of $H$. Two commensurations are \emph{equivalent} if they agree on some finite index subgroup of $H$. Let $\Comm(H)$ be the set of equivalence classes of commensurations. We give $\Comm(H)$ the structure of a group by defining  for $[\phi],[\psi]\in \Comm(H)$, $[\phi][\psi]=[\phi'\circ \psi']$ for suitable representatives $\phi'\in[\phi]$ and $\psi'\in[\psi]$ such that $\phi'\circ \psi'$ is defined. 

There is a  homomorphism $\Delta:G\to \Comm(H)$ associated to a group $G$ containing a commensurated subgroup $H\alnorm G$, defined as follows.  For every $g\in G$, there exist finite index subgroups $H_1,H_2\leq H$ such that $gH_1g^{-1}=H_2$, and so conjugation by $g$ induces an element $\Delta(g)\in \Comm(H)$. The map $\Delta:G\rightarrow \Comm(H)$ is readily verified to be a well-defined homomorphism called the \emph{modular homomorphism}.

\subsection{Coarse bundles and quasi-isometries}\label{sec:cbundles}
In this subsection we recall the definition and basic properties of coarse bundles as used in  \cite{whyte2010coarse,margolisxer2021geometry}. 
\begin{defn}\label{defn:coarse bundle}
	Let $(X,d_X)$,  $(B,d_B)$ and $(F,d_F)$ be proper quasi-geodesic metric spaces. We say  \emph{$X$ is a coarse bundle over $B$ with fibre $F$} if there exist constants $K\geq 1$, $A\geq 0$, proper non-decreasing functions $\eta,\phi:\bbR_{\geq 0}\rightarrow \bbR_{\geq 0}$ and a map $p:X\rightarrow B$ such that:
	\begin{enumerate}
		\item \label{defn:coarse bundle 1} $p$ is $(K,A)$-coarse Lipschitz, i.e.\ for all $x,x'\in X$, \[d_B(p(x),p(x'))\leq Kd_X(x,x')+A;\]
		\item \label{defn:coarse bundle 2} for all $b\in B$, $D_b\coloneqq p^{-1}(b)$ is the \emph{fibre} of $X$ at $b$ and there is an $(\eta,\phi)$-coarse embedding $s_b:F\rightarrow X$ with $\im (s_b)= D_b$;
		\item \label{defn:coarse bundle 3} for all $b,b'\in B$,  $d_\Haus(D_b,D_{b'})\leq K d_B(b,b')+A$.
	\end{enumerate}
	When the above holds, we say that $p:X\rightarrow B$ is a \emph{$(K,A)$-coarse bundle}, or simply a \emph{coarse bundle}. We call $X$ the \emph{total space}, $B$  the \emph{base space} and $F$  the \emph{fibre}. 
\end{defn}

Our main source of coarse bundles is the following:

\begin{prop}[{\cite[Proposition 3.14]{margolisxer2021geometry}}]\label{prop:alnorm_cbundle}
	Let $G$ and $H\alnorm G$ be finitely generated groups equipped with the word metric. Let $G/H$ be the quotient space equipped with the relative word metric. Then the quotient map  $p:G\rightarrow G/H$ is a coarse bundle with fibre $H$.
\end{prop}

Another  (trivial) source of coarse bundles are spaces that split as direct products. We assume the direct product of metric spaces $X$ and $Y$ is equipped   with the $\ell_2$-metric $d((x,y),(x',y'))=\sqrt{d_X(x,x')^2+d_Y(y,y')^2}$ unless stated otherwise. The following is easy to verify:
\begin{lem}\label{lem:dirprod_cbundle}
	Suppose $B$ and $F$ are proper geodesic metric spaces, and $\pi_B:F\times B\rightarrow B$ is the projection onto the $B$ factor. Then $\pi_B:F\times B\rightarrow B$ is a coarse bundle with fibre $F$.
\end{lem} 

We are particularly interested in quasi-isometries of the total space that preserve fibres in the following sense.
\begin{defn}
	Let $p:X\rightarrow B$ and $p':X'\to B'$  be coarse bundles. We say that a quasi-isometry $f:X\rightarrow X'$ is \emph{$A$-fibre-preserving} if for every $b\in B$, there exists a $b'\in B'$ such that $d_\Haus(f(D_b),D'_{b'})\leq A$, where $D_b$ and $D'_{b'}$ are fibres of $p:X\rightarrow B$ and $p':X'\to B'$ respectively.  
	We say a quasi-action $\phi:G\to X^X$ is \emph{fibre-preserving} if there exists an $A$ such that every $\phi(g)$ is an $A$-fibre-preserving quasi-isometry.
\end{defn}

 \begin{prop}\label{prop:inducedqa}
	If $p:X\rightarrow B$ is a coarse bundle and  $G\qa X$ is a fibre-preserving quasi-action, then there is an \emph{induced quasi-action} $G\qa B$ satisfying  \[\sup_{b\in B}d_\Haus(D_{g\cdot b},g\cdot D_b)<\infty.\] Moreover, if $G\qa X$ is cobounded, then the induced quasi-action $G\qa B$ is also cobounded.
\end{prop}
\begin{proof}
	We pick constants $K$ and $A$ such that $p:X\rightarrow B$ is a $(K,A)$-coarse bundle, $G\qa X$ is a $(K,A)$-quasi-action and every $g\in G$ is an $A$-fibre-preserving quasi-isometry. Lemma 3.6 of \cite{margolisxer2021geometry} ensures there exist constants $K_1$ and $A_1$ such that for every $g\in G$, there exists a $(K_1,A_1)$-quasi-isometry $\lambda(g):B\rightarrow B$ such that $d_\Haus(D_{\lambda(g)(b)},g\cdot D_b)\leq A_1$ for all $g\in G$ and $b\in B$. 
		 Using Lemma \ref{lem:qaction_gp}, this implies  
		\begin{align*}
			d_\Haus(D_{\lambda({gh})(b)},D_{(\lambda(g)\circ\lambda(h))(b)})
			&\leq d_\Haus(gh\cdot D_{b},g\cdot D_{\lambda(h)(b)})+2A_1\\
			&\leq Kd_\Haus(h\cdot D_{b},D_{\lambda(h)(b)})+2A_1+3A\\
			&\leq KA_1+2A_1+3A\eqqcolon A_3
		\end{align*}
		for all $g,h\in G$ and $b\in B$. It follows from the definition of a $(K,A)$-coarse bundle that $\sup_{b\in B}d_B((\lambda(g)\circ\lambda(h))(b),\lambda(gh)(b))\leq KA_3+A$, and so $\lambda$ defines a quasi-action of $G$ on $B$.
	
	Now suppose $G\qa X$ is cobounded and choose $r$ such that for all $x,y\in X$, there exists $g\in G$ with $d_X( x,g\cdot y)\leq r$. Let $b,b'\in B$ and pick $x\in D_b$ and $y\in D_{b'}$. Choose $g\in G$ such that $d_X(x,g\cdot y)\leq r$, and note by the definition of $\lambda(g)$ there exists $z\in D_{\lambda(g)(b')}$ such that $d(z,g\cdot y)\leq A_1$. Since $d_X(x,z)\leq r+A_1$, $x\in D_b$ and $z\in D_{\lambda(g)(b')}$, we have $d_B(b,\lambda(g)(b'))\leq K(r+A_1)+A$. Thus the induced quasi-action $G\qa B$ is cobounded.
\end{proof}
\begin{defn}\label{defn:qitrivial}
	A coarse bundle $p:X\rightarrow B$ with fibre $F$ is \emph{quasi-isometrically trivial} if there is a quasi-isometry $X\rightarrow B\times F$ that is $A$-fibre preserving with respect to the projection $\pi:B\times F\rightarrow B$.
\end{defn}
We note that  if  $p:X\rightarrow B$ is a quasi-isometrically trivial coarse bundle with fibre $F$, then $X$ is quasi-isometric to $F\times B$.
A useful invariant to  determine whether a coarse bundle is quasi-isometrically trivial is the \emph{fibre-distortion function} of a coarse bundle, introduced by the author  and summarised below \cite{margolisxer2021geometry}. The fibre-distortion function is encoded in a more sophisticated coarse invariant called the \emph{holonomy} due to  Whyte \cite{whyte2010coarse}.

Let $p:X\rightarrow B$ be a coarse bundle with fibre $F$. For each $b\in B$, it follows from Definition \ref{defn:coarse bundle} that there exist coarse embeddings $s_b:F\rightarrow D_b$ and $t_b:D_b\rightarrow F$, with uniform distortion, such that $t_b\circ s_b$ and $s_b\circ t_b$ are uniformly close to the identity on $F$ and $D_b$ respectively. 
For each $b_0,b\in B$, since $D_{b_0}$ and $D_b$ are at finite Hausdorff distance, there is a closest point projection map $p_{b_0}^b:D_{b_0}\rightarrow D_b$.  We  thus define a map $\lambda_{b_{0}}^b:F\rightarrow F$ given by $t_{b}\circ p_{b_0}^b\circ s_{b_0}$.
  We now define the fibre-distortion function as follows:
\begin{defn}\label{defn:fibdis}
	If $f:X\rightarrow Y$ is a quasi-isometry, let \[\kappa(f)\coloneqq \inf \{K\geq 1 \mid \textrm{$f$ is a $(K,A)$-quasi-isometry for some $A\geq 0$} \}.\] 
		Suppose $p:X\rightarrow B$ is a coarse bundle with fibre $F$. Pick a basepoint $b_0\in B$.  For each $b\in B$, we define a function $\Lambda:B\rightarrow \mathbb{R}$ by $b \mapsto\log(\kappa( \lambda_{b_0}^b))$, which we call the \emph{fibre distortion function}.
\end{defn}
Before proving properties about the fibre-distortion function, we first give some examples.

\begin{exmp}\label{exmp:prod_trivial_fibre}
	Let $X=F\times B$ be a direct product of proper geodesic metric spaces. As in Lemma \ref{lem:dirprod_cbundle}, the projection $X\rightarrow B$ defines a coarse bundle structure. The fibre distortion function is trivial, since the closest point projection $p_{b}^{b'}:D_b\rightarrow D_{b'}$ is of the form $(f,b)\mapsto (f,b')$ and so $\kappa(p_{b}^{b'})=1$.
\end{exmp}

\begin{exmp}
	Let $G$ be a finitely generated group with $\mathbb{Z}\cong H=\lvert a\rvert \alnorm G$.  We can define a homomorphism $h:G\rightarrow \bbR$ such that for each $g\in G$, if there exist non-zero integers $m,n$ such that  $ga^ng^{-1}=a^m$, then  $h(g)=\log(\lvert\frac{m}{n}\rvert)$. This gives a well-defined homomorphism on $G$ called the \emph{height function}.  (See \cite{farbmosher1998bs1}, \cite{whyte2001baumslag}.) It follows from  \cite[Lemma 5.8]{margolisxer2021geometry} that the height function  descends to a well-defined map $G/H\rightarrow \bbR$ that is equal to the fibre distortion function of $G\rightarrow G/H$.
\end{exmp}

\begin{exmp}\label{exmp:fibre_abelian}
	Let $G$ be a finitely generated group with $\mathbb{Z}^n\cong H\alnorm G$. Let $B$ be a free abelian basis of $H$. Since $\Comm_G(H)\cong \text{GL}_n(\mathbb{Q})$, for each $g\in G$, there is a matrix $A_g\in \text{GL}_n(\mathbb{Q})$ such that conjugation by $g$ is equal to multiplication by $A_g$ on some finite index subgroup of $H$. Moreover, $A_g=A_{g'}$ if $g$ and $g'$ are in the same (left or right) $H$-coset.
	By Proposition \ref{prop:alnorm_cbundle}, the quotient map $q:G\rightarrow G/H$ is a coarse bundle. As shown in \cite[\S 5]{margolisxer2021geometry}, the fibre-distortion map is given by $gH\mapsto \lvert \log (\lVert A_g\rVert)\rvert $, where $\lVert \cdot \rVert$  is the $\ell_1$-operator norm.
	
	A subgroup of $\GL_n(\bbR)$ is bounded in the operator norm if and only if it is conjugate to a subgroup of $\Orth_n(\bbR)$. 
	Thus if $G$ is a group, $\bbZ^n\cong H\alnorm G$ is a commensurated subgroup and $\Delta:G\rightarrow \GL_n(\bbR)$ is the modular homomorphism, then $G\rightarrow G/H$ has bounded fibre distortion if and only if $\im(\Delta)$ is conjugate to a subgroup of $\Orth_n(\bbR)$.
\end{exmp}

We now justify  the use of the definite article in Definition \ref{defn:fibdis}.
\begin{prop}\label{prop:fib_dis_welldef}
	If $p:X\rightarrow B$ is a coarse bundle with fibre $F$, the fibre-distortion function $\Lambda$  is well-defined up to uniform additive error, independently of the choice of the basepoint  $b_0$ or the maps $\lambda_{b_{0}}^b$.
\end{prop}
To prove this, we make use of the following lemmas. The first follows easily from the definition of $\kappa$ above.
\begin{lem}\label{lem:close_kappa}
	If $f,g:X\rightarrow Y$ are quasi-isometries that are close,  then $\kappa(f)=\kappa(g)$.
\end{lem}

\begin{lem}[{\cite[Lemma 5.6]{margolisxer2021geometry}}]\label{lem:qi bounds composition}
	Suppose that $f:X\rightarrow Y$ is a quasi-isometry. If $g:W\rightarrow X$ and $h:Y\rightarrow Z$ are $(K,A)$-quasi-isometries, then \[\frac{\kappa (f)}{K^2}\leq \kappa(h\circ  f\circ g))\leq \kappa (f) K^2.\]
\end{lem}

\begin{proof}[Proof of Proposition \ref{prop:fib_dis_welldef}]
	It is easy to verify  $p_{b_0}^b$ and $t_b$  are  coarsely well-defined up to some additive error,  so Lemma \ref{lem:close_kappa} ensures  $\kappa( \lambda_{b_0}^b)$ and hence $\Lambda$,  depends only on $b_0$ and $b$ and not on the choice of $t_b$ and $p_{b_0}^b$.
	Now let $\Lambda_0$ and $\Lambda_1$ be fibre distortion functions with respect to basepoints $b_0$ and $b_1$. Then for all $b\in B$,  $\lambda_{b_1}^b=t_{b}\circ p_{b_1}^b\circ s_{b_1}$ is close to \[\lambda_{b_0}^b\circ \lambda_{b_1}^{b_0}=t_{b}\circ p_{b_0}^b \circ s_{b_0} \circ t_{b_0}\circ  p_{b_1}^{b_0}\circ s_{b_1}.\] 
	Therefore, Lemma \ref{lem:qi bounds composition} implies $\frac{\kappa (\lambda_{b_0}^b)}{K^2}\leq \kappa(\lambda_{b_1}^b)\leq \kappa (\lambda_{b_0}^b)K^2$, where $K\coloneqq \kappa(\lambda_{b_1}^{b_0})+1$. Thus for all $b\in B$, $\lvert \Lambda_0(b)-\Lambda_1(b)\rvert\leq 2\log(K)$.
\end{proof}

We now show fibre-preserving quasi-isometries preserve the fibre-distortion function, which is a slight generalisation of \cite[Proposition 5.5]{margolisxer2021geometry}.
\begin{prop}\label{prop:fib qi invariant}
	Suppose that $p:X\rightarrow B$ and $p':X'\rightarrow B'$ are coarse fibre bundles with fibres $F$ and $F'$, and with fibre distortion functions $\Lambda:B\rightarrow \bbR$ and $\Lambda':B\rightarrow \bbR$ respectively.  If  $f:X\rightarrow X'$ is a fibre-preserving quasi-isometry  and $\hat f:B\rightarrow B'$ is the induced quasi-isometry as in Lemma 3.6 of \cite{margolisxer2021geometry},  there is a constant $C\geq 0$ such that 
	\[\lvert \Lambda'(\hat{f}(b) )-\Lambda(b)\rvert\leq  C.\]
\end{prop}
\begin{proof}
	Suppose $f:X\rightarrow X'$ is $A$-fibre-preserving. The induced quasi-isometry $\hat f:B\rightarrow B'$ can therefore be chosen so that $d_\Haus(f(D_b),D_{\hat f(b)})\leq A$. For each $b\in B$, pick $ f_b:D_b\rightarrow D_{\hat f(b)}$ such that $d(f(x), f_b(x))\leq A$ for all $x\in D_b$. Define $\mu_{b}\coloneqq t_{\hat f(b)} f_b s_{b}:F\rightarrow F'$. Since $p_{b_0}^{b}$ and $p_{\hat f(b_0)}^{\hat f(b)}$ are closest point projections, it follows that 
	\[\sup_{x\in F}d((t_{\hat f (b)} f_b p_{b_0}^{b} s_{b_0})(x),( t_{\hat f (b)} p_{\hat f(b_0)}^{\hat f(b)} f_{b_0}  s_{b_0})(x))<\infty.\] 
	Lemma \ref{lem:close_kappa} ensures $\kappa(t_{\hat f (b)}f_bp_{b_0}^{b}s_{b_0})=\kappa(t_{\hat f (b)} p_{\hat f(b_0)}^{\hat f(b)} f_{b_0}  s_{b_0})$. Since $t_{\hat f (b)}f_bp_{b_0}^{b}s_{b_0}$ is close to $\mu_b\lambda_{b_0}^{b}$ and $t_{\hat f (b)} p_{\hat f(b_0)}^{\hat f(b)} f_{b_0}  s_{b_0}$ is close to $\lambda_{\hat f(b_0)}^{\hat f (b)} \mu_{b_0}$, we use Lemma \ref{lem:qi bounds composition} to deduce \[\frac{1}{A}\kappa(\lambda_{b_0}^{b})\leq \kappa(\lambda_{\hat f(b_0)}^{\hat f (b)})\leq A\kappa(\lambda_{b_0}^{b})\] for some constant $A$. The result now follows from  the definition of $\Lambda$ and $\Lambda'$.
	\end{proof}

We are particularly interested in the case where the fibre-distortion function is bounded:
\begin{cor}\label{cor:fibre-distortion_bounded}
	Suppose that $p:X\rightarrow B$ and $p':X'\rightarrow B'$ are coarse fibre bundles with fibres $F$ and $F'$, and with fibre distortion functions $\Lambda:B\rightarrow \bbR$ and $\Lambda':B\rightarrow \bbR$. If  $f:X\rightarrow X'$ is a fibre-preserving quasi-isometry, then $\Lambda$ is bounded if and only if $\Lambda'$ is.
\end{cor}

We mention a theorem of Kapovich--Kleiner--Leeb which says that if spaces split as direct products and the factors satisfy a form of coarse non-positive curvature, then quasi-isometries preserve the  direct product structure \cite{kapovich1998derham}. Kapovich--Kleiner--Leeb define two classes of spaces that they call  \emph{coarse type I} and \emph{coarse type II} (\cite[Definition 3.5]{kapovich1998derham}) which include:
\begin{enumerate}
	\item cocompact non-elementary hyperbolic spaces;
	\item cocompact geodesically complete  CAT(0) spaces containing a rank one geodesic;
	\item irreducible symmetric spaces of non-compact type;
	\item thick Euclidean buildings.
\end{enumerate}
We will not give the definition of spaces of coarse type I or coarse type II, and will only use the fact that the previous four  examples are of coarse type I or II; see \cite{kapovich1998derham} for more details.

\begin{thm}[{\cite[Theorem B]{kapovich1998derham}}]\label{thm:kkl}
	Let $X=Z\times \Pi_{i=1}^n X_i$, where each $X_i$ is either of coarse type I  or coarse type II and $Z$ is a geodesic metric space all of whose asymptotic cones are homeomorphic to $\bbR^n$. For each $K,A$, there exist constants $K',A',D$ such that the following hold.
	If $f:X\rightarrow X$ is a $(K,A)$-quasi-isometry, there is a homomorphism $\sigma:G\rightarrow \Sym(n)$ so that for every $i$, there exists a  $(K',A')$-quasi-isometry $f_i:X_i\rightarrow X_{\sigma(g)(i)}$ such that the following diagram commutes up to error at most $D$
	\[
	\begin{tikzcd}  
		X\arrow{r}{f}\arrow{d}{\pi_i}& X\arrow{d}{\pi_{\sigma(g) (i)}}\\
		X_i\arrow{r}{f_i}& X_{\sigma(g) (i)}
	\end{tikzcd}
	\]
	where $\pi_i:X\rightarrow X_i$ is the projection onto $X_i$.
\end{thm}

We first reformulate Theorem \ref{thm:kkl} in the language of quasi-actions and coarse bundles.

\begin{prop}\label{prop:coarsebundle_fibrepres}
	Let $X=Z\times \Pi_{i=1}^n X_i$ be as in Theorem \ref{thm:kkl}.  Suppose $I\subseteq \{1,\dots, n\}$, $X_I\coloneqq \Pi_{i\in I} X_i$, and $\pi_I:X\rightarrow X_I$ is the projection onto $X_I$.  Let   $G\qa X$ be a cobounded quasi-action, and let $\sigma:G\rightarrow \Sym(n)$ be as in Theorem \ref{thm:kkl}. Let $G_I=\{g\in G\mid \sigma(g)(I)=I\}$, which is a finite index subgroup of $G$. Then the restricted quasi-action $G_I\qa X$ is fibre-preserving with respect to the projection $\pi_I:X\rightarrow X_I$.
\end{prop}

\begin{proof}
	After reordering the factors, we assume $I=\{1,\dots, r\}$.
	We show $G_I\qa X$ is fibre-preserving with respect to the coarse bundle $\pi_I:X\rightarrow X_I$. Suppose $G_I\qa X$ is a $(K,A)$-quasi-action and fix some $g\in G_I$. By Theorem \ref{thm:kkl}, for each  $i\in \{1,\dots, n\}$, there is a quasi-isometry $g_i:X_i\rightarrow X_{\sigma(g)(i)}$, such that $d(\pi_{\sigma(g)(i)}(g\cdot x),g_i\cdot \pi_i(x))\leq D$ for all $x\in X$, where $D$ is a constant  depending only on $X$, $K$ and $A$.
	
	For each $1\leq i\leq r$, let $p_i:X_I\to X_i$ be the projection map.  Let $b=(b_1,\dots, b_r)\in X_I$ and $g\in G_I$. We define $b'\in X_I$ so that $p_{\sigma(g) (i)}(b')=g_i\cdot b_i$.  Let  $x\in \pi_I^{-1}(b)$, so that $\pi_i(x)=b_i$ for $1\leq i\leq r$. 
	Then \[d_{X_{\sigma(g)(i)}}(\pi_{\sigma(g)(i)}(g\cdot x),p_{\sigma(g) (i)}(b'))=d(\pi_{\sigma(g)(i)}(g\cdot x),g_i\cdot \pi_i(x))\leq D\] for $1\leq i\leq r$. This implies $d_{X_I}(b', \pi_I(g\cdot x))\leq r D$ and so $g\cdot D_b\subseteq N_{rD} (D_{b'})$, where $D_b=\pi_I^{-1}(b)$ and $D_{b'}=\pi_I^{-1}(b')$ are the fibres over $b$ and $b'$ respectively. A similar argument,  replacing $b$ with $b'$ and $g$ with $g^{-1}$ ensures $g\cdot D_b$ and $D_{b'}$ are uniformly finite Hausdorff distance apart. This demonstrates that every $g\in G_I$ is $A'$-fibre-preserving for some uniform $A'$.
\end{proof}

We will now use Theorem \ref{thm:kkl} to quasi-conjugate quasi-actions on  products to isometric actions on products. 
We say a quasi-action  $G\qa \Pi_{i=1}^n X_i$ \emph{coarsely preserves the product structure}  if up to uniformly bounded error, each $g\in G$ splits as a product $(g_1,\dots, g_n)$, possibly permuting factors. Note that Theorem \ref{thm:kkl} ensures if each $X_i$ is coarse type I  or coarse type II (in particular,  there is no ``$Z$'' factor) then any $G\qa \Pi_{i=1}^n X_i$ preserves the product structure.

\begin{prop}\label{prop:qaction_prod}
	Let $X=\Pi_{i=1}^n X_i$, where each $X_i$ is either  of coarse type I or of coarse type II. Suppose $G\qa X$ is a (cobounded) quasi-action and $G_i$ is the finite index subgroup of $G$ that stabilises the factor $X_i$, i.e.\ $G_i=\{g\in G\mid \sigma(g)(i)=i\}$ where $\sigma$ is as in Theorem \ref{thm:kkl}. Suppose every $G_i\qa X_i$ is quasi-conjugate to an isometric action $G_i\curvearrowright Y_i$. Then $G\qa \Pi_{i=1}^n X_i$ is quasi-conjugate to an isometric action $G\curvearrowright \Pi_{i=1}^n Z_i$, where each $Z_i$ is isometric to some $Y_{j}$ and the quasi-conjugacy $\Pi_{i=1}^n X_i\rightarrow \Pi_{i=1}^n Z_i$ preserves the product structure.
\end{prop}
 The reason for the somewhat technical phrasing of the previous proposition is to account for quasi-actions that permute factors.  
The proof of Proposition \ref{prop:qaction_prod} will follow from the following   theorem of Kleiner--Leeb; see also  an argument of  Ahlin \cite{ahlin02producttrees}.
\begin{thm}[\cite{kleinerleeb09quasiactions}]\label{thm:kl_qaction}
	Let $G$ be a group, $H\leq G$ be a finite index subgroup and $\beta:H\qa X$ be a quasi-action. Then there is a quasi-action $\alpha:G\qa \Pi_{i\in G/H} X_i$ preserving the product structure such that
	\begin{itemize}
		\item each $X_i$ is quasi-isometric to $X$;
		\item the induced action of $G$ on $G/H$ is by left multiplication;
		\item the restriction of $H$ to $X_H$ is quasi-conjugate to  $\beta$.
	\end{itemize} 
	Such a quasi-action is unique up to a quasi-conjugacy preserving the product structure. Moreover, if $\beta$ is an isometric action, then we may assume each $X_i$ is isometric to $X$ and $\alpha$ is an isomeric action.
\end{thm}

\begin{proof}[Proof of Proposition \ref{prop:qaction_prod}]
	By Theorem \ref{thm:kkl}, the quasi-action $G\qa X=\Pi_{i=1}^n X_i$ induces a homomorphism $\sigma:G\rightarrow \Sym(n)$ determined by the way in which $G$ permutes  the factors of $X$. Let $I\subseteq \{1,\dots, n\}$ be a set of representatives of orbits of $\sigma(G)$, and for each $i\in I$, let $G(i)$ be the $\sigma(G)$-orbit containing $i$. Let $G_i\leq G$ be the finite index subgroup $\{g\in G\mid \sigma(g)(i)=i\}$.  
	 	
	Let $i\in I$. By Theorem \ref{thm:kkl} and Proposition \ref{prop:coarsebundle_fibrepres}, $G\qa X$ induces quasi-actions $G\qa \Pi_{j\in G(i)} X_j$ and $G_i\qa X_i$. By the hypotheses,  $G_i\qa X_i$ is quasi-conjugate to an isometric action $G_i\curvearrowright Y_i$. We  now apply Theorem \ref{thm:kl_qaction} to quasi-conjugate $G\qa \Pi_{j\in G(i)} X_j$ to an isometric action $G\curvearrowright \Pi_{j\in G(i)}Z_j$, where each $Z_j$  is isometric to $Y_i$.  Since the original quasi-action $G\qa X$ preserves the product structure, we can quasi-conjugate it to the isometric action $G\curvearrowright \Pi_{i\in I}(\Pi_{j\in G(i)}Z_j)$. Finally, observe that  $\Pi_{i\in I}(\Pi_{j\in G(i)}Z_j)$ and $\Pi_{i=1}^n Z_i$ are isometric via an isometry that permutes factors, and so $G\qa X$ is quasi-conjugate to $G\curvearrowright \Pi_{i=1}^n Z_i$.
\end{proof}

We conclude with an application of Theorem \ref{thm:kkl} that will be used when  studying groups quasi-isometric to products of hyperbolic spaces. We recall a hyperbolic metric space is of \emph{coarse algebraic type} if it is quasi-isometric to a rank one symmetric space or a locally finite  regular tree of valence at least three. 
\begin{lem}\label{lem:prodhyp_qitoalgebraicgroup}
	Let $\Gamma$ be a finitely generated group quasi-isometric to $\Pi_{i=1}^nX_i$, where each $X_i$ is a cocompact proper  geodesic non-elementary hyperbolic metric space. If $\Gamma$ is virtually a uniform lattice in a semisimple algebraic group, then each of the  $X_i$ is of coarse algebraic type.
\end{lem} 
\begin{proof}
	If $\Gamma$ is a uniform lattice in a semisimple algebraic group, it acts geometrically on a space $\Pi_{i=1}^mY_i$, where each $Y_i$ is either an irreducible  symmetric space  or an irreducible  thick Euclidean building. Note that all the $Y_i$ and $X_i$ are either of course type I or II, or are Euclidean spaces. Since $\Pi_{i=1}^nX_i$ and $\Pi_{i=1}^mY_i$ are quasi-isometric, we can thus apply Theorem \ref{thm:kkl} to deduce   $\Pi_{i=1}^mY_i$ contains no Euclidean factors, $n=m$ and after reordering factors, each $Y_i$ is quasi-isometric to $X_i$. Thus $Y_i$ cannot contain an isometrically embedded Euclidean plane, hence has rank one. Since rank one Euclidean buildings are locally finite infinite-ended trees, each $X_i$ is either quasi-isometric to  such a tree, or is quasi-isometric to a  rank one symmetric space. 
\end{proof}

\subsection{Topological groups and metric spaces}\label{sec:topgps}
In this section we recall some definitions and conventions regarding topological groups. Classical references for the theory of topological groups are \cite{hewittross1979abstract,bourbaki98gentop1-4}. We refer to  \cite{cornulierdlH2016metric} for a thorough exposition of the metric geometry of locally compact groups.

A \emph{topological group} is a group which is simultaneously a topological space, such that the multiplication and inverse maps are continuous. We  assume topological groups are Hausdorff unless stated otherwise.
A topological group is \emph{locally compact} if it is  locally compact as a topological space, i.e.\ every point has a compact neighbourhood.
Under some mild hypotheses, locally compact groups can be thought of as geometric objects in their own right.  A locally compact group $G$ is \emph{compactly generated} if there exists a compact subset $K\subseteq G$ that generates $G$. Note that a discrete group is  finitely generated if and only if it is compactly generated.

If $G$ is a compactly generated locally compact  group, then the word metric with respect to any compact generating set is unique up to quasi-isometry. More generally, Cornulier--{de la Harpe}  define the notion of a geodesically adapted metric on a compactly generated locally compact group $G$.
A metric $d$ on $G$ is said to be \emph{geodesically adapted} if it is:
\begin{enumerate}
	\item left-invariant;
	\item proper, i.e.\  all balls $N_r(1)\coloneqq \{g\in G\mid d(1,g)\leq r\}$ have compact closure;
	\item locally bounded, i.e.\  $N_r(1)$ is a neighbourhood of $1$ for $r$ sufficiently large;
	\item quasi-geodesic.
\end{enumerate}
For any  two geodesically adapted metrics on $G$, the identity map is a  quasi-isometry (\cite[Corollary 4.B.11]{cornulierdlH2016metric}). We remark that a geodesically adapted metric is in general  not assumed to be \emph{compatible}, i.e.\ the topology on $G$ is not necessarily the topology induced by the metric. For instance, the word metric with respect to a compact generating set is always  discrete, so cannot induce the original topology of $G$ unless $G$ is also discrete.

A theorem of Kakutani--Kodaira   implies that one can quotient out a compactly generated group by a compact normal subgroup to obtain a second-countable metrisable group \cite{KakutaniKodaira1944Uber}; see also \cite[Theorems 2.A.10 and 2.B.6]{cornulierdlH2016metric}.
\begin{cor}\label{cor:kakutanikodaira}
	Let $G$ be a locally compact compactly generated group.  Then there exists a compact normal subgroup $K\vartriangleleft G$ such that $G/K$ is  metrisable and second-countable.
\end{cor}

If $G$ is a locally compact group and $X$ is a proper quasi-geodesic space, an isometric action  $\rho: G\to \Isom(X)$  is said to be \emph{geometric} if:
\begin{enumerate}
	\item $\rho$ is continuous;
	\item it is \emph{cocompact}: there is a compact subset $K\subseteq X$ such that $\rho(G)K=X$;
	\item it is \emph{proper}: for all compact $K\subseteq X$, the set $\{g\in G\mid \rho(g)K\cap K\neq \emptyset\}$ has compact closure.
\end{enumerate}  We note that a less restrictive notion of geometric action --- which applies to actions that are not necessarily continuous --- is considered in \cite{cornulierdlH2016metric}.

\begin{lem}[{\cite[Theorem 4.C.5]{cornulierdlH2016metric}}]\label{lem:geom_qi}
	Suppose $G$ is a locally compact group, $X$ is a proper quasi-geodesic metric space and $\rho:G\to \Isom(X)$ is geometric. Then $G$ is compactly generated and the orbit map $g\mapsto \rho(g)x_0$ is a quasi-conjugacy from $G\curvearrowright G$ to $G\curvearrowright X$.
\end{lem}

Every locally compact group $G$ can be equipped with a left-invariant Radon measure $\mu$, unique up to rescaling, called the (left) \emph{Haar measure}. A \emph{lattice} is a discrete subgroup $\Gamma\leq G$ such that $\Gamma V =G$ for some Borel set $V\subseteq G$ with $\mu(V)<\infty$. A lattice $\Gamma\leq G$ is said to be  \emph{uniform} if $\Gamma V =G$  for some compact $V\subseteq G$.   If $\Gamma$ is a  discrete group, $X$ is a proper metric space and  $\rho:\Gamma\to \Isom(X)$ is a  geometric action, then $\im(\rho)$ is a uniform lattice and $\ker(\rho)$ is finite. 

\begin{prop}[{\cite[Remark 4.C.13]{cornulierdlH2016metric}}]\label{prop:copci}
	Let $G$ and $H$ be compactly generated locally compact groups and let $\phi:G\to H$ be a continuous proper homomorphism with cocompact image. Then $\phi$ is a quasi-isometry. In particular:
	\begin{enumerate}
		\item if $G$ is a uniform lattice  in $H$, then the inclusion $G\to H$ is a quasi-isometry;
		\item if $K\vartriangleleft G$ is a compact normal subgroup, the quotient map $G\to G/K$ is a quasi-isometry.
	\end{enumerate} 
\end{prop}
The following lemma shows a continuous extension of a geometric action of a uniform lattice is geometric.
\begin{lem}\label{lem:lattice_extension_geom}
	Let $X$ be quasi-geodesic metric space, let $\Gamma$ be a finitely generated group,  and let $\rho:\Gamma\to \Isom(X)$ be a geometric action. If $\Gamma\leq G$  is a uniform lattice and $\rho$ extends to a continuous map $\hat\rho:G\to \Isom(X)$, then $\hat \rho$ is also  geometric.
\end{lem}
\begin{proof}
	Since $\Gamma\leq G$ is a uniform lattice, $G=\Gamma L$ for some compact set $L$. Let $K\subseteq X$ be compact.  Since $\hat \rho$ is continuous, the  natural map $G\times X\to X$ given by $(g,x)\mapsto \hat\rho(g)(x)$ is continuous. Therefore $Q\coloneqq \hat \rho(L)(K)$ is also compact. Since $\rho$ is proper, the set $F\coloneqq \{g\in \Gamma\mid \rho(g)(Q)\cap Q\neq \emptyset\}$ is finite. Thus $\{g\in G\mid \hat\rho(g)(K)\cap K\neq \emptyset \}$ is contained in $FL$, hence has compact closure. It follows that $\hat \rho$ is proper. Since $\rho(\Gamma)$ is cocompact, so is $\hat\rho(G)$, hence $\hat \rho$ is geometric. 
\end{proof}

In Section \ref{sec:hyp_prod} we will make use of the following elementary lemma; see e.g. \cite[\S 2.C]{capracemonod2012lattice}.
\begin{lem}\label{lem:lattice_int_open}
	If $G$ is a locally compact group,  $\Gamma\leq G$ is a (uniform) lattice in $G$ and $U\leq G$ is an open subgroup, then $\Gamma\cap U$ is a (uniform) lattice in $U$.
\end{lem}

We also require the following lemma concerning centralisers in topological groups:
\begin{lem}\label{lem:open_cen}
	Let $G$ be a topological group let $\Gamma\leq G$ be a finitely generated discrete subgroup. If $\Gamma$ is   normalised by a dense subgroup of $G$, then its centraliser $C_G(\Gamma)$  is open.
\end{lem}
\begin{proof}
	Let $\Lambda\leq G$ be a dense subgroup normalising $\Gamma$. For each $h\in \Gamma$, we define  $\phi_h:G\rightarrow G$ by $g\mapsto ghg^{-1}$. As $\phi_h$ is continuous, $\phi_h(G)=\phi_h(\overline{\Lambda})\subseteq \overline{\phi_h(\Lambda)}\subseteq \overline{\Gamma}=\Gamma.$ As $\Gamma$ is discrete, the centraliser of any $h\in \Gamma$, which is equal to $\phi_h^{-1}(h)$, is open. Thus the centraliser of $\Gamma$, which is equal to the intersection of  $\phi_s^{-1}(s)$ for all $s$ in some finite generating set of $\Gamma$, is also open.
\end{proof}

\section{Discretisable quasi-actions and coarse stabilisers}\label{sec:disc_quasi_actions}
In this section we develop the basic theory of coarse stabilisers. We characterise discretisability in terms of coarse stabilisers, and prove a generalisation of the Milnor--Schwarz lemma for  quasi-actions admitting a coarse stabiliser.

\begin{defn}
	
Given a quasi-action $G\qa X$, a subset $T\subseteq G$ is \emph{bounded} if $\{t\cdot x_0\mid t\in T\}\subseteq X$ is bounded for some (equivalently any) $x_0\in X$.

\end{defn}
\begin{rem}
	Let $G\qa X$ be a quasi-action and let $S\subseteq G$ be a generating set that satisfies the conclusions of Proposition \ref{prop:generalms}. Then $T\subseteq G$ is bounded in  the preceding sense  if and only if it is bounded as a subspace of  the metric space $(G,d_S)$.
\end{rem}

 We recall the definition of a  coarse stabiliser:
\begin{defn}
	A subgroup  $H\leq G$ is a \emph{coarse stabiliser} of $G\qa X$ if:
\begin{enumerate}[(CS1)]
	\item\label{item:CS1} $H$ is bounded;
	\item\label{item:CS2} every bounded subset of $G$ is contained in finitely many left $H$-cosets.
\end{enumerate}	
\end{defn}
The motivating example of a coarse stabiliser is the following:
\begin{exmp}\label{exmp:stabvertex,cstab}
Suppose $Y$ is a connected locally finite graph and $G\curvearrowright Y$ is an action by graph automorphisms. Let $v$ be a vertex of $Y$ and let $H\coloneqq \stab_G(v)$. We claim that $H$ is a  coarse stabiliser of $G\curvearrowright Y$. Firstly, as $H\cdot v=v$, $H$ is bounded, so \ref{item:CS1} holds. We now show \ref{item:CS2} holds. 
If $T\subseteq G$ is bounded, then $T\cdot v$ is bounded in $Y$, hence consists of finitely many vertices. We note if $g,k\in G$, then $gv=kv$  if and only if $k\in gH$. Thus $T$ consists of finitely many left $H$-cosets.
\end{exmp}
 We now demonstrate the robustness of  coarse stabilisers:  they are well-defined  up to commensurability.
\begin{prop}\label{prop:coarsestabsarecomm}
Let $H$ be a coarse stabiliser of the quasi-action $G\qa X$. Then $H'\leq G$ is a  coarse stabiliser of $G\qa X$ if and only if $H$ and $H'$ are commensurable.
\end{prop}
\begin{proof}
First suppose $H'$ is a coarse stabiliser. Let $x\in X$. By \ref{item:CS1},  we can choose $R$ sufficiently large such that $H,H'\subseteq \{g\in G\mid d(g\cdot x,x)\leq R \}\eqqcolon B$.  As $B$ is bounded, \ref{item:CS2} implies $B$ is contained in both finitely many left $H$-cosets and in finitely many left $H'$-cosets. Since $H,H'\subseteq B$,  $H$ is contained in finitely many left $H'$-cosets, and  $H'$ is contained in finitely many left $H$-cosets; thus $H$ and $H'$ are commensurable.

For the converse, suppose $H$ and $H'$ are commensurable. It follows from \ref{item:CS2} that every bounded subset of $G$ is contained in finitely many left $H$-cosets, and hence is contained in finitely many left $H'$-cosets. Thus we need only show $H'$ is bounded. Since $H'\subseteq F H$ for some finite $F=\{x_1,\dots, x_n\}\subseteq G$, we see that for any $x\in X$,  $H'\cdot x\subseteq x_1H\cdot x\cup \dots \cup x_nH\cdot x$. As $H\cdot x$ is bounded, so is $H'\cdot x$. Thus $H'$ is a coarse stabiliser of $G\qa X$.
\end{proof}

\begin{cor}
	If $H\leq G$ is a coarse stabiliser of  $G\qa X$, then it is commensurated.
\end{cor}
\begin{proof}
	 Fix $x\in X$. For all $V\subseteq G$ and $g\in G$, note that $gVg^{-1}\cdot x$ has finite Hausdorff distance from $g\cdot V\cdot g^{-1} \cdot x$. Thus $V$ is bounded if and only if $gVg^{-1}$ is  bounded.
	 	 Suppose $G\qa X$ is a $(K,A)$-quasi-action with coarse stabiliser $H$. By Proposition \ref{prop:coarsestabsarecomm}, to show $H$ is commensurated   it is sufficient to show that for every $g\in G$,   $gHg^{-1}$ is a coarse stabiliser of $G\curvearrowright X$. 	Fix $g\in G$. Since $H$ is bounded, so is  $gHg^{-1}$.  If  $T\subseteq G$ is bounded, then $g^{-1}Tg$ is also bounded. Since $H$ is a coarse stabiliser,  $g^{-1}Tg\subseteq FH$ for some finite $F\subseteq G$. Thus $T\subseteq (gFg^{-1})gHg^{-1}$ and so every bounded subset is contained in finitely many left $gHg^{-1}$-cosets. Therefore $gHg^{-1}$ is a coarse stabiliser.
\end{proof}

We now prove a relative version of the  Milnor--Schwarz lemma for quasi-actions with coarse stabilisers. 
\begin{prop}\label{prop:cstab homspace}
	Suppose $X$ is a quasi-geodesic metric space and $G\qa X$ is a cobounded quasi-action with coarse stabiliser $H$.
	Then:
	\begin{enumerate}
		\item $G$ is finitely generated relative to $H$.
		\item $H$ is a commensurated subgroup of $G$.
		\item The natural left action $G\curvearrowright G/H$  is quasi-conjugate to $G\qa X$.
	\end{enumerate} 
\end{prop}
\begin{proof}
	Let $x\in X$. Since $H$ is a coarse stabiliser, there exists an $R$ such that $H\cdot x \subseteq N_R(x)$. 
	By increasing $R$ if necessary, Proposition \ref{prop:generalms} ensures we may assume  $G$ is generated by $S\coloneqq \{g\in G\mid d(g\cdot x, x)\leq R\}$ and $G\curvearrowright (G,d_S)$ is quasi-conjugate to $G\qa X$.  Note that $H\subseteq S$. Since $S$ is bounded and  $H$ is a coarse stabiliser, there is a finite subset $F\subseteq G$ such that $S\subseteq FH$. Without loss of generality, we may assume that $F\subseteq S$. We thus deduce that $F$ is a finite generating set of the pair $(G,H)$. Let $d_S$ be the word metric on $G$ with respect to $S$ and let $d_{G,H,F}=d_{F}$ be the relative word metric of the pair $(G,H)$ with respect to $F$. We show that the quotient  map $(G,d_S)\rightarrow (G/H,d_F)$ given by  $g\mapsto gH$ is a quasi-conjugacy. The map is surjective and equivariant, so we need only show it is a quasi-isometric embedding.
	
	Let $g,k\in G$ and suppose $d_S(g,k)=n$. Then $g^{-1}k=s_1\dots s_n$ for some $s_1, \dots, s_n\in S$. 
	Since $S\subseteq FH$, we have $s_i=f_ih_i$, for some $f_i\in F$ and $h_i\in H$ and so $d_{F}(gH,kH)\leq n$. 
	Now suppose $d_{F}(gH,kH)=m>0$. 
	Then there exist $f'_1, \dots, f'_m\in F$ and $h'_1, \dots, h'_m\in H$ such that $g^{-1}k=f'_1h'_1\dots f'_mh'_m$. Thus $d_S(g,k)\leq 2m$. If $gH=kH$, then $g^{-1}k\in H$, and so $d_S(g,k)\leq 1$. Thus $d_{S}(g,k)\leq 2m+1=2d_{F}(gH,kH)+1$ as required.
\end{proof}

We can use Proposition \ref{prop:cstab homspace} to deduce a necessary and sufficient condition for two quasi-actions, at least one of which has a coarse stabiliser,   to be quasi-conjugate. 
\begin{prop}\label{prop:cstabinvariant}
	Suppose  $X$ and $Y$ are quasi-geodesic metric spaces and that $G\qa X$ and $G\qa Y$ are cobounded quasi-actions. Suppose $H\leq G$ is a coarse stabiliser of $G\qa X$. Then $G\qa X$ is quasi-conjugate to $G\qa Y$ if and only if $H$ is also a coarse stabiliser of $G\qa Y$. 
\end{prop}
\begin{proof}
	Suppose  $f:X\rightarrow Y$ is a quasi-conjugacy and $H$ is a coarse stabiliser of $G\qa X$. Let $x\in X$ and $T\subseteq G$. Then it is straightforward to see that $d_\Haus(T\cdot f(x),f(T\cdot x))<\infty$. Thus $T\cdot x\subseteq X$ is bounded if and only $T\cdot f(x)\subseteq Y$ is. It readily follows that  $H$ is a coarse stabiliser of $G\qa Y$. 	Conversely, suppose $H$ is a coarse stabiliser of both $G\qa X$ and $G\qa Y$. Then Proposition \ref{prop:cstab homspace} implies that $G\qa X$ and $G\qa Y$ are both quasi-conjugate to $G\curvearrowright G/H$, hence are quasi-conjugate to one another.
\end{proof}
We recall from the introduction that  a cobounded quasi-action is discretisable if it can be quasi-conjugated to an  action on a locally finite graph. We  characterise discretisable  quasi-actions  as precisely those that have coarse stabilisers:
\begin{cor}\label{cor:discvscstab}
	Let $X$ be a quasi-geodesic metric space. A cobounded quasi-action $G\qa X$ is discretisable if and only if it has a coarse stabiliser. 
\end{cor}
\begin{proof}
First suppose $G\qa X$ is discretisable. By definition, $G\qa X$ is  quasi-conjugate to an action $G\curvearrowright Y$ on a locally finite graph $Y$. Let $v$ be a vertex of $Y$ and let $H\coloneqq \stab_G(v)$. As in Example \ref{exmp:stabvertex,cstab},  $H$ is a coarse stabiliser of $G\curvearrowright Y$, hence by Proposition \ref{prop:cstabinvariant}, $H$ is a coarse stabiliser of $G\qa X$.
Conversely, suppose $H\leq G$ is a coarse stabiliser of $G\qa X$. By Proposition \ref{prop:cstab homspace}, there exists a finite relative generating set $S$ of $(G,H)$. Proposition \ref{prop:relcg_locfinite} ensures the relative Cayley graph $\Gamma_{G,H,S}$ is locally finite. Since $H$ is the stabiliser of a vertex of $\Gamma_{G,H,S}$, Example \ref{exmp:stabvertex,cstab} ensures $H$ is also a  coarse stabiliser of $G\curvearrowright \Gamma_{G,H,S}$. Thus  Proposition \ref{prop:cstabinvariant} ensures $G\qa X$   is quasi-conjugate to $G\curvearrowright \Gamma_{G,H,S}$, and so   $G\qa X$ is discretisable.
\end{proof}

We conclude this section with a useful application of Proposition \ref{prop:cstab homspace} to  fibre-preserving cobounded quasi-actions on coarse bundles. We recall that if $G$ is finitely generated and $H\alnorm G$ is a finitely generated commensurated subgroup, then the quotient map  $q:G\to G/H$ is a coarse bundle with fibre $H$.
\begin{prop}\label{prop:quasi-conj_coarse bundle}
	Let $p:X\rightarrow B$ be a coarse bundle with fibre $F$, and let $G\qa X$ be a fibre-preserving proper cobounded quasi-action. Suppose the induced quasi-action $G\qa B$ (as in Proposition \ref{prop:inducedqa}) has coarse stabiliser $H$. 
	
	 Then $G$ and $H$ are finitely generated, $H\alnorm G$, and for any $x_0\in X$,  the orbit map $g\mapsto g\cdot x_0$ is a fibre-preserving quasi-conjugacy from the quotient map $q:G\to G/H$ to $p:X\to B$. In particular,  $H$ is  quasi-isometric to $F$ and $G/H$ is quasi-isometric to $B$.
\end{prop}
\begin{proof}
	Let $x_0\in X$ and set $b_0\coloneqq p(x_0)$. Since $G\qa X$ is a cobounded proper quasi-action, it follows from Proposition \ref{prop:generalms} that $G$ is finitely generated and the orbit map $f:G\rightarrow X$ given by $g\mapsto g\cdot x_0$ is a quasi-conjugacy. 
	 If $H\leq G$ is a coarse stabiliser of the induced quasi-action $G\qa B$, then Proposition \ref{prop:cstab homspace} ensures $G\curvearrowright G/H$ is quasi-conjugate to  $G\qa B$ via a quasi-conjugacy $\hat f:G/H\rightarrow B$. In particular, $G/H$ and $B$ are quasi-isometric. Since $G$ admits a cobounded  quasi-action  on all the spaces in the following diagram  and all the maps are coarsely $G$-equivariant, the following diagram commutes up to uniformly bounded error.
	 \[
	 \begin{tikzcd}  
	 	G\arrow{r}{f}\arrow{d}{q}& X\arrow{d}{p}\\
	 	G/H\arrow{r}{\hat f}& B
	 \end{tikzcd}
	 \] 
	 The above diagram implies $d_\Haus(f(H),D_{b_0})<\infty$, thus ensuring that $H$ is coarsely connected as a subspace of $G$. Therefore $H$ is finitely generated by Lemma \ref{lem:fg<->cc} and  quasi-isometric to $F$ by Lemma \ref{lem:qi_subspace}. Moreover,  coarse commutativity of the preceding diagram ensures $f$ is a fibre-preserving quasi-isometry.
\end{proof}

\section{Topological completions of quasi-actions}\label{sec:topcomp}

\subsection{Geometry of topological completions}\label{sec:geom_topcomp}
We now  define the notion of a topological completion of a quasi-action and establish its basic properties. 
In Proposition \ref{prop:topcompquasiconj} we show that the topological completion $\hat G$ of $G\qa X$ is  compactly generated and the left action of $\hat G$ on itself is quasi-conjugate to $G\qa X$. In Proposition \ref{prop:TCvsQC}, we show the topological completion of a cobounded quasi-action exists if and only if we can quasi-conjugate it to an isometric action on a proper metric space.
The results of this subsection hold not just for topological completions, but for the more general notion of  a weak topological completion.

\begin{defn}
	The \emph{weak topological completion of a quasi-action} $G\qa X$ is a  homomorphism $\rho:G\rightarrow \hat G$ such that:
	\begin{enumerate}
		\item \label{item:WT1} $\hat G$ is a  locally compact topological group; 
		\item \label{item:WT2} $\rho(G)$ is cocompact;
		\item \label{item:WT3} for every $T\subseteq G$,  $T$ is bounded if and only if $\overline{\rho(T)}$ is compact.
	\end{enumerate}
	A \emph{topological completion} of $G\qa X$ is a weak topological completion $\rho:G\rightarrow \hat G$ in which a stronger form of (\ref{item:WT2}) holds:
	\begin{enumerate}[($2'$)]
		\item $\rho(G)$ is dense in $\hat G$.
	\end{enumerate}
\end{defn}
For ease of exposition, we sometimes say that $\hat G$ is a (weak) topological completion of $G\qa X$, where the  homomorphism $\rho:G\rightarrow \hat G$ is implicit. 
For instance, we say $G\qa X$ has a totally disconnected topological completion if there  exists a topological completion $\rho:G\rightarrow \hat G$ such that $\hat G$ is totally disconnected.

\begin{rem}\label{rem:topcomp,normal}
	Suppose  $G\qa X$ has a (weak) topological completion $\rho:G\rightarrow \hat G$ with $K\leq \hat G$ a compact normal subgroup. If $\pi:\hat G\rightarrow \hat G/K$ is the quotient map, then $\pi\circ\rho:G\rightarrow \hat G/K$ is a (weak) topological completion of $G\qa X$.
\end{rem}

Given a finitely generated group $\Gamma$ equipped with a (left-invariant) word metric $d$, $\Gamma$ acts on itself by left-multiplication; we call this action $\phi:\Gamma\to\Isom(\Gamma,d)$ the \emph{regular representation} of $\Gamma$. More generally, if $G$ is a locally compact compactly generated group, there is a regular representation $G\to \Isom(G,d)$ when $G$ is equipped with a geodesically adapted metric. Although the regular representation depends on the choice of generating set, it is unique up to quasi-conjugacy.
The following proposition tells us that given a  weak topological completion $\rho: G\rightarrow \hat G$ of $G\qa X$, $\hat G$ is quasi-isometric to $X$ and the quasi-action of $G$ on $X$ is $\rho$-quasi-conjugate to the regular representation of  $\hat G$. 
\begin{prop}\label{prop:topcompquasiconj}
	Suppose $X$ is a quasi-geodesic metric space and $\phi:G\rightarrow  X^X$ is a cobounded quasi-action. Let $\rho:G\rightarrow \hat G$ be a weak topological completion.
	Then $\hat G$ is compactly generated. Moreover, $\phi$  is $\rho$-quasi-conjugate to the regular representation of $\hat G$.
\end{prop}
\begin{proof}
	Since $\rho(G)$ is cocompact, there exists a compact subset $K\subseteq \hat G$ such that $\rho(G)K=\hat G$.  Applying Proposition \ref{prop:generalms} to the quasi-action $\phi:G\rightarrow X^X$, there exists a bounded generating set $S\subseteq G$, and so  $\overline{\rho(S)}$ is compact. Since $\rho(S)$ generates  $\rho(G)$ and $\rho(G)K=\hat G$, we deduce  $\overline{\rho(S)}\cup K$ is a compact generating set of $\hat G$.
	
	Since $\hat G$ is compactly generated, there exists a  geodesically adapted metric $d_{\hat G}$  on $\hat G$. 
	As $d_{\hat G}$ is left invariant, we can define an isometric action $\psi:G \rightarrow \Isom(\hat G,d_{\hat G})$ by $\psi(g)(k)\coloneqq \rho(g)k$. Using the definition of a topological completion and a geodesically adapted metric,  we see that for arbitrary $T\subseteq G$, $x_0\in X$ and $g_0\in \hat G$, the following are equivalent:
	\begin{enumerate}
		\item $\{\phi(t)(x_0)\mid t\in T\}$ is bounded in $X$;
		\item $\overline{\rho(T)}$ is compact;
		\item $\{\psi(t)(g_0)=\rho(t)g_0\mid t\in T\}$ is bounded in $(\hat G,d_{\hat G})$.
	\end{enumerate}
	Applying  Proposition \ref{prop:generalms} to both $\phi:G\rightarrow X^X$ and $\psi:G\rightarrow \Isom(\hat G,d_{\hat G})$, we deduce that there is a sufficiently large bounded generating set $S'\subseteq G$ such that $\phi$ and $\psi$ are both quasi-conjugate to $G\curvearrowright (G,d_{S'})$, hence are quasi-conjugate to one another. It follows from the definition of $\psi:G\curvearrowright (\hat G, d_{\hat G})$ that it is $\rho$-quasi-conjugate to the regular representation $\hat G\curvearrowright (\hat G, d_{\hat G})$. Thus $\phi$ is $\rho$-quasi-conjugate to the regular representation of $\hat G$. 
\end{proof}

We will also require the following consequence of the Milnor--Schwarz lemma for topological groups \cite[Theorem 4.C.5.]{cornulierdlH2016metric}. 

\begin{cor}\label{cor:qconj_geomaction}
	Let $X$ be a quasi-geodesic metric space and let $G\qa X$ be a cobounded quasi-action with weak topological completion  $\rho:G\rightarrow\hat G$.  Suppose $Y$ is a proper quasi-geodesic metric space and  $\phi:\hat G\rightarrow \Isom(Y)$  is a geometric action. Then $G\qa X$ is $\rho$-quasi-conjugate to $\hat G\curvearrowright Y$.  
\end{cor}
\begin{proof}
The proof of Theorem 4.C.5 of \cite{cornulierdlH2016metric} implies that for some fixed $y_0\in Y$,  $d_{\hat G}(g,h)\coloneqq d_Y(\phi(g)(y_0),\phi(h)(y_0))$ defines a geodesically adapted pseudo-metric $d_{\hat G}$ on $\hat G$. This ensures the regular representation $\hat G\curvearrowright (\hat G,d_{\hat G})$ is quasi-conjugate to $\hat G\curvearrowright Y$. Proposition \ref{prop:topcompquasiconj} thus implies that $G\qa X$ is $\rho$-quasi-conjugate to $\hat G\curvearrowright Y$ as required.
\end{proof}
We note that if $X$ is a proper metric space, $\Isom(X)$ can be equipped with the topology of pointwise convergence, making it a second-countable locally compact group acting continuously and properly on $X$ (see \cite[Lemma 5.B.4]{cornulierdlH2016metric}). Moreover, if $\Isom(X)$ acts cocompactly on $X$, this action is geometric. We  give the following geometric characterisation for a quasi-action to have a  topological completion:
\begin{prop}\label{prop:TCvsQC}
Let $X$ be a quasi-geodesic metric space and let $G\qa X$ be a cobounded quasi-action. The following are equivalent:
\begin{enumerate}
	\item \label{item:TCvsQC_metric} there exists a proper metric space $Y$ such that $G\qa X$ is quasi-conjugate to an isometric action $G\curvearrowright Y$;
	\item \label{item:TCvsQC_wTC} $G\qa X$ has a weak topological completion;
	\item \label{item:TCvsQC_TC} $G\qa X$ has a topological completion.
\end{enumerate}
Moreover, if $Y$ is as in (\ref{item:TCvsQC_metric}), then the action $\rho:G\rightarrow \Isom(Y)$ is a weak topological completion, and the corestriction  $G\rightarrow \overline{\im(\rho)}$ of $\rho$ is a topological completion.
\end{prop}
\begin{proof}
	$(\ref{item:TCvsQC_metric})\implies (\ref{item:TCvsQC_wTC})$: Suppose $G\qa X$ is quasi-conjugate to an isometric action $\rho:G\rightarrow \Isom(Y)\eqqcolon \hat G$ for some proper metric space $Y$. We claim that $\rho$ is a weak topological completion of $G\qa X$.   Fix some  $y_0\in Y$. Since the action $\hat G\curvearrowright Y$ is geometric,   the set 
	 $S_R\coloneqq \{\phi\in \hat G\mid d_Y(\phi(y_0),y_0)\leq R\}$ is compact for every $R\geq 0$.
	
	We pick $K$ and $A$ large enough such that $G\qa X$ is a $(K,A)$-quasi-action and there is a $(K,A)$-quasi-conjugacy $f:X\rightarrow Y$. 
	Pick $x_0\in X$ such that $d_Y(f(x_0),y_0)\leq A$.
		Since $G\qa X$ is cobounded, there is a constant $B$ such that for all $x\in X$, there exists a $g_x\in G$ with $d_X(g_x\cdot x,x_0)\leq B$. 
		Let $\psi\in \hat G$ and choose $x\in X$ such that $d_Y(f(x),\psi(y_0))\leq A$. 
		Since 
		\begin{align*}
			d_Y\Big(y_0,(\rho(g_{x})\circ\psi)(y_0)\Big)&\leq
			d_Y\Big(f(x_0),\rho(g_x)(f(x))\Big)+2A\\
			&\leq d_Y\Big(f(x_0),f(g_x\cdot x)\Big)+3A\\
			&\leq Kd_X(x_0,g_x\cdot x)+4A\leq KB+4A.
		\end{align*}  
	Thus $\rho(g_x)\psi\in S_{KB+4A}$. Since $\psi\in \hat G$ was arbitrary, we deduce that $\rho(G)S_{KB+4A}=\hat G$, verifying that $\rho(G)$ is a cocompact subgroup of $\hat G$.
	
	We now show $T\subseteq G$ is bounded if and only if $\overline{\rho(T)}$ is compact. Indeed, if $T$ is bounded there is an $r\geq 0$ such that $d(x_0,t\cdot x_0)\leq r$ for all $t\in T$. It follows that 
	\begin{align*}
		d_Y(y_0,\rho(t)(y_0))\leq d_Y(f(x_0),f(t\cdot x_0))+2A\leq Kr+3A.
	\end{align*}
Thus $\rho(T)\subseteq S_{Kr+3A}$ and hence $\overline{\rho(T)}$ is compact. Conversely, suppose $\overline{\rho(T)}$ is compact. Since the action $\hat G\times Y\rightarrow Y$ is continuous,  $\{\phi(y_0)\mid \phi\in \overline{\rho(T)}\}\subseteq Y$ is compact hence is contained in $B_{r'}(y_0)$ for $r'$ sufficiently large. Thus for all $t\in T$, we have \begin{align*}
	d_X(x_0,t\cdot x_0)&\leq Kd_Y(f(x_0),f(t\cdot x_0))+KA\\
	 &\leq Kd_Y\Big(f(x_0),\rho(t)(f(x_0))\Big)+2KA\\
	 &\leq Kd_Y(y_0,\rho(t)(y_0))+4KA\leq Kr'+4KA,
	\end{align*}
and so $T$ is bounded as required. Thus $\rho:G\rightarrow \hat G$ is a weak topological completion of $G\qa X$.

$(\ref{item:TCvsQC_wTC})\implies (\ref{item:TCvsQC_TC})$: If $\rho:G\rightarrow \hat G$ is a weak topological completion of $G\qa X$, then the corestriction of $\rho$ to $\overline{\im(\rho)}$  defines a topological completion $G\rightarrow \overline{\im(\rho)}$.

$(\ref{item:TCvsQC_TC})\implies (\ref{item:TCvsQC_metric})$:
Suppose $\rho: G\rightarrow \hat G$ is a topological completion of $G\qa X$. 
 By Corollary \ref{cor:kakutanikodaira} and Remark \ref{rem:topcomp,normal}, we can assume without loss of generality that $\hat G$ is second-countable. By Proposition 4.B.9 of \cite{cornulierdlH2016metric},  $\hat G$ admits a geodesically adapted left-invariant proper metric $d$ compatible with the topology of $G$. Thus Proposition \ref{prop:topcompquasiconj} implies $G\qa X$ is $\rho$-quasi-conjugate to the regular representation $\hat G\curvearrowright (\hat G,d)$ on the proper metric space $(\hat G,d)$.
\end{proof}

Recall a \emph{uniform lattice} in a locally compact group is a cocompact discrete subgroup. A \emph{uniform lattice embedding} is a monomorphism $\rho:\Gamma\to G$ whose image is a uniform lattice. The following elementary lemma relates uniform lattice embeddings and weak  topological completions:
\begin{lem}\label{lem:uniform_lattice}
	If $\Gamma$ is a finitely generated group, $G$ is a locally compact group and  $\rho:\Gamma\to G$ is a uniform lattice embedding, then there is a cobounded quasi-action $G\qa \Gamma$ with topological completion $G$. 
\end{lem}
\begin{proof}
	Since $G$ contains a finitely generated group as a uniform lattice, $G$ is compactly generated. Fixing a geodesically adapted metric on $G$, Proposition \ref{prop:copci} implies $\rho:\Gamma\to G$ is a quasi-isometry. Hence we can quasi-conjugate the regular representation $G\curvearrowright G$ to a cobounded quasi-action $G\qa \Gamma$. By Proposition \ref{prop:TCvsQC}, $G$ is a topological completion of $G\qa \Gamma$.
\end{proof}

\subsection{Uniqueness of topological completions}\label{sec:topcomp_unique}
When studying the large-scale geometry of finitely generated groups, it is often convenient to study groups up to passing to and from finite index subgroups and quotienting out by finite normal subgroups. Similarly, one frequently studies locally compact groups up  to quotienting out by a  compact normal subgroup. This is the case  for topological completions, where elements in such a subgroup correspond to quasi-isometries that are  uniformly close to the  identity map:

\begin{prop}\label{prop:compact_normal}
	Let $X$ be a metric space and $G\qa X$ be a cobounded quasi-action. Suppose $\rho:G\rightarrow \hat G$ is a weak topological completion and $L\leq \hat G$ is a compact normal subgroup. Then there is a constant $D\geq 0$ such that for every $g\in \rho^{-1}(L)$ and $x\in X$, $d(g\cdot x,x)\leq D$.
\end{prop}
\begin{proof}
	Suppose $G\qa X$ is an $(K,A)$-quasi-action.
	We set $J\coloneqq \rho^{-1}(L)$ and choose $x_0\in X$. Since $\overline{\rho(J)}\subseteq L$ is compact, $J$ is bounded. Pick  $R$ sufficiently large such that $d(x_0,g\cdot x_0)\leq R$ for all $g\in J$. 
	Since $G\qa X$ is cobounded, there is a constant $C$ such that for all $x\in X$, there exists a $g_x\in G$ such that $d(g_x\cdot x,x_0)\leq C$.  Using the fact that $J$ is a normal subgroup of $G$, we see that for all $g\in J$ and $x\in X$, setting $h=g_xgg_x^{-1}\in J$ and applying Lemma \ref{lem:qaction_gp},
	\begin{align*}
		d(g\cdot x,x)&\leq d(g_x^{-1}h\cdot g_x\cdot x,x) +A\\
		& \leq Kd(h\cdot g_x\cdot x,g_x\cdot x)+4A\\
		&\leq Kd(h\cdot x_0,x_0)+K^2C +KA+KC+4A\\
		&\leq KR+K^2C+KA+KC+4A\eqqcolon D
	\end{align*}
as required.
\end{proof}

Having established that compact normal subgroups can be quotiented out without losing any geometric information about the quasi-action, we now show that topological completions are unique modulo a compact normal subgroup. Unlike the results of the previous section, Theorem \ref{thm:topcomp unique} only holds for topological completions, not weak topological completions.
 \begin{thm}[Uniqueness of topological completions]\label{thm:topcomp unique}
	Let $X$ be a quasi-geodesic metric space and let $G\qa X$ be a cobounded quasi-action. Suppose that  for $i=1,2$,  $\rho_i:G\rightarrow\hat G_i$  are  two topological completions of  $G\qa X$. Then there exist compact normal subgroups $K_i\vartriangleleft \hat G_i$ and a topological isomorphism $\lambda:\hat G_1/K_1\rightarrow \hat G_2/K_2$ such that the following diagram commutes
	\begin{figure}[H]
		\begin{tikzcd}[row sep=tiny,ampersand replacement=\&]
			\& \hat G_1 \ar[r,"\pi_1",twoheadrightarrow] \& \hat G_1/K_1 \ar[dd,"\lambda"]\\
			G\ar[ur,"\rho_1"]\ar[dr,"\rho_2"]\& \& \&	\\
			\& \hat G_2 \ar[r,"\pi_2",twoheadrightarrow] \& \hat G_2/K_2 
		\end{tikzcd},
	\end{figure}
\noindent where  $\pi_i:\hat G_i\rightarrow \hat G_i/K_i$ are  quotient maps. 
\end{thm}
The difficulty in proving Theorem \ref{thm:topcomp unique} is that   $\hat G_1$ and $\hat G_2$ are not a priori algebraically related in any way that preserves the topology. Our strategy  is to identify the set of ``Cauchy sequences'' in $G$, denoted $\cC_1$ and $\cC_2$, whose images correspond precisely to convergent sequences  in $\hat G_1/K_1$ and $\hat G_2/K_2$. To do this, we require the following elementary lemma  which says that the limit of a  sequence of convergent sequences converges if and only if any sufficiently large diagonal sequence converges. Given two functions $\mu,\kappa:\bbN\to\bbN$, we write $\kappa\geq \mu$ if $\kappa(i)\geq \mu(i)$ for all $i\in \bbN$.
\begin{lem}\label{lem:diagonalconv}
	Let $X$ be a first countable regular topological space. Suppose $S$ is a set   and $\rho:S\rightarrow X$ is an arbitrary  map. Let $\cC_\rho$ be the set of all sequences $(s_i)$ in $S$ such that $\rho(s_i)$ converges in $X$, and let $\Lambda:\cC_\rho\rightarrow X$ be the map $(s_i)\mapsto \lim_{i\rightarrow \infty}\rho(s_i)$.
	Suppose $x\in X$ and  that for each $j\in \bbN$, there is a sequence $c_j=(s_{i,j})_{i=1}^\infty\in \cC_\rho$. The following are equivalent:
	\begin{enumerate}
		\item \label{item:diagonalconv1} $(\Lambda(c_j))_{j=1}^\infty$ converges to $x\in X$;
		\item \label{item:diagonalconv2} there is a  function $\mu:\bbN\rightarrow \bbN$ such that $(\rho(s_{\kappa(j),j}))$ converges to $x$ for all $\kappa\geq \mu$.
	\end{enumerate} 
\end{lem}
\begin{proof}
	(\ref{item:diagonalconv1})$\implies$ (\ref{item:diagonalconv2}):  Let $(U_i)$ be a countable  open neighbourhood basis of $x$ such that $U_0=X$. Replacing $U_i$ with $\cap_{k=1}^iU_k$ if necessary, we can assume $U_i\subseteq U_j$ when $i\geq j$. Since $(\Lambda(c_j))$ converges to $x\in X$, there exists a strictly  increasing sequence $(n_k)$ such that $n_0=0$ and  $\Lambda(c_j)\in U_k$ whenever $j\geq n_k$.
	For each $j$, choose $k(j)$ such that $n_{k(j)}\leq j< n_{k(j)+1}$; thus $(k(j))_{j=1}^\infty$ is an unbounded sequence. Since $\lim_{i\rightarrow \infty}(\rho(s_{i,j})) =\Lambda(c_j)\in U_{k(j)}$ and $U_{k(j)}$ is open, we pick $\mu(j)$ such that $\rho(s_{i,j})\in U_{k(j)}$ for all $i\geq \mu(j)$. If   $\kappa\geq \mu$, then for all $j$, we have $\rho(s_{\kappa(j),j})\in U_{k(j)}$. Since $\lim_{j\rightarrow\infty}k(j)=\infty$ and  $(U_j)$ is a neighbourhood basis of  $x$, it follows that $(\rho(s_{\kappa(j),j}))$ converges to $x\in X$.
	
	(\ref{item:diagonalconv2})$\implies$ (\ref{item:diagonalconv1}): We prove the contrapositive, i.e.\ we assume that $(\Lambda(c_j))$ does not converge to $x$ and show for every   $\mu:\bbN\rightarrow \bbN$, there exists a function $\kappa\geq \mu$ such that  $(\rho(s_{\kappa(j),j}))$ does not converge to $x$. Since $(\Lambda(c_j))$ is assumed to not converge to $x$, there exists an open neighbourhood $U$ of $x$ such that $D\coloneqq \{j\in \bbN\mid \Lambda(c_j)\notin U\}$ is infinite.  Since $X$ is regular, there exists a neighbourhood $V$ of $x$ such that $\overline{V}\subseteq U$. For each $j\in D$, pick an open neighbourhood $U_j$ of $\Lambda(c_j)$ that is disjoint from $V$.  For $j\in D$, we pick $\kappa(j)\geq \mu(j)$ large enough so that $\rho(s_{\kappa(j),j})\in U_j$. We define $\kappa (j)=\mu(j)$ for all $j\notin D$. Since $V$ is a neighbourhood of $x$ and $\rho(s_{\kappa(j),j})\notin V$ for all  $j\in D$, we see that $\rho(s_{\kappa(j),j})$ cannot converge to $x$.
\end{proof}

\begin{proof}[Proof of Theorem \ref{thm:topcomp unique}]
	We first reduce to the case that both $\hat G_1$ and $\hat G_2$ are second-countable and metrisable. Indeed, if this is not the case,  then Corollary \ref{cor:kakutanikodaira} and Remark \ref{rem:topcomp,normal} ensure one can quotient by compact normal subgroups $L_i\vartriangleleft \hat G_i$ such that $\hat G_i/L_i$ is second-countable and metrisable. Applying the subsequent argument to these quotients then proves the theorem. Working with metrisable spaces simplifies the subsequent argument,  since it allows one to work with sequences rather than nets and allows us to apply Lemma \ref{lem:diagonalconv}.
	
	\underline{Defining the compact normal subgroups $K_1$ and $K_2$}
	An element $g\in \hat G_1$ is said to be \emph{$\rho_2$-trivial} if there exists a sequence $(g_i)$ in $G$ such $\rho_1(g_i)$ and $\rho_2(g_i)$  converge to $g\in \hat G_1$ and $1\in\hat G_2$ respectively. Let $K_1\subseteq \hat G_1$ be the set of $\rho_2$-trivial elements. We claim $K_1$ is a compact normal subgroup. Suppose $g,h\in  K_1$. Then there exist sequences $(g_i)$ and $(h_i)$ in $G$ such that  $\rho_1(g_i)\rightarrow g$, $\rho_2(g_i)\rightarrow 1$, $\rho_1(h_i)\rightarrow h$ and $\rho_2(h_i)\rightarrow 1$. Thus, $\rho_1(h^{-1}_ig_i)\rightarrow h^{-1}g$ and $\rho_2(h^{-1}_ig_i)\rightarrow 1$, verifying that $h^{-1}g\in K_1$. Since $1\in K_1$, $K_1$ is a subgroup.
	
	We now show $K_1$ is normal. Suppose $k\in K_1$ and $h\in \hat G_1$. We choose sequences $(k_i)$ and $(h_i)$ in $G$ such that $\rho_1(k_i)\rightarrow k$, $\rho_1(h_i)\rightarrow h$ and $\rho_2(k_i)\rightarrow 1$. Since $(\rho_1(h_i))$ is convergent,  $\{\rho_1(h_i)\mid i\in \bbN\}$ has compact closure;  therefore  $\{h_i\mid i\in \bbN\}$ is bounded and  so $\{\rho_2(h_i)\mid i\in \bbN\}$ also has compact closure. This implies $\rho_2(h_i)$ has a subsequence $\rho_2(h_{n_i})$ that converges to  $l\in \hat G_2$. It follows that $\rho_2(h_{n_i}k_ih_{n_i}^{-1})$ converges to $l1l^{-1}=1$. Since $\rho_1(h_{n_i}k_ih_{n_i}^{-1})$ converges to $hkh^{-1}$, we see that $hkh^{-1}$ is $\rho_2$-trivial. Therefore, $K_1$ is normal.
	
	Let $V_2$ be a compact identity neighbourhood in $\hat G_2$.  It follows from the definition of a topological completion that $V_1\coloneqq \overline{\rho_1(\rho_2^{-1}(V_2))}$ is compact. Thus if  $\rho_2(k_i)\rightarrow 1$, then $\rho_2(k_i) \in V_2$  for all $i$ sufficiently large, and so $\rho_1(k_i) \in V_1$ for all $i$ sufficiently large. Since $V_1$ is compact and contains $K_1$, to show $K_1$ is compact it is sufficient to show it is closed. If $k\in \overline{K_1}$, then there exists a sequence  $(k_j)$ in $K_1$ converging to $k$. Since $k_j\in K_1$ for each $j$, there exists a sequence $(g_{i,j})_{i=1}^\infty$ in $G$ such that $\lim_{i\rightarrow\infty}\rho_1(g_{i,j})= k_j$ and $\lim_{i\rightarrow\infty}\rho_2(g_{i,j})=1$. Applying Lemma \ref{lem:diagonalconv} to the  maps $\rho_i:G\rightarrow \hat G_i$ for both $i=1,2$, we deduce there exists a function $\kappa:\bbN\rightarrow \bbN$ such that $\lim_{j\rightarrow\infty}\rho_1(g_{\kappa(j),j})= k$ and $\lim_{j\rightarrow\infty}\rho_2(g_{\kappa(j),j})=1$, demonstrating that $k\in K_1$. Thus $K_1$ is closed, and so is a compact normal subgroup.
	Interchanging the roles of $\hat G_1$ and $\hat G_2$ gives a compact normal subgroup $K_2\vartriangleleft \hat G_2$ consisting of $\rho_1$-trivial elements of $\hat G_2$. 
	
	\underline{Defining  $\cC_1$ and $\cC_2$ and surjective homomorphisms $\Lambda_1$ and $\Lambda_2$}
	For $i=1,2$, let $\cC_i$ be the set of all sequences $(g_j)$ in $G$ such that $\rho_i(g_j)K_i$ converges in $\hat G_i/K_i$, and let $\Lambda_i:\cC_i\rightarrow \hat G_i/K_i$ be given by $(g_j)\mapsto \lim_{j\rightarrow \infty}(\rho_i(g_j))$. Equipped with componentwise multiplication, $\cC_i$ is a group  and  $\Lambda_i$ is a homomorphism.  Since $\rho_i(G)$ is dense, $\Lambda_i$ is surjective for $i=1,2$.
	
	\underline{$\cC_1=\cC_2$}
	We claim $\cC_1=\cC_2$. Indeed, suppose that $(g_i)\in \cC_1$. Then $\{\rho_1(g_i)\mid i\in \bbN\}$ has compact closure, so  $\{g_i\mid i\in \bbN\}$ is bounded and hence $Q\coloneqq\{\rho_2(g_i)\mid i\in \bbN\}$ also has compact closure. To show $(g_i)\in \cC_2$, it is sufficient to show that  all accumulation points of $Q$ are contained in the same $K_2$-coset.   Let $h,k\in \hat G_2$ be  two such accumulation points of $Q$, and choose subsequences $(\rho_2(g_{n_i}))$ and $(\rho_2(g_{m_i}))$ that converge to $h$ and $k$ respectively. Since $\{\rho_1(g_i)\mid i\in \bbN\}$ also has compact closure, we can assume --- by passing to further subsequences --- that $\rho_1(g_{n_i})$ and $\rho_1(g_{m_i})$ both converge to $h_1,k_1\in \hat G_1$ respectively. Since $(g_i)\in \cC_1$, we see $h_1^{-1}k_1\in K_1$ and so there exists a sequence $(t_i)$ in $G$ such that $(\rho_1(t_i))\rightarrow h_1^{-1}k_1$ and $(\rho_2(t_i))\rightarrow 1$.
	Thus  \[\lim_{i\rightarrow \infty}\rho_1(t_i^{-1}g_{n_i}^{-1}g_{m_i})=(h_1^{-1}k_1)^{-1}h_1^{-1}k_1=1\] and
	\[\lim_{i\rightarrow \infty}\rho_2(t_i^{-1}g_{n_i}^{-1}g_{m_i})=h^{-1}k\]
	so that $h^{-1}k\in K_2$ as required. Thus $(g_i)\in \cC_2$. By symmetry, we deduce $\cC_1=\cC_2$.
	
	\underline{$\ker(\Lambda_1)=\ker(\Lambda_2)$}
	We now show $\ker(\Lambda_1)=\ker(\Lambda_2)$. Suppose $(g_i)\in \ker(\Lambda_1)$. As before, $Q\coloneqq\{\rho_2(g_i)\mid i\in \bbN\}$ has compact closure. To demonstrate that $(g_i)\in \ker(\Lambda_2)$, it is  enough to show that every accumulation point of $Q$ is contained in $K_2$. Let $h$ be an accumulation point of $Q$ and let $(\rho_2(g_{n_i}))$  be a subsequence converging to $h$. As in the previous paragraph, we pass to a subsequence such that $(\rho_1(g_{n_i}))$ converges to some $h_1\in \hat G_1$. Since $(g_i)\in \ker(\Lambda_1)$, we must have $h_1\in K_1$, so there exists a sequence $(t_i)$ in $G$ such that $(\rho_1(t_i))\rightarrow h_1$ and $(\rho_2(t_i))\rightarrow 1$. Thus $(\rho_1(t_i^{-1}g_{n_i}))\rightarrow 1$ and $(\rho_2(t_i^{-1}g_{n_i}))\rightarrow h$, demonstrating that $h\in K_2$ and so $(g_i)\in \ker(\Lambda_2)$. By symmetry, we conclude $\ker(\Lambda_1)=\ker(\Lambda_2)$.
	
	\underline{Defining the topological isomorphism $\lambda: \hat G_1/K_1\rightarrow \hat G_2/K_2$.}
	Let $\iota:G\rightarrow \cC_1=\cC_2$ be the map $g\mapsto (g)$ taking an element to the constant sequence. The composition $\Lambda_i\circ \iota:G\rightarrow \hat G_i/K_i$ is equal to $\pi_i\circ \rho_i$ for $i=1,2$. Since $\Lambda_1$ and $\Lambda_2$ are surjective and have equal domains and kernels,  the first isomorphism theorem furnishes us with an abstract isomorphism $\lambda: \hat G_1/K_1\rightarrow \hat G_2/K_2$ such that $\lambda\circ \Lambda_1=\Lambda_2$. In particular, we have $\lambda\circ \pi_1\circ \rho_1=\pi_2\circ\rho_2$. All that remains is to show that $\lambda$ is a homeomorphism. We will show that the bijection $\lambda$ is open. It will then follow by symmetry that $\lambda^{-1}$ is open, i.e.\ $\lambda$ is continuous, hence is  a homeomorphism.
	
	If $U\subseteq \hat G_1/K_1$ be open, we show that $\lambda(U)$ is open. Suppose $g\in \lambda(U)$ and let $(g_j)$ be a sequence in $\hat G_2/K_2$ converging to $g$. It is enough to show there exists an $N$ such that $g_j\in \lambda(U)$ for all $j\geq N$. Pick $c_j=(g_{i,j})_{i=1}^\infty\in \cC_2$ such that $\Lambda_2(c_j)=g_j$.  Since $\lim_{j\rightarrow \infty} \Lambda_2(c_j)=g$, Lemma \ref{lem:diagonalconv} ensures that there is a function  $\mu:\bbN\rightarrow \bbN$ such that for every $\kappa\geq \mu$, $c_\kappa\coloneqq (g_{\kappa(j),j})\in \cC_2$ and $\Lambda_2(c_\kappa)=g$. Since $\cC_1=\cC_2$ and $\lambda\circ \Lambda_1=\Lambda_2$, we see that  $c_\kappa\in \cC_1$ and $\Lambda_1(c_\kappa)=\lambda^{-1}(g)$ for all $\kappa\geq\mu$. 
	It follows from Lemma \ref{lem:diagonalconv} that $\lim_{j\rightarrow \infty} \Lambda_1(c_j)=\lambda^{-1}(g)$. Since $\lambda^{-1}(g)\in U$ and $U$ is open, there exists an $N$ such that $ \Lambda_1(c_j)=\lambda^{-1}(g_j)\in U$ for all $j\geq N$. Thus  $g_j\in \lambda(U)$ for all $j\geq N$. It follows that $\lambda(U)$ is open as required. 
\end{proof}
\subsection{A topological characterisation of discretisability}\label{sec:topcomp_disc} We now give a topological criterion for discretisability of quasi-actions. This will be done using tools from the theory of  totally disconnected locally compact (TDLC) groups.
The foundational result in the theory of TDLC groups is the following:
\begin{thm}[\cite{vandantzig36}]\label{thm:vanDantzig}
	Every identity neighbourhood of a totally disconnected locally compact group  contains a compact open subgroup.
\end{thm}

We  recall some useful facts about compact open subgroups:
\begin{lem}\label{lem:cpctopen}
	Suppose $G$ is a locally compact group and $H$ is a compact open subgroup.
	\begin{enumerate}
		\item \label{item:cpctopen1} $G$ is finitely generated relative to $H$ if and only if $G$ is compactly generated.
		\item \label{item:cpctopen2} Any two compact open subgroups of $G$ are commensurable. In particular, $H\alnorm G$.
	\end{enumerate}
\end{lem}
\begin{proof}
	(\ref{item:cpctopen1}): Suppose $G$ is generated by a compact subset $K$. Since $ \{gH\}_{g\in G}$ is an open cover of $G$, there is a finite subset $S\subseteq G$ such that $K\subseteq \bigcup_{s\in S} sH$; thus $S$ is a finite generating set of the pair $(G,H)$. Conversely, if $S$ is a finite generating set of $(G,H)$, then $\bigcup_{s\in S} sH$ is a compact generating set of $G$.  
	
	(\ref{item:cpctopen2}): Suppose $H$ and $H'$ are compact open subgroups of $G$. The preceding argument implies  $H$ is contained in finitely many left cosets of $H'$, and $H'$ is contained in finitely many left cosets of $H$. Thus $H$ and $H'$ are commensurable. Since a conjugate of a compact open subgroup is also compact and open, we deduce that $H$ is commensurated.
\end{proof}

If $G$ is a compactly generated locally compact group, $H\leq G$ is a compact open subgroup and $S$ is a relative generating set of $(G,H)$, then the relative Cayley graph $\Gamma_{G,H,S}$, as defined in Section \ref{sec:alnorm}, is known as a \emph{Cayley--Abels graph} or a \emph{rough Cayley graph}; see for instance \cite{kronmoller08roughcayley}. Such a graph is always connected and locally finite, and the natural action of $G$ on $\Gamma_{G,H,S}$ is  continuous, proper and vertex-transitive with coarse stabiliser $H$. 
 The preceding lemma, coupled with  Proposition \ref{prop:cstabinvariant},  implies the following:
\begin{cor}\label{cor:tdlc_CAgraph}
	Let $G$ be  a compactly generated locally compact group containing a compact open subgroup.  Then $G$ acts geometrically on its Cayley--Abels graph. Moreover, up to quasi-conjugacy, this action is  independent of the choice of compact open subgroup and  finite relative generating set.
\end{cor}  
We can thus talk about  \emph{the} Cayley--Abels graph of a compactly generated locally compact  group containing a compact open subgroup.
Modulo a compact normal subgroup, the converse to Theorem \ref{thm:vanDantzig} holds:
\begin{lem}\label{lem:compact_open}
	A  locally compact group is compact-by-totally disconnected if and only if it contains a compact open subgroup.
\end{lem}
\begin{proof}
	Suppose $G$ contains a compact open subgroup $H$. Let $K\coloneqq \cap_{g\in G}gHg^{-1}$. Since $K$ is a compact normal subgroup of $G$, it is sufficient to show $G/K$ is totally disconnected. Let $\pi:G\to G/K$ be the quotient map and suppose $g_1K\neq g_2K$. Then there is some  $g\in G$ such that $g_1gH\neq g_2gH$. Now $\pi(g_1gH)$ and $\pi(g_2gH)$ are distinct cosets of the open subgroup $\pi(H)$, hence are clopen. Thus $g_1K$ and $g_2K$ lie in disjoint clopen subsets of $G/K$.  It follows that  $G/K$ is totally disconnected.
	
	Conversely, suppose  $G$ has a compact normal subgroup $K\vartriangleleft G$ such that $G/K$ is totally disconnected.  By Theorem \ref{thm:vanDantzig}, there is a compact open subgroup $H_0\leq G/K$. Since $K$ is compact, the preimage of $H_0$ under the quotient map $G\rightarrow G/K$ is a compact open subgroup of $G$.
\end{proof}

 We can now prove the following:

\begin{prop}\label{prop:disc_equiv}
\discequivs
\end{prop}
\begin{proof}
$(1)\Longleftrightarrow (2)$: This is  Corollary \ref{cor:discvscstab}.

$(1)\Longrightarrow (3)$: By definition, if $G\qa X$ is discretisable, there exists a locally finite graph $Y$ such that $G\qa X$ is quasi-conjugate to an action $\rho: G\rightarrow \Aut(Y)$. Since $\Aut(Y)$ is totally disconnected, every closed subgroup of $\Aut(Y)$ is also totally disconnected. It follows  from Proposition \ref{prop:TCvsQC} that $G\qa X$ has a totally disconnected topological completion. By Theorem \ref{thm:topcomp unique}, every topological completion of $G\qa X$ is compact-by-(totally disconnected).

$(3)\Longrightarrow (2)$: Suppose $G\qa X$ has a topological completion $\rho: G\rightarrow \hat G$ with $\hat G$ compact-by-(totally disconnected). By Proposition \ref{prop:topcompquasiconj} and Lemma \ref{lem:compact_open}, $\hat G$ contains a compact open subgroup $\hat H$. Let $H\coloneqq \rho^{-1}(\hat H)$. We claim that $H$ is a coarse stabiliser of $G\qa X$. Firstly, since $\hat H$ is compact it follows $H$ is bounded. Secondly, if $T\subseteq G$ is bounded, then $\overline{\rho(T)}$ is compact so is contained in finitely many left $\hat H$-cosets. It is evident that if for some $g,k\in T$, $\rho(g)$ and $\rho(k)$ are contained in the same left $\hat H$-coset, then $g$ and $k$ are contained in the left $H$-coset. Thus $T$ is contained in finitely many left $H$-cosets, verifying $H$ is indeed a coarse stabiliser of $G\qa X$. 
\end{proof}
The following corollary gives a sufficient condition for a finitely generated group to admit a quasi-action that is not discretisable.
\begin{cor}\label{cor:unif_lattice_not_disc}
	Let $\Gamma$ be a finitely generated group that is a uniform lattice in a  locally compact group $G$ that is not compact-by-(totally disconnected). Then there is a cobounded quasi-action $G\qa \Gamma$ that is not discretisable.
\end{cor}
\begin{proof}
	By Lemma \ref{lem:uniform_lattice}, there is a quasi-action $G\qa \Gamma$ with topological completion $G$. Since $G$  is not compact-by-(totally disconnected), Proposition \ref{prop:disc_equiv} implies $G\qa \Gamma$ is not discretisable.
\end{proof}

We conclude with a criterion determining when a topological completion is discrete:
\begin{prop}\label{prop:disc_topcomp}
	Let $X$ be a quasi-geodesic metric space and let $G\qa X$ be a cobounded quasi-action with coarse stabiliser $H$. Then the following are equivalent:
	\begin{enumerate}
		\item $H$ is commensurable to a normal subgroup of $G$;
		\item  $G\qa X$ has a discrete topological completion;
		\item every topological completion of $G\qa X$ is compact-by-discrete.
	\end{enumerate} 
\end{prop}
\begin{proof}
	$(1)\implies (2)$:  Suppose $H$ is commensurable to a normal subgroup $H_1$.  Then $G$ acts on the Cayley graph $Y$ of $G/H_1$ with coarse stabiliser $H_1$  by Example \ref{exmp:stabvertex,cstab}.  Propositions \ref{prop:coarsestabsarecomm} and \ref{prop:cstabinvariant} ensure  $G\qa X$ is quasi-conjugate to the natural action $G\to \Aut(Y)$. Since the image of $G$ in $\Aut(Y)$ is  discrete, Proposition \ref{prop:TCvsQC} ensures $G\qa X$ has discrete topological completion  $G/H_1$.
	
	$(2)\implies (1)$: Suppose $\rho:G\to \hat G$ is a topological completion of $G\qa X$ with $\hat G$ discrete. Let $H'\coloneqq \ker(\rho)$. We claim $H'$ is a coarse stabiliser of $G\qa X$. Since $\rho$ is a topological completion and  $\rho(H')=\{e\}$ is compact, $H'$ is bounded and so \ref{item:CS1} holds. Moreover, if $T\subseteq G$ is bounded, then as $\hat G$ is discrete,  $\rho(T)$ is finite hence $T\subseteq \rho^{-1}(\rho(T))$ is contained in finitely many left $H'$-cosets; thus \ref{item:CS2} holds. Therefore  $H'$ is a coarse stabiliser of $G\qa X$. By Proposition \ref{prop:coarsestabsarecomm}, $H$ is commensurable to the normal subgroup $H'$. 
	
	$(2)\iff (3)$: Theorem \ref{thm:topcomp unique} implies that if some topological completion is discrete, then every topological completion is compact-by-discrete. Conversely, Remark \ref{rem:topcomp,normal} ensures if some topological completion is compact-by-discrete, then there exists a discrete topological completion.
\end{proof}

\subsection{Tame spaces}\label{sec:tame}
In this subsection we investigate tame spaces, which were introduced by Whyte (see \cite{whyte01coarse,whyte2010coarse}). See also the related notion of \emph{nontranslatable} spaces due to Kapovich--Kleiner--Leeb \cite{kapovich1998derham}.  Our interest in tameness is due to  Theorem \ref{thm:tame_topcomp}, which states that cobounded quasi-actions on  tame spaces necessarily have topological completions. We show in Proposition \ref{prop:morsebdry_tame} that spaces which  satisfy a rather weak form of coarse negative curvature are tame.

We first establish some notation that will be used in this  subsection.
Given functions $f,g:X\rightarrow Y$ between metric spaces, a subset $F\subseteq X$ and $r\geq 0$, we write:
\begin{enumerate}
	\item  $f\close{F,r} g$ if $d(f(x),g(x))\leq r$ for all $x\in F$;
	\item $f\close{r} g$ if  $f\close{X,r} g$; 
	\item $f\close{} g$ if $f\close{r} g$ for some $r\geq 0$.
\end{enumerate}

\begin{defn}\label{defn:tame}
	We say that a space $X$ is \emph{tame} if for any $K\geq 1$ and $A\geq 0$, there exists a constant $B$ such that for every  $(K,A)$-quasi-isometry $f:X\rightarrow X$ such that $f\sim \id_X$, then $f\close{B} \id_X$. 
\end{defn}
We emphasise  that unless $X$ is bounded, the constant $B$ in the preceding definition  depends on the constants $K$ and $A$.
It is clear from the definition that tameness is a quasi-isometry invariant. The following example shows groups with infinite centre are not tame:

\begin{exmp}
	Let $G$ be a finitely generated  group equipped with a word metric $d_S$ with respect to some finite generating set $S$. Let $a\in G$ be central in $G$ and let $L_a:G\rightarrow G$ be left multiplication by $a$, which is an isometry of $(G,d_S)$. Then for all $g\in G$, \[d_S(L_a(g),g)=d_S(ag,g)=d_S(ga,g)=d_S(a,1_G)= \lvert a\rvert _S,\] so that $L_a\sim \id_G$. It follows that if $G$ has  infinite centre, then it cannot be tame.
\end{exmp}

We  require the following reformulation of tameness, which is easily seen to be equivalent to Definition \ref{defn:tame}:
\begin{lem}\label{lem:tame}
	A metric space $X$ is tame if and only if for every $K\geq 1$ and $A\geq 0$, there is a constant $B\geq 0$ such that the following holds: if
	$f,g:X\rightarrow X$ are $(K,A)$-quasi-isometries such that $f\sim g$, then $f\close{B}g$.
\end{lem}

Our interest in tame metric space comes from the following theorem, whose proof is delayed till later in this section.
\begin{thm}\label{thm:tame_topcomp}
	Let $X$ be a tame proper quasi-geodesic metric space. Then every cobounded quasi-action on $X$ has a topological completion.
\end{thm}

It is shown in \cite{whyte2010coarse} that non-elementary hyperbolic groups are tame. We will shortly see that  a much larger class of groups are tame, namely groups whose Morse boundary contains at least three points. We first explain how the Morse boundary is defined. 

Let $X$ be a proper geodesic metric space. Recall that  a \emph{$(K,A)$-quasi-geodesic} is the image of a  $(K,A)$-quasi-isometrically  embedded interval $I\subseteq \bbR$. The \emph{endpoints}  of a quasi-geodesic are the images of the endpoints of $I$ (if they exist).  Given a function $\phi:[1,\infty)\times [0,\infty)\rightarrow [0,\infty)$, we say that a geodesic $\gamma:\bbR\rightarrow X$ is \emph{$\phi$-Morse} if for any $(K,A)$-quasi-geodesic $\lambda$ with endpoints on $\gamma$, then $\lambda\subseteq N_{\phi(K,A)}(\gamma)$. A geodesic is said to be \emph{Morse} if it is $\phi$-Morse for some $\phi$, where $\phi$ is called the \emph{Morse gauge}.

The \emph{Morse boundary} of $X$, denoted $\partial_*X$, is defined to be the set of Morse geodesic rays in $X$ up to equivalence, where two rays are equivalent if they are at finite Hausdorff distance. Although the Morse boundary can be given the  structure of a topological space, studying the Morse boundary as a set is sufficient for our purposes.  The Morse boundary  was first defined by Cordes for arbitrary proper geodesic metric spaces in  \cite{cordes2017Morse}, and agrees with the Gromov boundary on hyperbolic spaces and the contracting boundary on CAT(0) spaces. Quasi-isometries preserve Morse geodesic rays up to finite Hausdorff distance \cite[Lemma 2.9]{cordes2017Morse}, so a quasi-isometry $f:X\rightarrow X$ induces a bijection $\partial_*f:\partial_*X\rightarrow \partial_*X$.

Although the Morse boundary can be empty, for a large class of groups it is not. In particular, it follows from work of Sisto that  if a group is acylindrically hyperbolic, then its Morse boundary contains at least three points \cite{sistosro2016Quasiconvexity}. Our interest in groups with nontrivial Morse boundary comes from the following:
\begin{prop}\label{prop:morsebdry_tame}
	A cocompact proper geodesic metric space whose Morse boundary contains at least three points is tame. 
\end{prop}
Coupled with Theorem \ref{thm:tame_topcomp}, this implies the following:
\begin{cor}\label{cor:morse_topcomp_exist}
	Let $X$ be a cocompact proper geodesic metric space whose   Morse boundary contains at least three points, e.g. an acylindrically hyperbolic group. Then every cobounded quasi-action on $X$ has a topological completion.
\end{cor}

To prove Proposition \ref{prop:morsebdry_tame}, we make use of some lemmas from \cite{cordes2017Morse} and \cite{charneyCordesMurray2019QuasiMobius}. The following lemma ensures that quasi-isometries send Morse geodesics to Morse geodesics up to some uniform Hausdorff distance. This lemma is implicit in the proof of  Theorem 3.7 in \cite{charneyCordesMurray2019QuasiMobius}, and can easily be deduced by adapting  the proof of Lemma 2.9 in \cite{cordes2017Morse}.
\begin{lem}\label{lem:qimorse}
	 Given $K\geq 1$ and $A\geq 0$ and a Morse gauge $\phi$, there is a constant $B\geq 0$ such that for any $(K,A)$-quasi-isometry $f:X\rightarrow X$ and any $\phi$-Morse bi-infinite geodesic $\gamma\subseteq X$, $f(\gamma)$ is Hausdorff distance at most $B$ from a Morse geodesic.
\end{lem}

Given a Morse gauge $\phi$, Charney et al.\ define  $\partial_*X^{(3,\phi)}$  to be the set of distinct triples $(\xi_1,\xi_2,\xi_3)$ in $\partial_* X$ such that  for each $i\neq j$, every bi-infinite geodesic from $\xi_i$ to $\xi_j$ is $\phi$-Morse \cite{charneyCordesMurray2019QuasiMobius}. As noted in \cite{charneyCordesMurray2019QuasiMobius}, if $\partial_* X$ contains three points, then $\partial_*X^{(3,\phi)}$ is non-empty for some $\phi$ large enough.

Let $T(\xi_1,\xi_2,\xi_3)$ denote an ideal triangle with vertices $\xi_1,\xi_2,\xi_3\in \partial_*X$.
It is well-known that for a triple of distinct points $(\xi_1,\xi_2,\xi_3)$ in the boundary of a hyperbolic group, the centre of $T(\xi_1,\xi_2,\xi_3)$ is well-defined up to some uniform error. For some fixed Morse gauge $\phi$, the same is true for triples in $\partial_*X^{(3,\phi)}$:

\begin{lem}[{\cite[Lemma 2.5]{charneyCordesMurray2019QuasiMobius}}]\label{lem:coarse_centre}
	Given a Morse gauge $\phi$, there is a constant $r_0$ such that for all $r\geq r_0$, there is an $R=R(r,\phi)$ such that for all $(\xi_1,\xi_2,\xi_3)\in \partial_*X^{(3,\phi)}$, the set 
	\begin{align*}
		E_r(\xi_1,\xi_2,\xi_3)\coloneqq \left\{ x\in X\mid \text{$x$ lies within $r$ of all three sides of some $T(\xi_1,\xi_2,\xi_3)$}	\right\}
	\end{align*}
	is non-empty and has diameter at most $R$.
\end{lem}

We are now in position to apply these lemmas  to prove Proposition \ref{prop:morsebdry_tame}:
\begin{proof}[Proof of Proposition \ref{prop:morsebdry_tame}]
	Since $\partial_*X$ has at least three points, say $\xi_1$, $\xi_2$ and  $\xi_3$,  we can choose a Morse gauge $\phi$ such that $(\xi_1,\xi_2,\xi_3)\in \partial_*X^{(3,\phi)}$. By Lemma \ref{lem:coarse_centre}, we can choose $r_1$ large enough so that $E_{r_1}(\xi_1,\xi_2,\xi_3)$ is non-empty and thus contains some $x_0\in X$. Since $X$ is cocompact, we can pick $r_2$ such that for every $x\in X$, there is an isometry $f_x:X\rightarrow X$ such that $d(f_x(x_0),x)\leq r_2$. Let $\xi^x_i\coloneqq (\partial_*f_x)(\xi_i)$ for $i=1,2,3$ and $x\in X$. Since each $f_x$ is an isometry, $(\xi_1^x,\xi_2^x,\xi_3^x)\in \partial_*X^{(3,\phi)}$  and $x\in E_{r_1+r_2}	(\xi_1^x,\xi_2^x,\xi_3^x)$.
	
	Let $K\geq 1$ and $A\geq 0$. Choose $B$ as in Lemma \ref{lem:qimorse} such every $(K,A)$-quasi-isometry sends a $\phi$-Morse bi-infinite geodesic to within Hausdorff distance $B$ from another Morse geodesic. Let $f:X\rightarrow X$ be a $(K,A)$-quasi-isometry such that $f\close{} \id_X$. Let $x\in X$. 
	Since $x\in E_{r_1+r_2}	(\xi_1^x,\xi_2^x,\xi_3^x)$, it follows $x$ has distance at most $r_1+r_2$ from a point on a geodesic from $\xi_1^x$ to $\xi_2^x$. Since $f$ is a $(K,A)$-quasi-isometry with $f\close{}\id_X$,  $f(x)$ has distance at most $r_3\coloneqq K(r_1+r_2)+A+B$ from a point on a (possibly different) Morse geodesic from $\xi_1^x$ to $\xi_2^x$. Repeating this argument for the other sides of the triangle $T(\xi_1^x,\xi_2^x,\xi_3^x)$ allows us to deduce that  $f(x)\in E_{r_3}(\xi_1^x,\xi_2^x,\xi_3^x)$. By Lemma \ref{lem:coarse_centre}, we can pick $R$ large enough so that for every $x\in X$, 	$E_{r_3}(\xi_1^x,\xi_2^x,\xi_3^x)$ has diameter at most $R$. Since $x$ and $f(x)$ are both contained in $ E_{r_3}(\xi_1^x,\xi_2^x,\xi_3^x)$, it follows $d(f(x),x)\leq R$ for all $x\in X$. This shows  $X$ is tame. 
\end{proof}

We now  begin our proof of  Theorem \ref{thm:tame_topcomp}. Our main ingredient is unpublished work of Whyte \cite{whyte01coarse}, which we give a full account of. See also  \cite[Section 4]{dymarz2015envelopes} for a published account of this work. 
The notion of coarse convergence of functions plays a fundamental role in this subsection.
\begin{defn}
	Let $X$  be a metric space. For a sequence of functions $f_i:X\rightarrow X$ and a function $f:X\to X$, we write $f_i\cc{r} f$ if for every $x\in X$, there is an $N_x\in \bbN$ such that $d(f_i(x),f(x))\leq r$ for all $i\geq N_x$. We say $(f_i)$  \emph{coarsely converges} to $f:X\rightarrow X$, written $f_i\cc{} f$, if there exists an $r\geq 0$ such that $f_i\cc{r}f$. 	We say that $(f_i)$ is \emph{coarsely convergent} if it coarsely converges to some $f:X\rightarrow X$, in which case  we say that $f$ is the \emph{coarse limit} of $(f_i)$. 
\end{defn}

	In order to prove Theorem \ref{thm:tame_topcomp}, we first need to develop some preliminary  lemmas concerning coarse convergence of quasi-isometries of tame spaces.  
In what follows, it  will simplify the argument  if we assume that $X$ is locally finite. The following remark ensures this can always be done:

\begin{rem}\label{rem:bdd_geom}
	Suppose $X$ is a proper quasi-geodesic tame metric space and  $G\qa X$ is a quasi-action. Proposition \ref{prop:coarse_proper} ensures that $X$ is quasi-isometric to a locally finite tame metric space $Y$, and hence the quasi-action $G\qa X$ is quasi-conjugate to a quasi-action $G\qa Y$. Clearly $G\qa X$ has a topological completion if and only if $G\qa Y$ does, so it suffices to prove that every cobounded quasi-action on a locally finite tame quasi-geodesic metric space has a topological completion. 
\end{rem}

		We first show coarse convergence of quasi-isometries of  locally finite tame spaces satisfies the following uniformity property:

\begin{lem}\label{lem:uniform_cc_tame}
	Let $X$ be a tame locally finite metric space and let $K\geq 1$ and $A\geq 0$. There exists $B\geq 0$ such that whenever $(f_i)$ is a coarsely convergent sequence of $(K,A)$-quasi-isometries of $X$, the following hold: 
	\begin{enumerate}
		\item There exists a $(K,A)$-quasi-isometry $f:X\rightarrow X$ such that $f_i\cc{B}f$. 
		\item If $f:X\rightarrow X$ is any $(K,A)$-quasi-isometry such that $f_i\cc{} f$, then $f_i\cc{B}f$. 
	\end{enumerate}
\end{lem}
\begin{proof}
	Let $x_0\in X$. Since $(f_i)$ is coarsely convergent, $(f_i(x_0))$ is a bounded sequence. By Lemma \ref{lem:coarse AA}, there exists a  subsequence $(f_{n_i})$ and a $(K,A)$-quasi-isometry $f:X\rightarrow X$  such that for each $x\in X$, $f_{n_i}(x)=f(x)$ for $i$ sufficiently large. We claim $f_i\cc{} f$. Indeed, since $(f_i)$ is coarsely convergent, there exists some $g:X\rightarrow X$ and $r\geq 0$ such that $f_i\cc{r} g$. Thus for each $x\in X$,  $d(f_{n_i}(x),g(x))\leq r$ for $i$ sufficiently large, and hence $d(f(x),g(x))\leq r$. It follows that $f_i\cc{} f$ as required.

	We now choose $B$ as in Lemma \ref{lem:tame} so that if $h,g:X\rightarrow X$ are $(K,A)$-quasi-isometries such that $h\close{} g$, then $h\close{B} g$.  Suppose $g:X\rightarrow X$ is a $(K,A)$-quasi-isometry such that $f_i\cc{} g$. We claim that $f_i\cc{B} g$. Indeed, there exists some $r$ such that $f_i\cc{r}g$. Thus  for all $x\in X$, there exists an $N_x$ such that $d(f_i(x),g(x))\leq r$  for all $i\geq N_x$. Let $I_x\coloneqq \{f_i(x)\mid i\geq N_x\}\subseteq N_{r}(g(x))$. Since $X$ is locally finite, $I_x$ is finite. An element of $y\in I_x$ is \emph{live} if $\{i\in \bbN\mid f_{i}(x)=y\}$ is infinite.  We can thus pick $M_x\geq N_x$ such that every element of $\{f_i(x)\mid i\geq M_x\}$ is live. For every live $y\in I_x$, we pick a subsequence $(f_{m_i})$ such that $(f_{m_i}(x))$ is eventually constant and equal to $y$. Applying Lemma \ref{lem:coarse AA} again we can  pass to a further subsequence $(f_{k_i})$ that converges pointwise to a $(K,A)$-quasi-isometry $h:X\rightarrow X$. Arguing  as in the previous paragraph, we deduce $g\close{} h$. 
	As $h(x)=y$ and both $g$ and $h$ are $(K,A)$ quasi-isometries, it follows  that $d(g(x),y)\leq B$.  Thus $d(f_i(x),g(x))\leq B$ for all $i\geq M_x$, and so $f_i\cc{B} g$.
\end{proof}		
		
In order to use coarse convergence to define a topology, we  show that  coarse convergence of sequences of quasi-isometries are `continuous'  with respect to multiplication and taking inverses.
\begin{lem}\label{lem:mult_inv_cont}
	Let $X$ be a metric space and suppose  $f_i \cc{} f$ and $g_i \cc{} g$, where all maps are self quasi-isometries of $X$ with uniform constants. Then
	\begin{enumerate}
		\item $g_i\circ f_i\cc{} g\circ f$;
		\item $\overline{f_i}\cc{} \overline{f}$, where $\overline{f_i}$ and $\overline{f}$ are coarse inverses to $f_i$ and $f$ with uniform constants.
	\end{enumerate}
\end{lem}
\begin{proof}
	We pick $K$ and $A$ large enough such that  $\overline f\circ f\close{A} \id_X$,  $\overline {f_i}\circ f_i\close{A} \id_X$ for all $i$, and all  maps under consideration are $(K,A)$-quasi-isometries. Choose $r$ large enough such that $f_i \cc{r} f$ and $g_i \cc{r} g$. We claim $g_i\circ f_i\cc{R} g\circ f$, where $R\coloneqq Kr+A+r$. Indeed, let $x\in X$ and choose $N$ large enough such that $d(f_i(x),x)\leq r$ and $d(g_if(x),gf(x))\leq r$ for all $i\geq N$. Then \begin{align*}
		d(g_if_i(x),gf(x))\leq d(g_if_i(x),g_if(x))+d(g_if(x),gf(x))\leq Kr+A+r=R
	\end{align*}
	for all $i\geq N$, so $g_if_i\cc{R} gf$ as required.
	
	We also claim $\overline{f_i}\cc{R'} \overline{f}$, where $R'\coloneqq K(r+2A)+5A$. Indeed, let $y\in X$ and choose $x\in X$ such that $d(f(x),y)\leq A$. We pick $M$ large enough so that $d(f_i(x),f(x))\leq r$ for $i\geq M$. Thus 
	\begin{align*}
		d(\overline{f_i}(y),\overline{f}(y))
		&\leq  d\Big(\overline{f_i}(f(x)),\overline{f}(f(x))\Big)+2KA+2A\\
		&\leq d\Big(\overline{f_i}(f(x)),x\Big)+2KA+3A\leq d\Big(\overline{f_i}(f_i(x)),x\Big)+Kr+2KA+4A\\
		&\leq Kr+2KA+5A=R'
	\end{align*}
	for all $i\geq M$. Thus $\overline{f_i}\cc{R'} \overline{f}$.
\end{proof}		
If $\phi:G\to X^X$ is a quasi-action, we would like to define a topology on $G$ so that $g_i\to g$ if and only if $\phi(g_i)\cc{} g$. Doing so directly gives a topology that might be neither   Hausdorff nor  locally compact. This motivates the following definition:		
		
	\begin{defn}
		A quasi-action $\phi:G\rightarrow  X^X$ is said to be:
		\begin{enumerate}
			\item \emph{coarsely faithful} if whenever $\phi(g)\close{} \id_X$, then $g=e_G$.
			\item \emph{coarsely complete} if whenever $\phi(g_i)$ is a coarsely convergent sequence, there exists a $g\in G$ such that  $\phi(g_i)\cc{}\phi(g)$.
		\end{enumerate} 
	\end{defn}
\begin{rem}
We caution the reader that the  term coarsely faithful is at variance with the usual rule of thumb that prefixing a mathematical property with the adjective `coarsely' weakens the property in question. Indeed, an isometric action $\phi: G\rightarrow \Isom(X)$ might be  faithful but not coarsely faithful. 
\end{rem}
The following proposition allows us to quasi-conjugate a quasi-action to one that is coarsely faithful and coarsely complete:
\begin{prop}\label{prop:coarsecomplete}
	Let $X$ be a locally finite quasi-geodesic tame metric space and let $\phi:G\rightarrow  X^X$ be a cobounded quasi-action. Then there exists a coarsely faithful and coarsely complete cobounded quasi-action $\hat G\qa X$ and a  homomorphism $\rho:G\to \hat G$ such that  the $\id_X$ is a $\rho$-quasi-conjugacy from $G\qa X$ to $\hat G\qa X$. Moreover, every $g\in \hat G$ is the coarse limit of a sequence in $\rho(G)$.
\end{prop}
To prove an action is coarsely complete, we use the following  coarse analogue of Lemma~\ref{lem:diagonalconv}.
\begin{lem}\label{lem:clim of clims}
	Let $X$ be a tame locally finite metric space.
	Suppose that  $(f_j)_{j=1}^\infty\cc{} f$ and $(f_{i,j})_{i=1}^\infty\cc{} f_j$ for each $j$, where all these maps are $(K,A)$-quasi-isometries of $X$. Then there exist sequences $(n_j)$ and $(m_j)$ in $\bbN$ such that $(f_{m_j,n_j})_{j=1}^\infty\cc{} f$. 
\end{lem}
\begin{proof}
	Since $X$ is a locally finite metric space, hence is countable, we may enumerate $X$ as $X=\{x_1,x_2,\dots, \}\subseteq X$. Choose $B$ as in  Lemma \ref{lem:uniform_cc_tame}.  Since $(f_j)\cc{} f$, for each $j\in \bbN$ we can choose $n_j$ such that $d(f(x_l),f_k(x_l))\leq B$ for each $l\leq j$ and $k\geq n_j$. Similarly, for each $j$  we can choose $m_j$ such that $d(f_{n_j}(x_l),f_{k,n_j}(x_l))\leq B$ for each $l\leq j$ and $k\geq m_j$. We claim $(f_{m_j,n_j})\cc{} f$. 
	Indeed, let $x\in X$ so that $x=x_l$ for some $l$.  Therefore \begin{align*}
		d(f_{m_j,n_j}(x),f(x))\leq d(f_{m_j,n_j}(x),f_{n_j}(x))+d(f_{n_j}(x),f(x))
		\leq 2B
	\end{align*}
	for all $j\geq l$, so that $(f_{m_j,n_j})\cc{} f$ as required.
\end{proof}

\begin{proof}[Proof of Proposition \ref{prop:coarsecomplete}]
	Let $\hat G$ be the subset of $\QI(X)$ consisting of  equivalence classes of quasi-isometries that arise as coarse limits of coarsely convergent  sequences in $\{\phi(g)\mid g\in G\}$, and let $\rho:G\rightarrow \hat G$ be the map $g\mapsto [\phi(g)]$. It follows from Lemma \ref{lem:mult_inv_cont} that $\hat G$ is  a subgroup of $\QI(X)$.  By Lemma \ref{lem:uniform_cc_tame} there exists constants $K$ and $A$ such that each $g\in \hat G$ can be represented by a $(K,A)$-quasi-isometry $f_g:X\rightarrow X$. Pick $B\geq A$ as in Lemma \ref{lem:tame} such that if  $f,g:X\to X$ are $(K,A)$-quasi-isometries such that $f\close{} g$, then $f\close{B} g$. Since $f_g\circ f_h\close{B}f_{gh}$ for all $g,h\in G$, the map $\hat\phi:\hat G \rightarrow X^X$ given by $g\mapsto f_g$ is a $(K,B)$-quasi-action $\hat G\qa X$. Since $\phi(g)\close{B}\hat \phi(\rho(g))$ for all $g\in G$, $\id_X$ is a $\rho$-quasi-conjugacy. 
	
	As $\hat G$ is a subgroup of $\QI(X)$ and $g=[\hat\phi (g)]=[f_g]$,  it is clear $\hat \phi$ is coarsely faithful. Moreover,   Lemma \ref{lem:clim of clims} ensures the coarse limit of a coarsely convergent sequence in $\{\hat \phi(g)\mid g\in \hat G\}$ is  a coarse limit of a coarsely convergent sequence in $\{\phi(g)\mid g\in G\}$, implying $\hat \phi$ is coarsely complete.
\end{proof}	
	
	We now prove the following:
	\begin{prop}\label{prop:topcc}
			Let $X$ be a locally finite quasi-geodesic  metric space and let $\phi:G\to X^X$ be a cobounded  quasi-action that is coarsely faithful and coarsely complete.  Then $G$ can be given the structure of a second-countable locally compact compactly generated topological group such that  $(g_i)$ converges to $g$ if and only if $\phi(g_i)\cc{} \phi(g)$. Moreover, $G$ has no nontrivial compact normal subgroups and the identity map $G\to G$ is a topological completion of $\phi$.
	\end{prop}

\begin{rem}
	Since a first countable topological space  is completely  determined by its convergent sequences, any topology on $G$ satisfying the properties in Proposition \ref{prop:topcc} is unique.
\end{rem}
	
	 For the remainder of this section, which is devoted to proving Proposition \ref{prop:topcc},  we  fix $\phi:G\to X^X$ as in the statement of Proposition \ref{prop:topcc}.  We also fix $K$ and $A$ such that $\phi$ is a $(K,A)$-quasi-action, and  a constant $B\geq 0$ as in Lemma \ref{lem:uniform_cc_tame}.
 We define  a topology on $G$ and show it satisfies the required properties. 
For ease of notation, when unambiguous we will write $g\close{} g'$ to denote $\phi(g)\close{} \phi(g')$, $g_i\cc{} g$ to denote  $\phi(g_i)\cc{} \phi(g)$, etc. We recall   $g\cdot x$ can be used to denote $\phi(g)(x)$.
\begin{prop}\label{prop:deftopology}
	Let $\cU$ be the set of all  subsets $U\subseteq G$ that satisfy the following property:
	if $g\in U$ and $(g_i)\cc{} g$, there exists an $N$ such that $g_i\in U$ for all $i\geq N$. Then $\cU$ form the open sets of a topology on $G$ called the \emph{topology of coarse convergence}.
\end{prop}
\begin{proof}
	It is obvious that $\emptyset,G\in \cU$.  Suppose $(g_i)\cc{} g$ and  $g\in \cup_{j\in J} U_j$, where  each $U_j\in\cU$. Then $g\in U_j$ for some $j$, so  there exists an $N$ such that $g_i\in U_j\subseteq \cup_{j\in J} U_j$ for all $i\geq N$. Therefore  $\cup_{j\in J} U_j\in \cU$.
	Now suppose $(g_i)\cc{} g$ and  $g\in  U_1\cap \dots \cap U_n$, where each $U_j\in \cU$.  Since $g\in U_j$ and $U_j\in \cU$, there exists an $N_j$ such that $g_i\in U_j$ for all $i\geq N_j$. Thus $g_i\in U_1\cap \dots \cap U_n$ for all $i\geq \max(N_1,\dots, N_n)$ and so $ U_1\cap \dots \cap U_n\in\cU$.
\end{proof}
We now define a basis of   the topology of coarse convergence.

\begin{lem}\label{lem:fund nbhd}
	For each $h\in G$ and  finite $F\subseteq X$,  we define \[U_{h,F}\coloneqq \{g\in G\mid \text{if $(g_i)\cc{} g$, $\exists N$ such that $g_i\close{F,B} h$ for all $i\geq N$}\}.\] Then $h\in U_{h,F}$ for all $h\in G$ and  $F\subseteq X$, and  $
		\cB\coloneqq \{U_{h,F}\mid h\in G, F\subseteq G \text{ is finite}\}$ is a countable basis of open sets in $G$. 
\end{lem}

\begin{proof}
	We first show each $U_{h,F}$ is open. Suppose for contradiction  that $ U_{h,F}$ is not open. Then there exists some $g\in U_{h,F}$ and a sequence $(g_i)\cc{} g$ such that $g_i\notin U_{h,F}$ for all $i$. 
	Since $g_i\notin U_{h,F}$, there is a sequence $(g^{i}_j)_{j=1}^\infty\cc{} g_i$ such that $g^{i}_j\nclose {F,B} h$ for all $j$. By Lemma \ref{lem:clim of clims}, there is a  sequence  $(k_i)\cc{} g$ such that   $k_i\nclose {F,B} h$ for all $i$,  contradicting our assumption that $g\in U_{h,F}$. Thus $U_{h,F}$ is open.

	We now show $h\in U_{h,F}$ for every finite $F\subseteq G$. Indeed, suppose  $(h_i)\cc{} h$. It follows from Lemma \ref{lem:uniform_cc_tame} and our choice of $B$  that $(h_i)\cc{B} h$, so there exists an $N$ such that $h_i\close{F,B} h$ for all $i\geq N$. Thus $h\in U_{h,F}$.
	
	We now show $\cB$ is a basis of $G$. Let $U\subseteq G$ be open and  let $h\in U$. We claim there is some finite $F_h\subseteq G$ such that $U_{h,F_h}\subseteq U$. It then follows that $U=\cup_{h\in U}U_{h,F_h}$, so that $\cB$ is indeed a basis of $G$. To prove the claim, we assume for contradiction that $U_{h,F}\not\subseteq U$ for any finite $F\subseteq G$. Since $X$ is locally finite, we enumerate $X=\{x_1,x_2,\dots\}$ and set $X_i\coloneqq \{x_1,\dots, x_i\}$ for each $i$. For each $i$, we can  thus choose $g_i\in U_{h,X_i}\setminus U$. By considering the constant sequence $(g_i)_{j=1}^\infty\cc{} g_i$, it follows that $g_i\close{X_i,B}h$. Thus   $(g_i)\cc{B} g$ but $g_i\notin U$ for all $i$, contradicting our assumption that $U$ is open. 
	
	All that remains is to show $\cB$ 
	is countable. We first note that $U_{h,F}=U_{h',F'}$ if $h|_F=h'|_{F'}$. Thus even though $G$ might be uncountable, since $X$ is countable (as it is locally finite) there are only countably many functions from finite subsets of $X$ to $X$, and hence countably many $U_{h,F}$ for $h\in G$ and finite $F\subseteq X$.
\end{proof}

\begin{proof}[Proof of Proposition \ref{prop:topcc}]
	 It follows from Proposition \ref{prop:deftopology} and Lemma \ref{lem:fund nbhd} that when equipped with the topology of coarse convergence, $G$ is a second-countable topological space and that $(g_i)\to g$ if and only if $(\phi(g_i))\cc{} g$. Moreover, Lemma \ref{lem:mult_inv_cont} implies multiplication and inverses are continuous.  Since topological groups are assumed to be Hausdorff, we  also verify $G$ is Hausdorff.
	 Indeed, as $G\qa X$ is coarsely faithful,   $g\not\close{} h$ for $g\neq h\in G$, so there exists some $x\in X$ such that $d(g\cdot x, h\cdot x)>2B$. It follows that $U_{g,\{x\}}$ and $U_{h,\{x\}}$ are disjoint neighbourhoods of $g$ and $h$, and so $G$ is Hausdorff. Thus $G$ is a  second-countable topological group. 
		
		As $G$ is a Hausdorff and second-countable topological group,  the Birkhoff--Kakutani theorem implies $G$ is metrisable. Thus a subspace of $G$ is  compact if and only if it is  sequentially compact. To show local compactness of $G$, we show every $g\in G$ has a sequentially compact neighbourhood $V_g$. Fix $x_0\in X$ and $g\in G$ and let  $V_g$ be the closure of $U_{g,\{x_0\}}$ in $G$. We claim that $V_g$ is sequentially compact.

		Indeed, $h\in V_g$ if and only if there is a sequence $(h_i)$ in $U_{g,\{x_0\}}$ such that $h_i\cc{} h$. For every $h_i\in U_{g,\{x_0\}}$, $d(g\cdot x_0,h_i\cdot x_0)\leq B$. Thus Lemma \ref{lem:uniform_cc_tame} implies that for every $h\in V_g$, $d(h\cdot x_0,g\cdot x_0)\leq 2B$. It thus follows from Lemma \ref{lem:coarse AA} that every sequence $(h_i)$ in $V_g$ has a coarsely convergent subsequence $(h_{n_i})$. Since $G\qa X$ is coarsely complete, there exists some $h\in G$ such that $h_{n_i}\cc{} h$. By Lemma \ref{lem:clim of clims}, $h$ is the coarse limit of some sequence in $U_{g,\{x_0\}}$, and hence $h\in V_g$. Thus for each $g\in G$,  $V_g$ is a sequentially compact set containing an open neighbourhood $U_{g,\{x_0\}}$ of $g$. Therefore $G$ is locally compact.
		
		To show the identity map $\id_G$ is a topological completion of $G\qa X$, we need only  show   $T\subseteq G$ has compact closure  if and only if it is bounded. Suppose $T\subseteq G$ has compact closure and fix $x_0\in X$. Since $\{U_{g,\{x_0\}}\mid g\in \overline{T}\}$ is an open cover of $\overline{T}$, there are finitely many $g_1,\dots, g_n\in \overline{T}$ such that $T\subseteq \cup_{i=1}^n U_{g_i,\{x_0\}}$. It follows that $\{g\cdot x_0\mid g\in T\}\subseteq \cup_{i=1}^n N_B({g_i\cdot x_0})$, and so $T$ is bounded. Conversely, suppose that $T$ is bounded, so that for some $x_0\in X$, $\{g\cdot x_0\mid g\in T\}$ is bounded. It follows from Lemma \ref{lem:coarse AA} and the coarse completeness of $G\qa X$ that any infinite sequence $(g_i)$ in $T$ has a subsequence $(g_{n_i})$ that coarsely converges to some $g\in G$. Since $(g_i)\cc{} g$, $(g_i)\to g$ and so $g\in \overline{T}$; therefore $\overline{T}$ is  compact.
		The final claim that $G$ has no nontrivial compact normal subgroups follows easily from Proposition \ref{prop:compact_normal} and the hypothesis that $G$ is coarsely faithful.
\end{proof}

We can now use the topology of coarse convergence to prove Theorem \ref{thm:tame_topcomp}:
\begin{proof}[Proof of Theorem \ref{thm:tame_topcomp}]
Let $X$ be a proper quasi-geodesic tame metric space and let $G\qa  X$ be a cobounded quasi-action. By Proposition \ref{prop:coarse_proper} (see also Remark \ref{rem:bdd_geom})  there is a locally finite metric space $Y$ such that  $G\qa X$ is quasi-conjugate to $G\qa Y$. Combining Propositions \ref{prop:coarsecomplete} and \ref{prop:topcc}, we deduce $G\qa Y$ has a topological completion $\rho:G\to \hat G$. 
\end{proof}

\subsection{Quasi-morphisms and quasi-actions on \texorpdfstring{$\bbR$}{R}}\label{sec:qmorph}
In this section, we give many examples of quasi-actions on $\bbR$ that do \emph{not} have topological completions. Whilst not directly used or needed elsewhere in this article, these examples serve to contextualise our main  results by demonstrating that, despite the plethora of spaces $X$ such that  every cobounded quasi-action on $X$ does  have a topological completion (e.g.\  Theorem \ref{thm:tame_topcomp} and Theorem \ref{thm:trichcodim2}), quasi-actions on more  general spaces need not have topological completions.   To do this, we make use of the well-known notion of a  quasi-morphism.

A \emph{quasi-morphism} is a map $f:G\to \bbR$ such that \[\sup_{g,h\in G}\lvert f(g)+f(h)-f(gh)\rvert\leq D \] for some $D$ called the \emph{defect} of $f$. A map $f:G\to \bbR$ is said to be \emph{bounded} if $\sup_{g\in G}\lvert f(g)\rvert <\infty$ and \emph{unbounded} otherwise. Two quasi-morphisms $f,f':G\to \bbR$ are said to be \emph{equivalent} if their difference is bounded. A quasi-morphism is said to be \emph{trivial} if it is equivalent to a homomorphism. A quasi-morphism $f:G\to \bbR$ is said to be \emph{homogeneous} if $f(g^n)=nf(g)$ for all $g\in G$ and $n\in \bbN$.  Every quasi-morphism is equivalent to a homogeneous quasi-morphism, which is necessarily unique; see for instance \cite[Proposition 2.10]{frigerio2017bounded}.

A quasi-morphism $f:G\to \bbR$  induces a quasi-action $\phi_f:G\to \bbR^\bbR$ via  the formula $\phi_f(g)(x)=f(g)+x$. We call  $\phi_f$ the \emph{quasi-action induced by $f$}. Since we are primarily interested in cobounded quasi-actions, we first determine when the induced quasi-action is cobounded:
\begin{lem}
	Let $f:G\to \bbR$ be a quasi-morphism. The induced quasi-action $G\qa \bbR$ is cobounded if and only if $f$ is unbounded.
\end{lem}
\begin{proof}
	If $f$ is bounded, then every quasi-orbit of the induced quasi-action $G\qa \bbR$ is bounded, so  the quasi-action is not cobounded. Conversely, suppose $f$ is unbounded. As noted above, there is a homogeneous quasi-morphism $\hat f:G\to \bbR$ and a constant $B$  such that $\lvert \hat f(g)-f(g)\rvert \leq B$ for all $g\in G$. Since $f$ is unbounded, $\hat f$ is not identically zero and so there is some $g\in G$ such that $\hat f(g)\neq 0$. For each $x\in \bbR$, we set $n_x\coloneqq \left \lfloor\frac{x}{\hat f(g)}\right\rfloor\in \bbZ$; thus, $\lvert x-\hat f(g)n_x\rvert\leq \lvert \hat f(g)\rvert$ for all $x\in \bbR$.
	It follows that \[\lvert f(g^{n_x})-x\rvert\leq \lvert \hat f(g^{n_x})-x\rvert +B=\lvert n_x\hat f(g)-x\rvert +B\leq \lvert \hat f(g)\rvert +B\] for all $x\in \bbR$, which implies the induced quasi-action $G\qa \bbR$ is $(\lvert \hat f(g)\rvert +B)$-cobounded.
\end{proof}

 We prove the following:
\begin{thm}\label{thm:quasimorphism}
	Let $f:G\to \bbR$ be an unbounded quasi-morphism. Then the induced cobounded quasi-action $G\qa \bbR$ has a topological completion if and only if $f$ is trivial.
\end{thm}
\begin{proof}
	First suppose that $f$ is trivial. Then there is a homomorphism $f':G\to \bbR$  and a constant $B$ such that $\lvert f'(g)-f(g)\rvert \leq B$. The quasi-action induced by $f'$ is an isometric action on $\bbR$. Thus the identity map  is a $(1,B)$-quasi-conjugacy from the quasi-action $G\qa \bbR$ induced by $f$ to the isometric action $G\curvearrowright \bbR$ induced by $f'$. It follows from Proposition \ref{prop:TCvsQC} that  the quasi-action $G\qa \bbR$ induced by $f$ has a topological completion.
	
	Conversely, suppose the quasi-action $\phi_f:G\to \bbR^\bbR$ has a topological completion $\rho:G\to \hat G$. Let $d$ be a geodesically adapted metric on $\hat G$, which exists by Proposition \ref{prop:topcompquasiconj}.   Proposition \ref{prop:topcompquasiconj} also ensures $\phi_f$ is $\rho$-quasi-conjugate to the regular representation $\hat G\curvearrowright (\hat G,d)$. In particular $(\hat G,d)$ is two-ended and so a result of Abels, refined by Cornulier \cite[Corollary 19.39]{cornulier2018quasi}, states that $\hat G$ acts geometrically on $\bbR$. Corollary \ref{cor:qconj_geomaction} thus implies $\phi_f$ is quasi-conjugate to an isometric action $\rho:G\to \Isom(\bbR)$. Since the original quasi-action $\phi_f$ consists of translations, $\im(\rho)$ consists of translations and so there is a  homomorphism $\psi:G\to \bbR$ such that $\rho(g)(x)=\psi(g)+x$.
	
	As a consequence of the preceding argument, there exist constants $K\geq 1$, $A\geq 0$ and a   $(K,A)$-quasi-isometry $h:\bbR\to\bbR$ such that \[\lvert h(\phi_f(g)(x))-\rho(g)(h(x))\rvert \leq A\] for all $g\in G$ and $x\in \bbR$. There is a constant $B$ and a homogeneous quasi-morphism $\hat f:G\to \bbR$ of defect $B$ such that  $\lvert \hat f(g)-f(g)\rvert \leq B$ for all $g\in G$. Therefore \[\lvert h(\hat f(g))-\psi(g)\rvert\leq \lvert h( \phi_f(g)(0))-\rho(g)(0)\rvert +KB+A\leq KB+2A\eqqcolon C \] for all $g\in G$.  Thus for all $g,k\in G$ and $n\in \bbN$, we have 
	\begin{align*}
		n\lvert \hat f(gk)-\hat f(g)-\hat f(k)\rvert &= \lvert \hat f((gk)^n)-\hat f(g^n)-\hat f(k^n)\rvert \\
		&\leq \lvert \hat f((gk)^n)-\hat f(g^nk^n)\rvert +B\\
		&\leq K \lvert h(\hat f((gk)^n))-h(\hat f(g^nk^n))\rvert +KA+B\\
		&\leq K\lvert \psi((gk)^n)-\psi(g^nk^n)\rvert +2KC +KA+B\\&=2KC +KA+B
	\end{align*} as $\hat f$ has defect $B$ and $\psi$ is a homomorphism. Since this holds for every $n$, it follows $\hat f$ is a homomorphism. As $\hat f$ is equivalent to $f$, we deduce $f$ is trivial.
\end{proof}
We conclude by noting that there are many examples of nontrivial quasi-morphisms. For instance, Brooks gave examples of non-trivial quasi-morphisms on the free group of rank two \cite{brooks1981some}; see also \cite{epsteinfujiwara1997second}. Coupled with Theorem \ref{thm:quasimorphism}, this  implies:
\begin{cor}
	There exist many cobounded quasi-actions on $\bbR$ with no  topological completion.
\end{cor}

\section{Quasi-actions on hyperbolic spaces}\label{sec:qaction_hyp}
In this section we  prove Theorem \ref{thm:trichotomyhyp_main},   classifying cobounded quasi-actions on proper hyperbolic metric spaces. We make use of   structure theorems for locally compact hyperbolic groups due Caprace--Cornulier--Monod--Tessera. We combine these with  results from Section \ref{sec:topcomp} to deduce  Theorem \ref{thm:trichotomyhyp_main}.

For definiteness, we say that a metric space $X$ is \emph{hyperbolic} if it is  a geodesic metric space and  there is a $\delta\geq 0$ such that for any triple of geodesic segments $[x,y]$, $[y,z]$ and $[x,z]$ in $X$, $[x,y]\subseteq N_{\delta}([y,z]\cup[x,z])$. This definition is a quasi-isometry invariant in the sense that if a geodesic metric space is quasi-isometric to a hyperbolic space, then it is also hyperbolic.

A   locally compact group $G$ is said to be \emph{hyperbolic} if it is  admits a compact generating set such that the associated Cayley graph is  hyperbolic.  Hyperbolicity of  $G$ is independent of the choice of compact generating set. Henceforth, locally compact hyperbolic groups are always equipped with such a metric.   Similarly, if $G$ is finitely generated relative to $H$, we say $G/H$ is hyperbolic if some (hence every) relative Cayley graph $\Gamma_{G,H}$ is hyperbolic.
If $G$ is a locally compact hyperbolic group,  it  has a well-defined boundary $\partial G$ on which $G$ acts  by homeomorphisms.  We remark that although the Cayley graph of $G$ is not proper (unless $G$ is discrete), $G$ acts geometrically on some proper geodesic hyperbolic metric space  \cite[Corollary 2.6]{caprace2015amenable}.

\begin{defn}
	A hyperbolic metric space  is said to be \emph{non-elementary} if its boundary has at least three points. In particular, if $G$ is a locally compact hyperbolic group, then $G$ is non-elementary if and only if  $\partial G$ has at least three points.
\end{defn}
We now recall some structural results about hyperbolic locally compact  groups. Firstly, we can characterise when such groups are amenable:
\begin{lem}[{\cite[Proposition 3.1 and Lemmas 5.2 and 5.3]{caprace2015amenable}}]\label{lem:amen_focal_char}
	Let $G$ be a hyperbolic locally compact  group. If $G$ is elementary, it is amenable. If $G$ is non-elementary, then it is amenable if and only if it fixes a point of $\partial G$. Moreover, a non-elementary amenable hyperbolic group is   not unimodular.
\end{lem}
We will make use of the following structural result concerning compact normal subgroups of hyperbolic groups. Recall that the \emph{amenable radical}  of a locally compact group is the maximal amenable normal subgroup.
\begin{lem}[{\cite[Lemmas 3.6, 5.1 and 5.2]{caprace2015amenable}}]\label{lem:gen_cpctamenrad}
	Let $G$ be a non-elementary locally compact hyperbolic group. Then $G$ has a unique maximal compact normal subgroup $K\vartriangleleft G$ which is precisely the kernel of the action of $G$ on $\partial G$. Moreover, if $G$ is non-amenable, then $K$ is the amenable radical of $G$.
\end{lem}

To understand non-amenable locally compact hyperbolic groups, we  require the following result of Caprace et al.  \cite[Proposition 5.10]{caprace2015amenable}, which is itself a refinement of a result of Mineyev--Monod--Shalom \cite{mineyev2004ideal}. The proof of Theorem \ref{thm:hypisom_general} makes use of the solution to Hilbert's fifth problem due to Gleason and Yamabe.
\begin{thm}\label{thm:hypisom_general}
	Let $G$ be  a non-amenable locally compact hyperbolic group and let $K\vartriangleleft G$ be the maximal compact normal subgroup. Then either:
	\begin{enumerate}
		\item $G/K$ is a virtually connected centre-free simple rank one Lie group, in which case $G$ acts geometrically  on a  rank one irreducible symmetric space of non-compact type;
		\item $G/K$ is totally disconnected, in which case $G$ acts geometrically on a locally finite graph.
	\end{enumerate}
\end{thm}

\emph{Millefeuille spaces} were introduced by Caprace et al.\  in order to  characterise amenable  hyperbolic groups \cite{caprace2015amenable}. We briefly summarise their construction, referring the reader  to \cite{caprace2015amenable} and \cite{cornulier2018quasi} for additional details. Let $X$  be a homogeneous simply connected negatively curved
Riemannian manifold. Then $X$ is isometric to a  solvable Lie group $N\rtimes_\alpha \bbR$ equipped with a left-invariant metric, where $N$ is a nilpotent Lie group and  $\alpha:N\rightarrow N$ is a contracting automorphism of $N$. Projecting to the $\bbR$ factor defines a Busemann function $b:X\rightarrow \bbR$. Fix $k\geq 1$, and let $T_{k+1}$ be a $(k+1)$-regular tree with a surjective Busemann function $b':T_{k+1}\rightarrow \bbR$ taking integer values on vertices. The associated millefeuille space is the topological space \[X[k]=\{(x,y)\mid b(x)=b'(y)\},\] which can naturally be equipped with a $\CAT(-1)$ metric. We note that $X[k]$ depends (up to isometry) only on $X$ and $k$, and not the Busemann functions $b$ and $b'$. The millefeuille space $X[k]$ is said to be \emph{pure} if neither $X$ nor $T_{k+1}$ is a line.
The following theorem is the main result  of   \cite{caprace2015amenable}. See also a summary of the work described by Cornulier \cite[\S 19.2.5]{cornulier2018quasi}. 
\begin{thm}[{\cite[Theorem 7.3]{caprace2015amenable}}]\label{thm:hypisom_amen}
	Let $G$  be a non-elementary amenable  locally compact hyperbolic group.  Then either:
	\begin{enumerate}
		\item $G$ acts geometrically on  a simply connected homogeneous  negatively curved
		Riemannian manifold;
	\item $G$ acts geometrically on  a pure millefeuille space;
	\item $G$ acts geometrically on a  regular  tree of finite valency greater than two.
	\end{enumerate}
\end{thm}

We can now use Theorems \ref{thm:hypisom_general} and \ref{thm:hypisom_amen} to prove:
\begin{thm}\label{thm:trichotomyhyp_main}
\trichotomyhyp
\end{thm}
  \begin{proof}
  Let $G\qa X$ be a cobounded quasi-action on a proper non-elementary hyperbolic space. By Corollary \ref{cor:morse_topcomp_exist}, $G\qa X$ has a topological completion $\rho:G\to \hat G$. By Proposition \ref{prop:topcompquasiconj}, $\hat G$ is a compactly generated locally compact group that is quasi-isometric to $X$, hence $\hat G$ is a non-elementary locally compact hyperbolic group. By Theorems \ref{thm:hypisom_general} and \ref{thm:hypisom_amen}, $\hat G$ acts geometrically on $Y$, where $Y$ is  either a homogeneous  negatively curved manifold,  a pure millefeuille space, or a  locally finite graph. Corollary \ref{cor:qconj_geomaction} thus implies $G\qa X$ is quasi-conjugate to $G\curvearrowright Y$.
  \end{proof}

As mentioned in the introduction, we can sharpen the conclusions of Theorem \ref{thm:trichotomyhyp_main} under an additional hypothesis determining whether  $G\qa X$  extends to an action $G\curvearrowright \partial X$ that either does or does not fix a point of $\partial X$. We thus have the following two corollaries:

\begin{cor}\label{cor:dichothyp_main}
	\dichothyp
\end{cor}
\begin{cor}\label{cor:trichot_hyp_fixpt}
	Let $X$ be a proper non-elementary  hyperbolic metric space. If $G\qa X$ is a cobounded quasi-action that fixes a point of $\partial X$, then one of the following holds:
\begin{enumerate}
	\item \textbf{(continuous case)} $G\qa X$ is quasi-conjugate to an isometric action $G\curvearrowright Y$, where $Y$ is a simply connected homogeneous  negatively curved
	Riemannian manifold.
	\item \textbf{(mixed case)} $G\qa X$ is quasi-conjugate to an isometric action  $G\curvearrowright Y$, where $Y$ is a pure millefeuille space.
	\item \textbf{(discrete case)} $G\qa X$ is quasi-conjugate to an isometric action $G\curvearrowright Y$, where $Y$ is a regular  tree of finite valency greater than two.
\end{enumerate}	
\end{cor}
\begin{proof}[Proofs of Corollaries \ref{cor:dichothyp_main} and \ref{cor:trichot_hyp_fixpt}]
	As in the proof of Theorem \ref{thm:trichotomyhyp_main}, the quasi-action $G\qa X$ has a topological completion $\rho:G\to\hat G$. We claim $\hat G$ is amenable if and only if the quasi-action $G\qa X$ fixes a point of $\partial X$. Indeed, Proposition \ref{prop:topcompquasiconj} ensures there is a $\rho$-quasi-conjugacy $X\to \hat G$. This induces a homeomorphism $\partial X\to \partial \hat G$ that conjugates the $G$ action on $\partial X$ to a $\rho(G)$  action on $\partial \hat G$. Since $\hat G$ acts continuously on $\partial \hat G$ (see for instance   \cite[Lemma 3.4]{furman2001mostowmargulis}) and $\rho(G)\leq \hat G$ is dense  in $G$, it follows that $\rho(G)$ fixes a point of $\partial \hat G$ if and only if $\hat G$ does. The claim now follows from Lemma \ref{lem:amen_focal_char}. 
	We now proceed as in the proof of Theorem \ref{thm:trichotomyhyp_main}, noting that when  $G$ does not fix a point of $\partial X$  we can apply  Theorem  \ref{thm:hypisom_general}, and  when $G$ does fix   a point of $\partial X$ we can apply Theorem  \ref{thm:hypisom_amen}.
\end{proof}
We will use the  following property of pure millefeuille spaces:
\begin{prop}[{\cite[Corollary 19.22]{cornulier2018quasi}}]\label{prop:millefeuille}
A  pure millefeuille space is not quasi-isometric to any vertex-transitive locally finite graph. In particular, a pure millefeuille space is not quasi-isometric to a finitely generated group.
\end{prop}
 We recall that a metric space $X$ is discretisable if every cobounded quasi-action on $X$ is discretisable.   We can thus use Theorem \ref{thm:trichotomyhyp_main} to deduce:
\begin{cor}\label{cor:dischyp_qitohom}
	A  finitely generated hyperbolic group $\Gamma$ is discretisable if and only if it is not quasi-isometric to a simply connected negatively curved homogeneous space. In particular, if $\partial \Gamma$ is not homeomorphic to a sphere (of any dimension), then it is discretisable.
\end{cor}
The following  question, raised by Cornulier \cite[Conjecture 19.99]{cornulier2018quasi}, is relevant  to refining our description of  non-discretisable finitely generated hyperbolic groups:  Do there exist negatively curved homogeneous spaces other than symmetric spaces that are quasi-isometric to finitely generated groups? Since the isometry groups of such manifolds are not unimodular, they cannot contain lattices. In particular, a finitely generated group cannot act geometrically on a negatively curved homogeneous space other than a symmetric space. 
A positive answer to this conjecture is implied by the Pointed Sphere Conjecture \cite[Corollary 19.104]{cornulier2018quasi}, which asserts that if $X$ is a negatively curved homogeneous spaces other than a symmetric space, there is a distinguished point in $\partial X$ that is fixed by every self quasi-isometry  of $X$. Xie, Carrasco and others have proven special cases of this conjecture   \cite{xie2014large},\cite{piaggio2017Orlicz}. We refer to \cite{cornulier2018quasi} for more details.

\section{Products of hyperbolic groups}\label{sec:hyp_prod}
In this section we study the class $\cC$ of finitely generated groups quasi-isometric to products of non-elementary proper hyperbolic metric spaces. Theorem \ref{thm:hyp_productmain}, the key result of this section, shows any  group in $\cC$ acts geometrically on a product of symmetric spaces and locally finite graphs. This implies  Corollary \ref{cor:hyp_product_lattice}, which says  a group in $\cC$ is virtually isomorphic to  a lattice  in a product of locally compact hyperbolic groups with trivial amenable radical. In Theorem \ref{thm:qiprod_comm}, we show that in non-exceptional cases, groups in $\cC$  contain commensurated subgroups, and split as direct products precisely when these commensurated subgroups are weakly separable. 
 We prove Theorems \ref{thm:mixedlattice_main} and \ref{thm:splitlinear_main}, which show that under the presence of additional algebraic hypotheses, namely residual finiteness and linearity, groups in $\cC$ are either arithmetic lattices in semisimple algebraic groups, or  split as direct products of hyperbolic groups.  Theorem \ref{thm:boundary_rigidity} proves that $\CAT(0)$ groups in $\cC$ are boundary rigid.

\subsection{A structure theorem for groups quasi-isometric to products of hyperbolic spaces}
An isometric action $\Gamma\curvearrowright \Pi_{i=1}^nY_i$ is said to  \emph{preserve the product structure} if every isometry splits as a product of isometries, possibly permuting the factors. 
	If $\Gamma\curvearrowright \Pi_{i=1}^nY_i$ is an isometric action preserving the product structure, we let $\Gamma^*\leq \Gamma$ denote the maximal  finite index subgroup that does not permute factors.
 We note that $\Gamma^*$ acts isometrically on each $Y_i$ and hence induces an action on the boundary $\partial Y_i$ by homeomorphisms.
\begin{thm}\label{thm:hyp_productmain}
	Let $\Gamma$ be a finitely generated group quasi-isometric to $\Pi_{i=1}^nX_i$, where each $X_i$ is a cocompact proper  geodesic non-elementary hyperbolic metric space. Then $\Gamma$ acts geometrically on $\Pi_{i=1}^nY_i$, preserving the product structure, where each $Y_i$ is quasi-isometric to $X_i$ and is either a rank one symmetric space  or a locally finite graph.
	Moreover, $\Gamma^*$ does not fix a point of $\partial Y_i$ for any $i$.
\end{thm}
We first prove a weaker version of Theorem \ref{thm:hyp_productmain}:
\begin{lem}\label{lem:hypprod_intermediate}
	Let $\Gamma$ be a finitely generated group quasi-isometric to $\Pi_{i=1}^nX_i$, where each $X_i$ is a cocompact proper geodesic non-elementary hyperbolic metric space. Then $\Gamma$ acts geometrically on $\Pi_{i=1}^nY_i$, preserving the product structure, where each $Y_i$ is quasi-isometric to $X_i$ and is either a negatively curved homogeneous space, a pure millefeuille space, or a locally finite graph. Moreover, if $\Gamma^*$  fixes a point of  $\partial Y_i$ for some locally finite graph $Y_i$, then $Y_i$ is a regular  tree of finite valency greater than two.
\end{lem}
\begin{proof}
	Since $\Gamma$ is quasi-isometric to $\Pi_{i=1}^nX_i$, it has a cobounded quasi-action on $\Pi_{i=1}^nX_i$. Noting each $X_i$ is of coarse type I, we can combine Proposition \ref{prop:qaction_prod} and Theorem \ref{thm:trichotomyhyp_main} to deduce    $\Gamma\qa \Pi_{i=1}^nX_i$ is quasi-conjugate to an isometric action  $\Gamma\curvearrowright\Pi_{i=1}^nY_i$, where each $Y_i$ is one of the three desired types of space. 
	 Proposition \ref{prop:qaction_prod} and Theorem \ref{thm:hypisom_amen} ensure that  if the  subgroup $\Gamma^*$  fixes a point of some $\partial Y_i$, then $Y_i$ can be assumed to be either a homogeneous space, a pure millefeuille space, or a regular  tree of finite valency greater than two. Without loss of generality, we may assume that if $\Gamma^*$ fixes a point of some $\partial Y_i$ and $Y_i$ is a graph, then $Y_i$ is  such a  tree.
\end{proof}
Theorem \ref{thm:hyp_productmain} will follow from Lemma \ref{lem:hypprod_intermediate} and the following proposition, which is a generalisation of the fact that a finitely generated non-elementary hyperbolic group cannot fix a point of its boundary.
\begin{prop}\label{prop:hyp_fixptbdry}
	Let $\Gamma$ be a finitely generated group acting geometrically on $\Pi_{i=1}^nY_i$ such that the conclusions of Lemma \ref{lem:hypprod_intermediate} are satisfied. Then $\Gamma^*$ does not fix a point of $\partial Y_i$  for any $i$.
\end{prop}
A $\CAT(0)$ space $X$ admitting an isometric action of $G$  is said to be \emph{$G$-minimal} if there is no non-empty proper $G$-invariant closed convex subset of $X$. A proper $\CAT(0)$ space $X$ has \emph{no Euclidean de Rham factor} if it admits no isometric splitting as $X=\bbR^n\times X'$  for  $n>0$.
To prove Proposition \ref{prop:hyp_fixptbdry}, we make use of the following result of Caprace--Monod:
\begin{thm}[{\cite[Theorems 3.11 and 3.14]{capracemonod2009isometry}}]\label{thm:cat0_bdry}
	Suppose a finitely generated group acts geometrically on a proper $\CAT(0)$ space $X$. 
Then $X$ contains a canonical closed convex non-empty $\Isom(X)$-invariant
	$\Isom(X)$-minimal subset $X'\subseteq X$ such that $\Isom(X')$ fixes no point of $\partial X'=\partial X$.
	
	If $\Gamma$ is a finitely generated group  acting geometrically on $X$, then $X'$ is $\Gamma$-minimal.
Moreover, if  $X'$ has no Euclidean de Rham factor,    then $\Gamma$ does not fix a point of $\partial X'$. 
\end{thm}

Throughout this section, we use the following lemma about groups acting on products of spaces such that at least one factor is a locally finite graph. It is a geometric analogue of Lemma \ref{lem:lattice_int_open}. 
\begin{lem}\label{lem:product_graph_stab}
	Let $\Gamma$ be a finitely generated  group acting cocompactly by isometries on a proper metric space $X=X_1\times X_2$, preserving the product structure and without permuting factors. Suppose $X_2$ is a locally finite graph and $\Gamma$ acts on $X_2$ by graph automorphisms. If  $v$ is a vertex of $X_2$, then the subgroup $\Gamma_1\leq \Gamma$ stabilising $X_1\times \{v\}$ acts cocompactly  on $X_1\times \{v\}$ and is commensurated in $\Gamma$.
	In particular, if $\Gamma$ acts geometrically on $X$, then $\Gamma_1$ acts geometrically on $X_1\times \{v\}$.
\end{lem}
\begin{proof}
	 Let $x_0\in X_1$. Since $\Gamma$ acts cocompactly on $X_1\times X_2$, there is an $r$ such that the $\Gamma$-orbits of $K\coloneqq N_{r}(x_0,v)$ cover $X$.  For each $x\in X_1$, there is a $g_x\in \Gamma$ such that $g_x(x,v)\in K$, and so  $d_{X_1}(g_x x, x_0),d_{X_2}(g_x v, v)\leq r$. As $X_2$ is a locally finite graph, the set $T=\{g_x v\mid x\in X_1\}$ is finite, so we may pick $x_1,\dots, x_n\in X_1$ such that $T\coloneqq \{g_{x_1} v,\dots,g_{x_n} v \}$.	Set $L\coloneqq \cup_{i=1}^nN_{2r}(x_i)\subseteq X_1$. For each $x\in X_1$, we have $g_xv=g_{x_i}v$ for some $i$, and so $g_x^{-1}g_{x_i}\in \Gamma_1$. Since  \[d(g_x^{-1}g_{x_i}x_i,x)=d(g_{x_i}x_i,g_xx)\leq d(g_{x_i}x_i,x_0)+d(x_0,g_xx)\leq 2r,\] we deduce $X_1$ is covered by $\Gamma_1$-orbits of the compact set $L$. Since $\Gamma$ projects to a cocompact action on $X_2$ and $\Gamma_1$ is the stabiliser of a vertex, it follows $\Gamma_1$ is commensurated.
\end{proof}
We are now ready to prove Proposition \ref{prop:hyp_fixptbdry}:
\begin{proof}[Proof of Proposition \ref{prop:hyp_fixptbdry}]
	By rearranging the factors of $\Pi_{i=1}^nY_i$, we may assume that $i>r$ if and only if the following both hold:
	\begin{itemize}
		\item $Y_i$ is a locally finite graph;
		\item  $\Gamma^*$ does not fix a point of  $\partial Y_i$.
	\end{itemize}  It remains to show $\Gamma^*$ does not fix a point of $\partial Y_i$ for any $i\leq r$.

Applying Lemma \ref{lem:product_graph_stab} iteratively, we deduce $\Gamma^*$ has a subgroup $\Gamma'$ acting geometrically on $Y'\coloneqq \Pi_{i=1}^rY_i$. By the construction of  $Y'$, each $Y_i$ for $i\leq r$ is either a negatively curved homogeneous space, a pure millefeuille space or a regular  tree of finite valency greater than two; in particular  each $Y_i$ is CAT($-1$). Thus $Y'$ is a  CAT(0) space with isometry group acting cocompactly. Since $Y'$ is geodesically complete, it is $\Gamma'$-minimal \cite[Lemma 3.13]{capracemonod2009strucure}. Since the de Rham decomposition of $Y'$ is essentially unique --- see \cite[Theorem 1.9]{capracemonod2009strucure} --- it follows that $Y'$ has no Euclidean de Rham factor.   Theorem  \ref{thm:cat0_bdry} ensures $\Gamma'$ does not fix a point on the visual boundary $\partial Y'$. As $\partial Y'$ is the join   $\partial Y_1*\dots* \partial Y_r$ and $\Gamma'$ preserves the product structure without permuting factors,  $\Gamma'$ cannot fix a point of  $\partial Y_i$ for any $i\leq r$.  
\end{proof}

\begin{proof}[Proof of Theorem \ref{thm:hyp_productmain}]
	We first note that  $\Gamma$ acts geometrically on $\Pi_{i=1}^nY_i$ as in Lemma \ref{lem:hypprod_intermediate}. By Proposition \ref{prop:hyp_fixptbdry}, $\Gamma^*$ does not fix a point of $\partial Y_i$ for any $i$. Thus no $Y_i$ can be a pure Millefeuille space or a non-symmetric homogeneous space, since the isometry groups of such spaces have global fixed points on $\partial Y_i$. Thus each $Y_i$ is either a locally finite graph or a rank one symmetric space.
\end{proof}

We also have the following algebraic consequence of Theorem \ref{thm:hyp_productmain}. 
\begin{cor}\label{cor:hyp_product_lattice}
	Let $\Gamma$ be a finitely generated group quasi-isometric to $\Pi_{i=1}^nX_i$, where each $X_i$ is a cocompact  non-elementary proper hyperbolic metric space. There is a finite index subgroup $\Gamma^*\leq\Gamma$ and a finite normal subgroup $F\vartriangleleft \Gamma^*$ such that $\Gamma^*/F$ has trivial amenable radical and   is a uniform lattice in $G_1\times \dots \times G_n$, whose projection to each factor is dense.  Moreover, each $G_i$ is a  hyperbolic locally compact group with trivial amenable radical that is quasi-isometric to $X_i$, and is either a virtually connected simple rank-one Lie group or is totally disconnected.
\end{cor}
\begin{proof}
	By Theorem \ref{thm:hyp_productmain}, $\Gamma$ acts geometrically on $\Pi_{i=1}^nY_i$,  where each $Y_i$ is quasi-isometric to $X_i$ and is either a rank one symmetric space  or a locally finite graph.  Moreover, the finite index subgroup $\Gamma^*$ acts geometrically on $\Pi_{i=1}^nY_i$, preserving the product structure, not permuting the factors, and not fixing a point of $\partial Y_i$ for any $i$.  This  geometric action induces a homomorphism $\rho:\Gamma^*\to \Isom(Y_1)\times \dots \times \Isom(Y_n)$. Since this action is geometric, $\ker(\rho)$ is finite and $\rho(\Gamma^*)$ is a uniform lattice.  
	 Letting $H_i$ be the closure of the projection of $\rho(\Gamma^*)$ to $\Isom(Y_i)$, we have a map  $\phi:\Gamma^*\rightarrow H_1\times \dots \times H_n$ with finite kernel and whose image is a uniform lattice with dense projections to each factor.

	 Each $H_i$ is a closed cocompact subgroup of $\Isom(Y_i)$, hence  is  a non-elementary  locally compact hyperbolic group that is  quasi-isometric to $X_i$. Since  $\Gamma^*$ does not fix a point on the boundary of  $Y_i$, Lemma \ref{lem:amen_focal_char} ensures  $H_i$ is non-amenable. Lemma \ref{lem:gen_cpctamenrad} thus implies $H_i$   has compact amenable radical $K_i\vartriangleleft H_i$ and so the quotient  $G_i\coloneqq H_i/K_i$ has trivial amenable radical and is also a  locally compact hyperbolic group quasi-isometric to $X_i$. By Theorem \ref{thm:hypisom_general}, $G_i$ is either a simple rank one Lie group or is totally disconnected.  Composing $\phi$ with the quotient map in each factor gives a map $\psi:\Gamma'\rightarrow G_1\times \dots \times  G_n$, with finite kernel and whose image is a uniform lattice.  
\end{proof}

\subsection{Algebraic splitting for groups quasi-isometric to products of hyperbolic spaces}
We use Corollary \ref{cor:hyp_product_lattice} to deduce further structural results about groups quasi-isometric to products of hyperbolic groups.  We first establish some general facts about products of hyperbolic groups with trivial amenable radical.

\begin{lem}\label{lem:triv_centraliser}
	Suppose $G$ is a locally compact hyperbolic group  with trivial amenable radical.  If $H\leq G$ is  nontrivial and either  normal or cocompact (but not necessarily closed), then its  centraliser $C_G(H)$ is trivial.
\end{lem}

\begin{proof}
	By \cite[Corollary 2.6]{caprace2015amenable}, $G$ acts geometrically on a proper geodesic hyperbolic metric space $X$. By Corollary \ref{cor:qconj_geomaction}, there is a  quasi-conjugacy $G\to X$ that  induces a homeomorphism $\partial G\cong \partial X$ conjugating  the action of $G$ on $\partial G$ to the action of $G$ on $\partial X$. We assume $G$ is nontrivial otherwise there is nothing to prove. Thus $G$ is  non-amenable, and Lemma \ref{lem:amen_focal_char} ensures $X$ is non-elementary and   $G$ does not fix a point of $\partial X$.  
	Recall an element $g\in G$ is \emph{hyperbolic} if its limit set consists of precisely two distinct points $\{g^{+\infty},g^{-\infty}\}\subseteq \partial X$, see e.g.\ \cite{gromov1987hyperbolic} for details. The centraliser of a hyperbolic element $g$ fixes  $g^{+\infty}$ and $g^{-\infty}$, and we use this observation to prove the lemma.	 
	
	First suppose $H$ is a cocompact subgroup.  As $G$ is  non-amenable and $H\leq G$ is cocompact,  the set of endpoints of hyperbolic elements of $H$ is dense in $\partial X$ \cite[8.2.D, 8.2.G]{gromov1987hyperbolic}). Therefore $C_{G}(H)$ fixes $\partial X$ pointwise. Since $G$ has trivial amenable radical,  Lemma \ref{lem:gen_cpctamenrad} ensures  the set of  elements of $G$ fixing $\partial X$ is trivial, hence $C_{G}(H)$ is trivial.
	
	Now suppose $H$ is a nontrivial normal subgroup of $G$. As $H$ is normal,  any $H$-invariant subset of $\partial X$ is also $G$-invariant. Since $G$ is non-amenable, $H$ cannot stabilise a finite subset of $\partial X$. The  classification of groups acting  isometrically on hyperbolic space implies $H$  contains  hyperbolic elements $g_1$ and $g_2$ whose limit sets are disjoint.  It follows that  $C_{G}(H)$ fixes at least four points in the boundary, hence has   bounded orbits. As  $C_{G}(H)$ is a closed subgroup with bounded orbits, it is  compact. As $H$ is normal,  so is  $C_{G}(H)$, and hence is trivial as $G$ has trivial amenable radical.
\end{proof}

\begin{cor}\label{cor:triv_cen}
	Suppose $G=G_1\times \dots\times  G_n$, where each $G_i$ is a locally compact hyperbolic group  with trivial amenable radical.  If $H\leq G$ is either normal or  cocompact, then its centraliser $C_G(H)$ is trivial.
\end{cor}
\begin{proof}
	Normal and cocompact subgroups are both preserved by continuous  epimorphisms. Therefore,  for each $1\leq i\leq n$, the projection of  $H$ to  $G_i$  has trivial centraliser in $G_i$ by Lemma \ref{lem:triv_centraliser}. Thus $C_G(H)$ projects to the trivial subgroup in each factor, so is also trivial.
\end{proof}

The following result gives a sufficient criterion for a group quasi-isometric to a product of hyperbolic spaces to split as a direct product modulo a finite normal subgroup. Caprace--Monod proved an analogous result for CAT(0) groups  \cite[Theorem 4.2]{capracemonod2009isometry}.
\begin{prop}\label{prop:irred}
	Let $G_1$ and $G_2$ each be the product of finitely many (possibly one)  hyperbolic locally compact groups with trivial amenable radical, and let $\rho:\Gamma\to  G_1\times G_2$ be a homomorphism with finite kernel and whose image is a uniform lattice.  
	If the projection of $\rho(\Gamma)$ to one of $G_1$ or $G_2$ is discrete, then the projection onto the other factor is also discrete. 
		When this is the case,   $\Gamma/\ker(\rho)$  virtually splits  as $\Gamma_1\times \Gamma_2$, where $\Gamma_i$ is  a uniform lattice in $G_i$ for $i=1,2$.
\end{prop}
\begin{proof}
	Set $\Lambda\coloneqq \rho(\Gamma)$ and let $p_i:G\rightarrow G_i$ be the projection for $i=1,2$. Let $H_i\leq G_i$ be the closure of $p_i(\Lambda)$. Note $\Lambda$ is a uniform lattice in $H_1\times H_2$. Suppose $H_1$ is discrete. Then $\{1\}\times H_2$ is open in $H_1\times H_2$,  and so Lemma \ref{lem:lattice_int_open} ensures $\Lambda\cap (\{1\}\times H_2)$ is a uniform lattice in $\{1\}\times H_2$.  Therefore, $\Lambda_2\coloneqq p_2(\Lambda\cap (\{1\}\times H_2))$ is a uniform lattice  of $H_2$ normalised by the dense subgroup $p_2(\Lambda)$. As $\Lambda_2$ is a cocompact subgroup of $G_2$,  Corollary \ref{cor:triv_cen} ensures its centraliser $C_{H_2}(\Lambda_2)$ is trivial. Since Lemma \ref{lem:open_cen}  says $C_{H_2}(\Lambda_2)$ is open,  $H_2$ is discrete.   As $\Lambda$ is a uniform lattice of the discrete group $H_1\times H_2$,  it is of finite index and so $\Lambda'\coloneqq \Lambda'_1 \times \Lambda'_2$ is a finite index subgroup of $\Lambda$, where  $\Lambda'_1\coloneqq p_1(\Lambda\cap (H_1\times \{1\}))$ and $\Lambda'_2\coloneqq p_2(\Lambda\cap (\{1\}\times H_2))$. 
\end{proof}
We now present a refinement of Theorem \ref{thm:hyp_productmain}, algebraically characterising when such a group splits as a direct product. We assume  at least one factor is not quasi-isometric to a symmetric space; the case where all the factors are quasi-isometric to symmetric spaces is considered by Kleiner--Leeb \cite{kleinerleeb97rigidity}.
\begin{thm}\label{thm:qiprod_comm}
	Let $\Gamma$ be a finitely generated group quasi-isometric to $X\times Y$, where:
	\begin{itemize}
		\item $X$ is a product of finitely many cocompact non-elementary proper hyperbolic metric spaces;
		\item $Y$ is a cocompact non-elementary proper hyperbolic metric space not quasi-isometric to a rank one symmetric space.
	\end{itemize}
	Then there exists a finitely generated subgroup $\Lambda\leq \Gamma$ such that:
	\begin{enumerate}
		\item $\Lambda$ is  commensurated by a finite index subgroup $\Gamma^*\leq \Gamma$;
		\item $\Lambda$ is quasi-isometric to $X$;
		\item the quotient space $\Gamma^*/\Lambda$ is quasi-isometric to $Y$;
		\item $\Lambda$ is weakly separable in $\Gamma^*$ if and only if $\Gamma^*$ contains a finite normal subgroup $F$, contained in $\Lambda$,  such that $\Gamma^*/F$ is virtually isomorphic to  $\Gamma_X\times \Gamma_Y$, where $\Gamma_X$ is commensurable to $\Lambda/F$ and $\Gamma_Y$ is a finitely generated group quasi-isometric to $Y$.
	\end{enumerate} 
\end{thm}
\begin{proof}
	Suppose $X=\Pi_{i=1}^nX_i$ and $Y=X_{n+1}$. Applying Theorem \ref{thm:hyp_productmain}, $\Gamma^*$ acts geometrically on $\Pi_{i=1}^{n+1}Y_i$ preserving the product structure, where each $Y_i$ is quasi-isometric to $X_i$ and is either a rank one symmetric space or a locally finite graph. Since $Y$ is not quasi-isometric to a symmetric space,  $Y_{n+1}$ is a locally finite graph.
	Set $X'=\Pi_{i=1}^nY_i$ and $Y'\coloneqq Y_{n+1}$. Then $X'$ is quasi-isometric to $X$, $Y'$ is quasi-isometric to $Y$, and $\Gamma^*$ acts geometrically on $X'\times Y'$, preserving the product structure and not permuting the factors.  
	
	Let $v\in V(Y')$ be an arbitrary vertex. Applying Lemma \ref{lem:product_graph_stab}, we deduce $\Lambda\coloneqq \stab_{\Gamma^*}(X'\times \{v\})$ is  commensurated by $\Gamma^*$ and  acts geometrically on $X'\times \{v\}$; in particular $\Lambda$  is quasi-isometric to $X$. Moreover, since $\Lambda$ is a coarse stabiliser of the natural action $\Gamma^*\curvearrowright Y'$ by Example \ref{exmp:stabvertex,cstab}, the quotient space $\Gamma^*/\Lambda$ is quasi-isometric to $Y'$  by Proposition \ref{prop:cstab homspace}. Thus $\Gamma^*/\Lambda$ is quasi-isometric to $Y$.
	
	Arguing as in the proof of  Corollary \ref{cor:hyp_product_lattice}, the action of $\Gamma^*$ on  $\Pi_{i=1}^{n+1}Y_i$ induces a homomorphism  $\Gamma^*\to G_1\times \dots\times G_{n+1}$, where $G_i$ is obtained by taking the closure of the projection of $\Gamma^*$ to $\Isom(Y_i)$ and quotienting out by the maximal compact normal subgroup. In particular, each  $G_i$ is a locally compact hyperbolic group with trivial amenable radical. By Remark \ref{rem:topcomp,normal} and Proposition \ref{prop:TCvsQC}, $G_{i}$ is a topological completion of $\Gamma^*\curvearrowright Y_i$. 
	Let $H$ be the closure of the projection of $\Gamma^*$ to $G_1\times\dots\times G_n$. Therefore we have a homomorphism $\rho:\Gamma^*\to H\times G_{n+1}$ with finite kernel whose projection to each factor is dense.
	
	Suppose $\Lambda$ is weakly separable. By Lemma \ref{lem:weaksep}, $\Lambda$ is commensurable to a normal subgroup. Proposition \ref{prop:disc_topcomp} implies  every topological completion of $\Gamma^*\qa Y'$ is compact-by-discrete.  Since $G_{n+1}$ is a topological completion of  $\Gamma^*\qa Y'$ with trivial amenable radical,  $G_{n+1}$ is discrete. Proposition \ref{prop:irred} tells us that $H$ is also discrete and $\Gamma^*/\ker(\rho)$ has a finite index subgroup that splits as $\Gamma_X\times \Gamma_Y$, where $\Gamma_X$ and $\Gamma_Y$ are commensurable to  $H$ and $G_{n+1}$.  Thus $\Gamma_X$ is quasi-isometric to $X$ and $\Gamma_Y$ is quasi-isometric to $Y$. Since $\Lambda/\ker(\rho)$  projects  to a finite index subgroup of $H$ and to a finite subgroup of $G_{n+1}$,  $\Lambda/\ker(\rho)$ is commensurable to $\Gamma_X\times\{1\}\cong \Gamma_X$ as required. 
	
	Conversely, suppose $\Gamma^*$ contains a finite normal subgroup $F$, contained in $\Lambda$, such that $\Gamma^*/F$ virtually splits as  $\Gamma_X\times \Gamma_Y$ with $\Gamma_X$ commensurable to $\Lambda/F$. Since $\Gamma_X$ is normalised by a finite index subgroup of $\Gamma^*/F$,  Lemma \ref{lem:weaksep} implies $\Lambda/F$ is weakly separable in $\Gamma^*/F$. Thus $\Lambda$ is weakly separable in $\Gamma^*$.
\end{proof}

We now prove  two theorems that show groups quasi-isometric to products of hyperbolic spaces split as direct products under additional algebraic hypotheses, namely residual finiteness and linearity  respectively. We use arithmeticity results of Caprace--Monod \cite[\S 5]{capracemonod2009isometry}.

\begin{thm}\label{thm:mixedlattice_main}
	Let $\Gamma$ be a  finitely generated residually finite group quasi-isometric to $X_1\times X_2$, where  $X_1$ and $X_2$ are cocompact proper  non-elementary hyperbolic metric spaces. If $X_1$ is quasi-isometric to an irreducible rank one symmetric space and $X_2$ is not, then  either:
	\begin{enumerate}
		\item $\Gamma$ is virtually an arithmetic lattice in a semisimple algebraic group;
		\item \label{item:mixedlatice_2} $\Gamma$ virtually splits as $\Gamma_1\times \Gamma_2$, where $\Gamma_i$ is quasi-isometric to $X_i$.
	\end{enumerate}
	Moreover, if $X_2$ is not quasi-isometric to a tree, then (\ref{item:mixedlatice_2}) holds.
\end{thm}
Before beginning the proof, we recall the notion of an irreducible lattice.
\begin{defn}
	A  lattice $\Gamma\leq G_1\times \dots \times  G_n$ in a product of locally compact groups  is \emph{irreducible}  if:
\begin{enumerate}
	\item 	  the projection of $\Gamma$ to each $G_i$ is dense;
	\item 	  the projection of $\Gamma$ to any proper subproduct of $G_1\times \dots \times  G_n$ is non-discrete.
\end{enumerate}
\end{defn}
We caution the reader that outside the world of semisimple algebraic groups, there are many different notions of irreducibility that are not in general equivalent; see the discussion in \cite[\S 2.B]{capracemonod2012lattice}. 
To prove Theorem \ref{thm:mixedlattice_main}, we will apply the following theorem of  Caprace--Monod; see also \cite[Theorem 4.6]{baderfurmansauer2020lattice}.
\begin{thm}[{\cite[Theorem 5.18]{capracemonod2009isometry}}]\label{thm:abstract_mixedlattice}
	Let $\Gamma\leq S\times D$ be a finitely generated irreducible lattice, where $S$ is connected semisimple Lie group and $D$ is  totally disconnected with trivial amenable radical. Assume the projection $\Gamma\rightarrow S$ is injective. Then $D$ is a semisimple algebraic group and the image of $\Gamma$ in $S\times D$ is an arithmetic lattice.
\end{thm}
Note that Caprace--Monod prove a more general theorem where the projection $\Gamma\rightarrow S$ is not necessarily injective, and $D$ need not have trivial amenable radical \cite[Theorem 5.18]{capracemonod2009isometry}. This more general statement will not be needed here. In order to show injectivity of the projection to the Lie group factor, we appeal to the following theorem of Caprace--Kropholler--Reid--Wesolek.
\begin{lem}[{\cite[Corollary 33]{caprace_kropholler_reid_wesolek_2020}}]\label{lem:inj_proj}
	Let $\Gamma$ be a lattice in the product $G_1\times G_2$ of two locally compact
	groups. Assume  $G_2$ is totally disconnected and non-discrete and that every infinite
	closed normal subgroup of $G_2$ has trivial centraliser in $G_2$. Suppose further that $\Gamma$ is finitely
	generated and that the canonical projection $\Gamma\rightarrow G_2$ has a dense image. If $\Gamma$ is residually
	finite, then the projection $\Gamma \rightarrow G_1$ is injective.
\end{lem}

\begin{proof}[Proof of Theorem \ref{thm:mixedlattice_main}]
	Let $\Gamma$ be a finitely generated residually finite group  quasi-isometric to $X_1\times X_2$.  Corollary \ref{cor:hyp_product_lattice} says that there is a finite normal subgroup $F\vartriangleleft \Gamma$ such that $\Gamma/F$  is a uniform lattice in $G_1\times G_2$, where $G_1$ and $G_2$ satisfy  the conclusions of Corollary \ref{cor:hyp_product_lattice}.  In particular, $X_1$ and $X_2$ are quasi-isometric to $G_1$ and $G_2$ respectively. Since $\Gamma$ is residually finite, we can replace it with a finite index subgroup if necessary so that $\Gamma$ is a uniform lattice in $G_1\times G_2$. If either $G_1$ or $G_2$ is discrete, Proposition \ref{prop:irred} implies  $\Gamma$ virtually splits as a product $\Gamma_1\times \Gamma_2$, where $
	\Gamma_i$ is quasi-isometric to $X_i$. We therefore suppose $G_1$ and $G_2$ are both non-discrete. 
	
	 Since both $G_1$ and $G_2$ are hyperbolic and have trivial amenable radical, Theorem \ref{thm:hypisom_general} ensures each of $G_1$ and $G_2$ is either a virtually  connected rank one simple Lie group, or is totally disconnected. We claim $G_1$ is the former and $G_2$ is the latter. Indeed, as $X_2$ is not quasi-isometric to a symmetric space, $G_2$ must be totally disconnected. As  $G_1$ is quasi-isometric to an irreducible rank one symmetric space $Y$,  $G_1$ acts continuously properly and cocompactly on $Y$; see \cite[Theorem 19.25]{cornulier2018quasi}. As $G_1$ has trivial amenable radical, this action is faithful. Suppose for contradiction $G_1$ is totally disconnected.   Since $G_1$ is a closed subgroup of the Lie group $\Isom(Y)$, $G_1$  is also a  Lie group, hence it satisfies the \emph{no small subgroup property}, i.e. a sufficiently small identity neighbourhood contains no non-trivial subgroups. By Theorem \ref{thm:vanDantzig}, a totally disconnected group satisfying the no small subgroup property is discrete, contradicting our hypothesis that  $G_1$ is not discrete. Therefore, $G_1$ is  a virtually connected rank one simple Lie group.
	 
	Combining Lemmas \ref{lem:triv_centraliser} and \ref{lem:inj_proj}, we deduce that the projection $\Gamma\rightarrow G_1$ is injective. Passing to  finite index subgroups of $G=G_1\times G_2$ and $\Gamma$ as needed, we may assume $G_1$ is connected. We can thus apply Theorem \ref{thm:abstract_mixedlattice}  to deduce $\Gamma$ is an arithmetic lattice in a semisimple algebraic group. By Lemma \ref{lem:prodhyp_qitoalgebraicgroup}, this can only occur when $X_2$ is  quasi-isometric to  a $k$-regular tree for some $k\geq 3$. 
\end{proof}

We conclude this section with a theorem in a similar vein to Theorem \ref{thm:splitlinear_main}, where strengthening our algebraic assumption on $\Gamma$ allows one to deduce $\Gamma$ is arithmetic or virtually splits. We recall from the introduction that a group is  \emph{linear} if it has a faithful finite-dimensional linear representation over a field of characteristic $\neq 2,3$.
\begin{thm}\label{thm:splitlinear_main}
	Let $\Gamma$ be a  finitely generated linear group quasi-isometric to $\Pi_{i=1}^nX_i$, where each $X_i$ is a cocompact proper  non-elementary hyperbolic metric space and $n\geq 2$. Then either:
	\begin{enumerate}
		\item $\Gamma$ is virtually a uniform arithmetic lattice in a semisimple algebraic group, in which case each $X_i$ is of coarse algebraic type;
		\item After permuting the factors of $\Pi_{i=1}^nX_i$, there is an $1\leq r<n$ such that $\Gamma$ virtually splits as a direct product $\Gamma_1\times \Gamma_2$, where $\Gamma_1$ is quasi-isometric to $\Pi_{i=1}^rX_i$ and $\Gamma_2$ is quasi-isometric to $\Pi_{i=r+1}^nX_i$. 
	\end{enumerate} 

	Moreover, $\Gamma$ virtually splits as  $\Pi_{k=1}^m\Gamma_k$ for $m\leq n$, where each $\Gamma_k$ is either a uniform arithmetic lattice in a semisimple algebraic group   or a hyperbolic group quasi-isometric to $X_i$ for some $i$. 
		In particular, if at most one of the $X_i$ is of coarse algebraic type, then $m=n$ and each $\Gamma_i$ is quasi-isometric to $X_i$.
\end{thm}

We prove Theorem \ref{thm:splitlinear_main} by invoking the following theorem of  Caprace--Monod. The \emph{quasi-centre} of a topological group is the subgroup consisting of elements with open centraliser.
\begin{thm}[{\cite[Theorem 5.1]{capracemonod2009isometry}}]\label{thm:arithm_linear}
	Let $\Gamma\leq G=G_1\times \dots \times G_n$ be a finitely generated uniform  irreducible lattice, where $G$ is a locally compact group with trivial amenable radical and trivial quasi-centre. Suppose  $\Gamma$ admits a faithful Zariski-dense representation into a semisimple group over a field of characteristic $\neq 2,3$.
	After replacing $G$ and $\Gamma$ by a finite index subgroups, $G$ is a semisimple algebraic group and  $\Gamma$ is an arithmetic lattice   in $G$. 
\end{thm}
\begin{rem}
The definition of irreducibility used here is  a priori weaker than the notion used in \cite{capracemonod2009isometry}, so technically Theorem \ref{thm:arithm_linear} is not a special case of \cite[Theorem 5.1]{capracemonod2009isometry}.          
 Nonetheless, the present notion of irreducibility implies irreducibility as used in \cite{monod2006superrigidity}, so we can deduce Theorem \ref{thm:arithm_linear} by following the strategy indicated in the discussion after Proposition 5.5 of \cite{capracemonod2009isometry}. Indeed,  since we have assumed the lattice is uniform,   it follows from \cite[Corollary 4 \& Lemma 59]{monod2006superrigidity} that $(\Gamma, G)$ is a \emph{super-rigid pair} in the sense of \cite[\S 5.A.]{capracemonod2009isometry}, and so we can apply Proposition 5.5 of \cite{capracemonod2009isometry} to deduce Theorem \ref{thm:arithm_linear}.
\end{rem}

To ensure the hypotheses of Theorem \ref{thm:arithm_linear} are satisfied, we use the following proposition to show the desired locally compact group  has trivial quasi-centre. 
\begin{prop}\label{prop:qztriival}
	Let $\Gamma\leq G=G_1\times \dots \times G_n$ be a uniform irreducible lattice, where each $G_i$ is a nontrivial hyperbolic locally compact group with trivial amenable radical. If $\Gamma$ is residually finite, then $\QZ(G)$ is trivial.
\end{prop}
This proposition will be deduced from the following lemma.
A group is \emph{topologically locally finite} if every finite subset is contained in a compact subgroup.
\begin{lem}[{\cite[Proposition 4.9]{capracemonod2009isometry}}]\label{lem:qztopfinite}
	Let $\Gamma \leq G = G_1 \times G_2$ be a uniform lattice in a product of compactly
	generated locally compact groups. Assume that $G_2$ is totally disconnected and that the
	centraliser in $G_1$ of any uniform lattice of $G_1$ is trivial. If the projection $\Gamma\rightarrow G_1$ is injective, then the quasi-centre $\QZ(G_2)$ is topologically locally finite.
\end{lem}

\begin{proof}[Proof of Proposition \ref{prop:qztriival}]
	As each projection $G\rightarrow G_i$ is open,  the projection of $\QZ(G)$ to $G_i$ is contained in $\QZ(G_i)$. Therefore, it is enough to show each $\QZ(G_i)$ is trivial.  By Theorem \ref{thm:hypisom_general}, each $G_i$ is either a virtually connected simple rank one Lie group, or is  totally disconnected. In the former case $\QZ(G_i)$ is trivial, so it remains to show $\QZ(G_i)$ is trivial when  $G_i$ is totally disconnected.
	
	We assume $G_i$ is totally disconnected and write $H\coloneqq G_1\times \dots \times G_{i-1}\times G_{i+1}\dots\times G_n$.  Note that $\Gamma$ is a uniform lattice in $H\times G_i$ and that both $H$ and $G_i$ are non-discrete since $\Gamma$ is irreducible.  By Lemmas \ref{lem:triv_centraliser} and  \ref{lem:inj_proj},  the projection $\Gamma\rightarrow H$ is injective. Therefore Corollary \ref{cor:triv_cen} and Lemma \ref{lem:qztopfinite} imply $\QZ(G_i)$ is topologically locally finite. Therefore $\QZ(G_i)$ has amenable closure; see e.g. \cite[Lemma 2.1 and Corollary 2.4]{caprace2009Amenable}. As $G_i$ has trivial amenable radical,  $\QZ(G_i)$ is trivial.
\end{proof}
We now prove Theorem \ref{thm:splitlinear_main}, closely following the proof of \cite[Theorem 6.6]{capracemonod2009isometry}.
\begin{proof}[Proof of Theorem \ref{thm:splitlinear_main}]
	By Corollary \ref{cor:hyp_product_lattice} and using the fact $\Gamma$ is linear, hence residually finite, we replace $\Gamma$ with a  finite index subgroup that is a uniform lattice in a product $G_1\times\dots\times G_n$ as in Corollary \ref{cor:hyp_product_lattice}. In particular,  $\Gamma$ has trivial amenable radical, each $G_i$ is a hyperbolic locally compact group with trivial amenable radical quasi-isometric to $X_i$, and the projection of $\Gamma$ to each $G_i$ is dense.
	
	 If $\Gamma$ is not  irreducible, then Proposition \ref{prop:irred} ensures $\Gamma$ virtually splits as $\Gamma_1\times \Gamma_2$, where each $\Gamma_i$ is a uniform lattice in a proper subproduct of $G_1\times \dots\times G_n$, and we  have finished proving the first part of Theorem \ref{thm:splitlinear_main}.  We thus assume $\Gamma\leq G_1\times \dots \times G_n$ is irreducible; in particular, no $G_i$ is discrete. 
	
	By hypothesis, $\Gamma$ has a  faithful finite-dimensional linear representation $\rho:\Gamma\to\text{GL}(V)$ over a field of characteristic $\neq 2,3$. Taking the Zariski closure of $\rho(\Gamma)$ and quotienting out by the (solvable) radical gives a representation of $\Gamma$ to a semisimple algebraic group. Since $\Gamma$ has trivial amenable radical, this yields a  faithful representation of $\Gamma$ to a semisimple algebraic group with 
	Zariski-dense image. By Proposition \ref{prop:qztriival}, $G$ has trivial quasi-centre. We can thus apply Theorem \ref{thm:arithm_linear} to deduce $\Gamma$ is an arithmetic lattice in a semisimple algebraic group.
	This proves the first part of Theorem \ref{thm:splitlinear_main}. 
	
	Applying the first part of Theorem \ref{thm:splitlinear_main} inductively, we deduce $\Gamma$ virtually splits as a product $\Pi_{k=1}^m\Gamma_k$ of irreducible higher rank arithmetic lattices and  hyperbolic groups. Each hyperbolic factor is  quasi-isometric to some $X_i$, whilst each higher rank arithmetic factor is quasi-isometric to a subproduct of $\Pi_{i=1}^nX_i$ consisting of at least two factors. Lemma \ref{lem:prodhyp_qitoalgebraicgroup} thus implies  that if at most one $X_i$ is of coarse algebraic type, then there are no higher rank arithmetic factors, and so $\Gamma$ splits as a product $\Pi_{i=1}^n\Gamma_i$, where $\Gamma_i$ is quasi-isometric to $X_i$.
\end{proof}

\subsection{Biautomaticity of groups quasi-isometric to products of hyperbolic spaces}
It is a well known open problem to determine  whether hyperbolic groups act geometrically on a   $\CAT(0)$ space. The motivation behind this is that hyperbolic metric spaces typically have poor local geometry, e.g. finite subgroups acting isometrically on such a space need not fix a point. To partially remedy this, the notion of a Helly group was recently introduced by Chapolin et al. \cite{chalopin2020helly}, a class of groups that satisfy similar properties to $\CAT(0)$ groups.
\begin{defn}
	A \emph{Helly graph} is a graph in which any family of pairwise intersecting combinatorial balls have non-trivial intersection. A \emph{Helly group} is a finitely generated group acting geometrically on a Helly graph.
\end{defn}
It was shown by Lang and Chapolin et al. that hyperbolic groups are Helly  \cite{lang2013Injective,chalopin2020helly}.  Our interest in Helly graphs comes the following proposition, which follows easily from the theory developed by Lang and Chapolin et al.
\begin{prop}\label{prop:helly}
	Let $X$ be a locally finite hyperbolic graph.  There exists a Helly graph $\Helly(X)$ and a graph monomorphism $\sigma:X\to \Helly(X)$, which is also an isometric embedding, satisfying the following properties:
	\begin{enumerate}
		\item $\Helly(X)$ is locally finite;\label{item:helly1}
		\item $\Helly(X)=N_A(\im(\sigma))$ for some $A\geq 0$;\label{item:helly2}
		\item There is a monomorphism $\Psi:\Aut(X)\to \Aut(\Helly(X))$ such that  $\sigma\circ \alpha=\Psi(\alpha) \circ \sigma$ for all $\alpha\in \Aut(X)$.\label{item:helly3}
	\end{enumerate}
\end{prop}
\begin{proof}
	Let $V=V(X)$ be the vertex set of $X$ equipped with the subspace metric. We recall the construction of the \emph{injective hull} of  $V$, following the notation in \cite{lang2013Injective} and \cite{chalopin2020helly}, which we refer to for proofs. Let $E(V)\subseteq \bbR^V$ be the set of functions $f:V\to \bbR$ satisfying the following properties:
	\begin{enumerate}
		\item $f(x)+f(y)\geq d(x,y)$ for all $x,y\in V$;
		\item $f$ is extremal in the following sense: if $g\in \bbR^V$ satisfies:
		\begin{itemize}
			\item $g(x)+g(y)\geq d(x,y)$ for all $x,y\in V$ and
			\item $g(x)\leq f(x)$ for all $x\in V$,
		\end{itemize}  
	then $f=g$.
	\end{enumerate}
It is shown that $E(V)$ is a metric space when equipped with the metric $d_\infty(f,g)\coloneqq\sup_{y\in Y}\lvert f(y)-g(y)\rvert$. There is an isometric embedding $e:V\to E(V)$  defined by $y\mapsto d_y$, where $d_y(x)=d(x,y)$. 

As in \cite[\S 4]{chalopin2020helly},  the subspace $E^0(V)\coloneqq E(V)\cap \bbZ^V$ is called the \emph{discrete injective hull}. The graph $\Helly(X)$ is the graph with vertex set $E^0(V)$, where two vertices $f,g\in E^0(V)$ are joined by an edge if $d_\infty(f,g)=1$. As suggested by the terminology, the graph $\Helly(X)$ is in fact a Helly graph \cite[Theorem 4.4]{chalopin2020helly}. Since $e:V\to E(V)$ is an isometric embedding with image contained in $E^0(V)$,  $e$ induces a graph monomorphism  $\sigma:X\to \Helly(X)$ agreeing with $e$ on $V=V(X)$. It is shown in  Theorem 4.4 of \cite{chalopin2020helly} that $f,g\in E^0(V)$ can be joined by a path in $\Helly(X)$ of length $d_\infty(f,g)$, hence $\sigma$ is an isometric embedding. 

It is shown in \cite[Proposition 3.7]{lang2013Injective} that there is a monomorphism $\Phi:\Aut(X)\to\Isom(E(V))$ given by $\Phi(\alpha)(f)\coloneqq f\circ \alpha^{-1}$. Thus $\Phi$ satisfies $\Phi(\alpha)\circ e=e\circ \alpha$ for all $\alpha\in \Aut(X)$.   All elements of $\Phi(\Aut(X))$ preserve the subspace $E^0(V)\subseteq E(V)$ of integer valued functions, hence $\Phi$ induces a homomorphism $\Psi:\Aut(X)\to \Aut(\Helly(X))$ such that  $\sigma\circ \alpha=\Psi(\alpha) \circ \sigma$ for all $\alpha\in \Aut(X)$, showing (\ref{item:helly3}) holds.

All that remains is to show (\ref{item:helly1}) and (\ref{item:helly2}) hold. We now assume the graph $X$, or equivalently its vertex set $V$, is hyperbolic.  By \cite[Propositions 3.12  and 5.11]{chalopin2020helly}, there exists some $A\geq 0$ such that $N_A(\im(e))=E(V)$. It follows that $N_{A'}(\im(\sigma))=\Helly(X)$ for some $A'\geq 0$. For $x,y\in V$, we define the interval $I(x,y)$ to be $\{z\in V\mid d(x,z)+d(z,y)=d(x,y)\}$. We say that $V$ has \emph{$\beta$-stable intervals} if for every triple $x,y,y'\in V$ with $d(y,y')=1$, we have $d_\Haus(I(x,y),I(x,y'))\leq \beta$. As noted in \cite{lang2013Injective}, if $V$ is $\delta$-hyperbolic then it has $(\delta+1)$-stable intervals. A theorem of Lang now ensures that if $V$ is hyperbolic, then $E(V)$ is a proper metric space \cite[Theorem 1.1]{lang2013Injective}. Since $E(V)$ is proper, it follows that $\Helly(X)$ is locally finite.
\end{proof}
We use this to deduce:
\begin{cor}
	Let $X$ be a locally finite hyperbolic graph. Then the map $\sigma:X\to \Helly(X)$ is a quasi-isometry that  conjugates  a group action $\Gamma\curvearrowright X$ by graph automorphisms to an action $\Gamma\curvearrowright \Helly(X)$ on a Helly graph.
\end{cor}

We can thus refine Theorem \ref{thm:hyp_productmain} as follows:
\begin{cor}\label{cor:hyp_producthelly}
	Let $\Gamma$ be a finitely generated group quasi-isometric to $\Pi_{i=1}^nX_i$, where each $X_i$ is a cocompact   non-elementary  proper hyperbolic metric space. Then $\Gamma$ acts geometrically on $\Pi_{i=1}^nY_i$, preserving the product structure, where each $Y_i$ is quasi-isometric to $X_i$ and is either a rank one symmetric space  or a locally finite  Helly graph.
\end{cor}

Recall the direct product $\Pi_{i=1}^nX_i$ in Corollary \ref{cor:hyp_producthelly} is a direct product of metric spaces equipped with the $\ell_2$-metric.  However, if each $X_1,\dots, X_n$ is a graph, we can instead define a direct product $\boxtimes_{i=1}^nX_i$ to be  a graph with vertex set $V(X_1)\times \dots\times V(X_n)$, where two distinct vertices $v=(v_1,\dots,v_n)$ and $w=(w_1,\dots,w_n)$ are joined by an edge if for all $i$, either $v_i=w_i$ or $v_i$ and $w_i$ are joined by an edge. This induces the $\ell_\infty$-metric on the vertex set. If $\Gamma$ acts geometrically on  $\Pi_{i=1}^nX_i$ preserving the product structure and acting via a graph isomorphism on each factor, then $\Gamma$ acts geometrically on $\boxtimes_{i=1}^nX_i$. Moreover, if each $X_i$ is Helly, then $\boxtimes_{i=1}^nX_i$ is also Helly \cite[Proposition 5.1]{chalopin2020helly}. Combining  these observations with Corollary \ref{cor:hyp_producthelly}  shows the following:

\begin{thm}\label{thm:hypprod_helly}
	Let $\Gamma$ be a finitely generated group quasi-isometric to $\Pi_{i=1}^nX_i$, where each $X_i$ is a cocompact non-elementary proper  hyperbolic metric space not quasi-isometric to a rank one symmetric space. Then $\Gamma$ acts geometrically on $\boxtimes_{i=1}^nY_i$, where each $Y_i$ is a locally finite Helly graph quasi-isometric to $X_i$. Therefore,  $\Gamma$ is Helly and satisfies the conclusions of \cite[Theorem 1.5]{chalopin2020helly}.  In particular, $\Gamma$ is biautomatic.
\end{thm}

As mentioned in the introduction, Hughes--Valiunas show the conclusion of Theorem \ref{thm:hypprod_helly} is false if we weaken the hypothesis that each $X_i$  is not quasi-isometric to a symmetric space \cite{hughes2022commensurating}. Nonetheless, we can partially remedy the situation by relaxing the Helly property to its non-discrete analogue.  A metric space $X$ is \emph{injective} if a collection of pairwise intersecting metric balls have non-empty intersection. Groups acting geometrically on injective metric spaces possess many natural properties in common with $\CAT(0)$ groups; see \cite[Theorem 1.1]{haettel2021injective} for a list of such properties.
\begin{thm}\label{thm:group_prod_semi_hyp}
If $\Gamma$ is a finitely generated group quasi-isometric to $\Pi_{i=1}^nX_i$, where each $X_i$ is a cocompact non-elementary proper   hyperbolic metric space, then $\Gamma$ acts geometrically on an injective metric space. In particular, $\Gamma$ is semihyperbolic.
\end{thm} 
\begin{proof}
	By Theorem \ref{thm:hyp_productmain}, $\Gamma$ acts geometrically preserving the product structure  on  $\Pi_{i=1}^nY_i$, where $Y_i$ is either a rank one symmetric space or a locally finite hyperbolic graph.  As in the proof of Proposition \ref{prop:helly}, for each $i$ there is an isometric embedding $e_i:Y_i\to E(Y_i)$ where $E(Y_i)$ is the injective hull of $Y_i$; see for instance \cite{lang2013Injective}. We write $(\Pi_\infty)_{i=1}^nE(Y_i)$ to denote the product metric space equipped with  the $\ell_\infty$-metric. Since each $E(Y_i)$ is an injective space, so is $(\Pi_\infty)_{i=1}^nE(Y_i)$. By \cite[Proposition 3.7]{lang2013Injective}, the isometric action of $\Gamma$ on $\Pi_{i=1}^nY_i$ canonically extends to an isometric of $\Gamma$ on $(\Pi_\infty)_{i=1}^nE(Y_i)$. It remains to show this action is geometric.
	
	The result follows readily by showing each $e_i:Y_i\to E(Y_i)$ satisfies the following properties:
	\begin{enumerate}
		\item $N_{A_i}(e(Y_i))=E(Y_i)$ for some $A_i\geq 0$;
		\item $E(Y_i)$ is proper.
	\end{enumerate}
 As noted in the proof of Proposition \ref{prop:helly}, these properties follow from work of Lang if $Y_i$ is a locally finite hyperbolic graph  \cite{lang2013Injective}. In the case where $Y_i$ is a rank one symmetric space, these properties were shown by Haettel \cite[Proposition 4.6]{haettel2021injective}. Finally, we note that groups acting geometrically on injective spaces are semihyperbolic; see e.g. \cite[Theorem 1.1]{haettel2021injective}.
\end{proof}

\subsection{Boundary rigidity of \texorpdfstring{$\CAT(0)$}{CAT(0)} spaces quasi-isometric to products of hyperbolic groups}
Combining Corollary \ref{cor:hyp_product_lattice} with results of Monod and Caprace--Monod, we prove the following:

\begin{thm}\label{thm:boundary_rigidity}
	Let $\Gamma$ be a finitely generated quasi-isometric to $\Pi_{i=1}^nX_i$, where each $X_i$ is a cocompact proper  non-elementary hyperbolic metric space. Suppose $\Gamma$ acts geometrically on a proper $\CAT(0)$ space $Y$. Then $Y$ contains a canonical closed convex non-empty $\Gamma$-invariant
	subset $Y'$ that splits as $Y_1\times \dots \times Y_n$, where $Y_i$ is quasi-isometric to $X_i$.
		Moreover,  the visual boundary $\partial Y$ is homeomorphic to the join  $\partial X_1*\dots*\partial X_n$, where $\partial X_i$ is the Gromov boundary of $X_i$. In particular, $\partial Y$ depends only on $\Gamma$ and not the choice of  $\CAT(0)$ space $Y$.
\end{thm}

We first show a  $\CAT(0)$ space quasi-isometric to a product of hyperbolic spaces has no Euclidean de Rham factor.
\begin{lem}\label{lem:noeuc_derham}
	Let $\Gamma$ be a finitely generated group quasi-isometric to a product of proper non-elementary hyperbolic spaces and acting geometrically on a proper $\CAT(0)$ space $Y$. Then $Y$ has no Euclidean de Rham factor.
\end{lem}
\begin{proof}
	Since $\Gamma$ is quasi-isometric to a product of proper non-elementary hyperbolic spaces, it acts geometrically on $\Pi_{i=1}^nY_i$ as in Theorem \ref{thm:hyp_productmain}, preserving the product structure. Moreover, Theorem \ref{thm:hyp_productmain} ensures the finite index subgroup $\Gamma^*$ does not fix a point of $\partial Y_i$.
	Suppose for contradiction $Y$ splits as a direct product $Y=\bbR^n\times Y'$ with $n>0$ maximal.  It follows from \cite[Theorem 2]{caprace2019erratum} that $\Gamma$ contains a commensurated subgroup $\Lambda\cong \bbZ^n$.     Note that $\Lambda^*\coloneqq\Lambda\cap \Gamma^*\cong \bbZ^n$ is commensurated in $\Gamma^*$. 
	
	Since $\Lambda^*$ is an amenable subgroup acting isometrically on each $Y_i$, it either has bounded orbits or fixes a non-empty subset $S_i\subseteq \partial Y_i$ of cardinality at most two. In the latter case, as $\Lambda^*\leq \Gamma^*$ is commensurated, it follows that $S_i$ is also  $\Gamma^*$-invariant. This contradicts the fact that $Y_i$ is non-elementary and  $\Gamma^*$ acts on $Y_i$ cocompactly without fixing a point of $\partial Y_i$. Therefore $\Lambda^*$ acts  on each $Y_i$ with  bounded orbits, so $\Lambda^*$ acts with bounded orbits on $\Pi_{i=1}^nY_i$. As $\Lambda^*$ is infinite, this contradicts the fact that $\Lambda^*\leq \Gamma^*$ acts properly on $\Pi_{i=1}^nY_i$.
\end{proof}

We will use the  following theorems, which are special cases of results of Monod and Caprace--Monod respectively:
\begin{thm}[{\cite[Corollary 10]{monod2006superrigidity}}]\label{thm:monod_split}
	Let $X$ be a proper CAT(0) space. Suppose $G$ acts isometrically and minimally on $X$  and  does not fix  a point of $\partial X$. 
	If $G = G_1 \times \dots \times G_n$ is any product of groups $G_i$, then $X$ splits $G$-equivariantly as a product $X=X_1\times \dots\times X_n$, where $G_i$ acts isometrically on $X_i$.
\end{thm}

\begin{thm}[{\cite[Theorem 8.4]{capracemonod2009strucure}}]\label{thm:monod_superrigid}
	Let $\Gamma\leq G=G_1\times\dots\times G_n$ be an irreducible  uniform 	lattice, where each $G_i$ is a locally compact compactly generated group. Let $X$ be a proper $\CAT(0)$ space with finite-dimensional boundary and no Euclidean de Rham factor. If  $\Gamma$ acts isometrically and minimally on $X$ without fixing a point of $\partial X$, then the $\Gamma$-action on $X$ extends to a continuous $G$-action by
	isometries.
\end{thm}

\begin{proof}[Proof of Theorem \ref{thm:boundary_rigidity}]
	We proceed by induction on $n$, noting that when $n=1$, there is nothing to prove as the visual boundary and Gromov boundary of a $\CAT(0)$ space coincide. By Corollary \ref{cor:hyp_product_lattice}, after replacing $\Gamma$ with a finite index subgroup and quotienting out by a finite normal subgroup, we can assume $\Gamma$ is a uniform lattice in a product $G_1\times \dots \times G_n$, where $G_i$ is a hyperbolic locally compact group with trivial amenable radical and the projection of $\Gamma$ to each $G_i$ is dense. 
	
	Suppose $\Gamma$ acts geometrically on a proper CAT(0) space $Y$. By Theorem \ref{thm:cat0_bdry}, there exists a canonical closed convex non-empty $\Isom(Y)$-invariant $\Isom(Y)$-minimal subset $Y'\subseteq Y$. Theorem \ref{thm:cat0_bdry} also ensures $Y'$ is $\Gamma$-minimal.  By Lemma \ref{lem:noeuc_derham}, $Y'$ has no Euclidean de Rham factor, so Theorem \ref{thm:cat0_bdry} implies $\Gamma$ does not fix a point of $\partial Y'$.
	
	First suppose $\Gamma\leq G_1\times \dots \times G_n$ is an irreducible lattice. Then Theorem \ref{thm:monod_superrigid} implies the action of  $\Gamma$ on $Y'$ extends continuously to an isometric action of $G$ on $Y'$. Lemma \ref{lem:lattice_extension_geom} ensures the action of $G$ on $Y'$ is geometric.  Theorem \ref{thm:monod_split} ensures  $Y'$ splits as $Y'=Y_1\times \dots \times Y_n$, where $G_i$ acts isometrically on $Y$, so the action of $G_i$ on each $Y_i$ is geometric. Thus $X_i$ is quasi-isometric to $Y_i$ and so $\partial X_i$ and $\partial Y_i$ are homeomorphic. Thus $\partial Y'$ is homeomorphic to the join $\partial X_1*\dots*\partial X_n$. Since $Y'$ is a closed convex subspace of $Y$ such that $N_A(Y')=Y$ for $A$ sufficiently large, $\partial Y$ is also homeomorphic to the join $\partial X_1*\dots*\partial X_n$.
	
	Now suppose $\Gamma\leq G_1\times \dots \times G_n$ is not irreducible. After permuting the factors, there is some $1\leq r<n$ such that setting $H_1=G_1\times \dots\times G_r$ and $H_2=G_{r+1}\times\dots\times G_n$, the projection of $\Gamma$ to $H_1$ is discrete. Applying Proposition \ref{prop:irred}, we deduce that some finite index subgroup of $\Gamma$ splits as $\Gamma_1\times \Gamma_2$, where $\Gamma_i$ is a uniform lattice in $H_i$. By Theorem \ref{thm:cat0_bdry}, the action of $\Gamma_1\times \Gamma_2$ on $Y'$ is minimal and does  not fix a  point of $\partial Y'$. By Theorem \ref{thm:monod_split}, $Y'$ splits as $Y_1\times Y_2$, where $\Gamma_i$ acts geometrically on $Y_i$. Applying the inductive hypothesis, we deduce $\partial Y_1\cong \partial X_1*\dots*\partial X_r$ and $\partial Y_2\cong \partial X_{r+1}*\dots*\partial X_{n}$. Therefore $\partial Y\cong \partial Y'\cong  \partial X_1*\dots*\partial X_n$ as required.
\end{proof}

\section{Groups quasi-isometric to central extensions of hyperbolic groups}\label{sec:qicent}

In this section we prove  the following theorem and discuss its consequences:

\begin{thm}\label{thm:qitocent_main}
	Let $Q$ be a locally finite vertex-transitive non-elementary  hyperbolic graph and let  $n\geq 1$. A finitely generated group $G$ is quasi-isometric to $\bbR^n\times Q$ if and only if both the following hold:
	\begin{enumerate}
		\item there is a subgroup $\bbZ^n\cong H\alnorm G$ such that the quotient space $G/H$ is quasi-isometric to $Q$;
		\item the image of the modular homomorphism $G\rightarrow \Comm(H)\cong \GL_n(\bbQ)$ is conjugate to a subgroup of $O_n(\bbR)$.
	\end{enumerate} 
\end{thm}

\subsection{Quasi-isometrically trivial coarse bundles}
Consider a finitely generated group $G$ containing  a commensurated subgroup $H\cong \bbZ^n$. Proposition \ref{prop:alnorm_cbundle} ensures  the quotient map $p:G\to G/H$ is a coarse bundle. In this subsection we prove Proposition \ref{prop:qitrivial}, which gives necessary and sufficient conditions for $p:G\rightarrow G/H$ to be quasi-isometrically trivial in the sense of Definition \ref{defn:qitrivial}. Using a theorem of Mineyev concerning $\ell_\infty$-cohomology of hyperbolic simplicial complexes, we show in Lemma \ref{lem:lipsection} that one of the conditions of Proposition \ref{prop:qitrivial} is automatically satisfied when the quotient space $G/H$ is hyperbolic.

A \emph{coarse Lipschitz section} of a coarse bundle $p:G\rightarrow G/H$ is a coarse Lipschitz map $s:G/H\rightarrow G$ that is a section to $p$, i.e.\ $p\circ s=\id_{G/H}$. The following proposition partially generalises the equivalence of (1) and (2) in \cite[Proposition 8.3]{kleinerleeb2001symmetric}, which proves the special case in which $H$ is a central subgroup in $G$. 
\begin{prop}\label{prop:qitrivial}
	Let $G$ be a finitely generated group containing a commensurated subgroup $\bbZ^n\cong H\alnorm G$ for $n\geq 1$. Then $p:G\rightarrow G/H$  is quasi-isometrically trivial  if and only if both of the following conditions hold:
	\begin{enumerate}
		\item  $p:G\rightarrow G/H$ has a coarse Lipschitz section;
		\item the image of the modular homomorphism $\Delta:G\rightarrow \GL_n(\bbR)$ is conjugate to a subgroup of $\Orth_n(\bbR)$.
	\end{enumerate}
\end{prop}
\begin{proof}[Proof of the $\implies$ direction]
	If $p:G\rightarrow G/H$ is quasi-isometrically trivial, then $G$ is quasi-isometric to $H\times G/H$ via a  fibre-preserving quasi-isometry. Therefore, $p$  has a coarse Lipschitz section, and $p:G\rightarrow G/H$ has bounded fibre distortion by  Corollary \ref{cor:fibre-distortion_bounded}. It follows from Example \ref{exmp:fibre_abelian} that the image of the modular homomorphism $\Delta:G\rightarrow \GL_n(\bbR)$ is conjugate to a subgroup of $\Orth_n(\bbR)$ as required. 
\end{proof}
 The remainder of the proof of Proposition \ref{prop:qitrivial} is the content of the following lemma:
\begin{lem}\label{lem:qitrivial}
	Let $G$ and $H$ be as above and suppose $p:G\rightarrow G/H$ satisfies (1) and (2) of Proposition \ref{prop:qitrivial}. Then $p:G\rightarrow G/H$  is quasi-isometrically trivial.
\end{lem}

Throughout the remainder of this subsection,  we fix the following constants and notation. Let $G$ be a finitely generated group containing a commensurated subgroup  $\bbZ^n\cong H\alnorm G$. Let $S$ be a  symmetric generating set of $G$ and let  $d_G$  and  $\lvert \cdot \rvert_S$ be the word metric and word norm with respect to $S$.
 We also equip $G/H$ with the relative word metric $d_{G/H}$ with respect to $S$.
  We equip $\bbR^n$ with a standard Euclidean inner product $\lVert \cdot \rVert$ and  fix a monomorphism $\iota:H\rightarrow \bbR^n$ whose image is a uniform lattice. We equip $H$ with the metric $d_H(h,h')=\lVert \iota(h)-\iota(h')\rVert$.  
   
   Let $\Delta:G\rightarrow \GL_n(\bbR)$ be the modular homomorphism.  We assume the image of $\Delta$ is conjugate to a subgroup of $\Orth_n(\bbR)$. The modular homomorphism $\Delta$ satisfies the property that for all $g\in G$ and $h\in H\cap g^{-1}Hg$, $\Delta(g)(\iota(h))=\iota(ghg^{-1})$.  Indeed, the definition of $\Delta$ ensures that for all $g\in G$, there is a finite index subgroup $H_g\leq  H\cap g^{-1}Hg$ such that  $\Delta(g)(\iota(h))=\iota(ghg^{-1})$ for all $h\in H_g$. However, a homomorphism between groups isomorphic to $\bbZ^n$ is determined by its restriction to a finite index subgroup, hence we can assume  $H_g=H\cap g^{-1}Hg$ as above.
   
     We can choose  a proper non-decreasing function $\eta:\bbR_{\geq 0}\to\bbR_{\geq 0}$ and constants $K\geq 1$, $A\geq 0$ so that the following hold.
\begin{itemize}
 	\item  $N_A(\iota (H))=\bbR^n$.
	\item The inclusion $H\rightarrow G$ is an $(\eta,K,A)$-coarse embedding.
	\item The map $p:G\rightarrow G/H$ is a $(K,A)$-coarse bundle.
	\item $\lVert\Delta(g)v\rVert\leq A\lVert v\rVert$ for all $g\in G$ and $v\in \bbR^n$.
\end{itemize} 
The last point follows from the fact that $\im(\Delta)$ is conjugate to a subgroup of $\Orth_n(\bbR)$, hence its elements have uniformly bounded operator norm.
We first show the following technical lemma.
\begin{lem} \label{lem:projlemma}
	For each $r\geq 0$, there is a constant $B_r$ such that the following hold for all $h_1,h_2,h_3\in H$ and $g_1,g_2\in G$.
	\begin{enumerate} 
		\item \label{item:projlemma1} If  $\lvert g_1\rvert_S\leq r$ and $\lvert h_1g_1h_2\rvert_S\leq r$, then \[\lVert \iota(h_1)+\Delta(g_1)(\iota(h_2))\rVert\leq B_r.\]
		\item \label{item:projlemma2} If  $\lvert g_1\rvert_S,\lvert g_2\rvert_S\leq r$ and $\lvert h_1g_1h_2g_2h_3\rvert_S\leq r$, then  \[\lVert \iota(h_1)+\Delta(g_1)\iota(h_2)+\Delta(g_1g_2)\iota(h_3)\rVert\leq B_r.\]
		\item \label{item:projlemma3} If  $\lvert g_1\rvert_S\leq r$ and $\lVert \iota(h_1)+\Delta(g_1)(\iota(h_2))\rVert\leq r$,  then $\lvert h_1g_1h_2\rvert_S\leq B_r$. 
	\end{enumerate}
\end{lem}
\begin{proof}
	Fix  $r\geq 0$. Let $H_0\coloneqq \bigcap_{\lvert g\rvert_S\leq r} H^g$, noting  that $H_0$ is a finite index subgroup of $H$. Let   $\pi:H\rightarrow H_0$ be a closest point projection. We can pick a constant $D$ such that both $\sup_{h\in H}d_G(\pi(h),h)$  and $\sup_{h\in H}\lVert \iota(h)-\iota(\pi(h))\rVert$ are  bounded by  $D$. We set $B_r\coloneqq \max (\tilde\eta(3r+2D)+2AD,K(r+AD)+r+D)$, where $\tilde{\eta}$ is as in Remark \ref{rem:inv_control}.

	(\ref{item:projlemma2}): We observe \begin{align*}
		\lvert h_1g_1\pi(h_2g_2\pi (h_3)g_2^{-1})g_1^{-1}\rvert_S& \leq \lvert h_1g_1h_2g_2\pi (h_3)g_2^{-1}\rvert_S+D+r\\
		& \leq \lvert h_1g_1h_2g_2h_3\rvert_S+2D+2r\\
		&\leq 3r+2D
	\end{align*}
	Since $\pi(h_3)\in H\cap g_2^{-1}Hg_2$ and  $\pi(h_2g_2\pi(h_3)g_2^{-1})\in H\cap g_1^{-1}Hg_1$, we deduce \begin{align*}
		\lVert \iota(h_1)+\Delta(g_1)\iota(h_2)+\Delta(g_1g_2)\iota(h_3)\rVert& \leq \lVert \iota(h_1)+\Delta(g_1)\iota(h_2)+\Delta(g_1g_2)\iota(\pi(h_3))\rVert+AD\\
		& = \lVert \iota(h_1)+\Delta(g_1)(\iota(h_2g_2\pi(h_3)g_2^{-1}))\rVert+AD\\
		& \leq \lVert \iota(h_1)+\Delta(g_1)(\iota(\pi(h_2g_2\pi(h_3)g_2^{-1})))\rVert+2AD\\
		& = \lVert \iota(h_1g_1\pi(h_2g_2\pi(h_3)g_2^{-1})g_1^{-1} )\rVert+2AD\\
		& \leq \tilde\eta(3r+2D)+2AD\leq B_r.
	\end{align*}
	
	(\ref{item:projlemma1}): This follows from (\ref{item:projlemma2}) after setting $g_2=h_3=e_G$.
	
	(\ref{item:projlemma3}): Since $\pi(h_2)\in H\cap g_1^{-1}Hg_1$, we have \begin{align*}
		\lVert \iota(h_1g_1\pi(h_2)g_1^{-1})\rVert & = \lVert \iota(h_1)+\Delta(g_1)\iota(\pi(h_2))\rVert\\
		& \leq \lVert \iota(h_1)+\Delta(g_1)\iota(h_2)\rVert+\lVert \Delta(g_1)\iota(h_2)-\Delta(g_1)\iota(\pi(h_2))\rVert\\
		& \leq r+AD.
	\end{align*}
	Therefore, $\lvert h_1g_1h_2\rvert_S\leq \lvert h_1g_1\pi(h_2)g_1^{-1}\rvert_S+r+D\leq K(r+AD)+r+D\leq B_r.$
\end{proof}

\begin{proof} [Proof of Lemma \ref{lem:qitrivial}]
	Let $\sigma:G/H\rightarrow G$ be a coarse Lipschitz section of $p:G\rightarrow G/H$.  By increasing $K$ and $A$ if necessary, we may assume $\sigma$ is a $(K,A)$-coarse Lipschitz section. For each $g\in G$, let $\sigma_g\coloneqq \sigma(gH)$, $h_g\coloneqq \sigma_g^{-1}g\in H$ and  $v_g\coloneqq \iota (h_g)$. 
	
	If $g\in G$ and $s\in S$, then $d_{G/H}(gH,gsH)\leq 1$ and so $d(\sigma_g,\sigma_{gs})\leq K+A$.
	Since  $g=\sigma_gh_g$ and  $gs=\sigma_{gs}h_{gs}$, we have \[\lvert h_gsh_{gs}^{-1}\rvert_S=\lvert \sigma_g^{-1}\sigma_{gs}\rvert_S \leq K+A.\] Applying Lemma \ref{lem:projlemma}, \begin{align*}
		\lVert v_g-\Delta(s)(v_{gs})\rVert = \lVert \iota(h_g)+\Delta(s)(\iota(h_{gs}^{-1}))\rVert\leq B_{K+A},
	\end{align*} with $B_{K+A}$ as in Lemma \ref{lem:projlemma}.  Therefore,
	\begin{align}\label{eqn:sec_ineq}
		\lVert \Delta(gs)(v_{gs})-\Delta(g)(v_g)\rVert & \leq A \lVert\Delta(s)(v_{gs})-v_g\rVert\leq AB_{K+A}\eqqcolon A_1.
	\end{align} 
	
	We define a map $f:G\rightarrow \bbR^n\times G/H$ by $g\mapsto (\Delta(g)(v_g),gH)$. We will show $f$ is a quasi-isometry.  We equip $Y\coloneqq \bbR^n\times G/H$ with the $\ell_1$-metric, i.e.\ \[d_Y((v,gH),(w,kH))=\lVert v-w\rVert +d_{G/H}(gH,kH),\] noting this is quasi-isometric to the more standard $\ell_2$-metric on $Y$.

	For all $g\in G$ and $s\in S$, (\ref{eqn:sec_ineq}) implies \[d_Y(f(g),f(gs))= \lVert \Delta(g)(v_g)-\Delta(gs)(v_{gs})\rVert +d_{G/H}(gH,gsH)\leq  A_1+1. \] Lemma \ref{lem:coarse_lip} implies that  $f$ is coarse Lipschitz. 
	Let $(v,gH)\in Y$. Since $N_A(\iota(H))= \bbR^n$, there exists an $h\in H$ such that $\lVert \iota(h)-\Delta(g^{-1})(v)\rVert\leq A$, and so $\lVert \Delta(g)(\iota(h))-v\rVert\leq A^2$. Since $\sigma_{g}H=gH$ and $h_{\sigma_{g}h}=h$, we see $f(\sigma_gh)=(\Delta(g)\iota(h),gH)$ and so $d(f(\sigma_gh),(v,gH))\leq A^2$. Thus $N_{A^2}(f(G))=Y$. 
	
	Let $g,t\in G$ and suppose $d(f(g),f(t))=N$. It follows that  $\lVert \Delta(t)(v_t)-\Delta(g)(v_{g})\rVert\leq N$ and $d_{G/H}(gH,tH)\leq N$. Since $p:G\rightarrow G/H$ is a $(K,A)$-coarse bundle, we have $d(g,tH)\leq KN+A$. Thus $t'\coloneqq gs_1\dots s_m\in tH$ for some $s_1,\dots, s_m\in S$ with $m\leq KN+A$. It follows that \begin{align*}
		\lVert \Delta(t)(v_{t'})- \Delta(g)(v_{g}) \rVert &\leq \sum_{i=1}^m\lVert \Delta(gs_1\dots s_i)(v_{gs_1\dots s_i})- \Delta(gs_1\dots s_{i-1})(v_{gs_1\dots s_{i-1}}) \rVert\\
		&\leq mA_1\leq A_1(KN+A)
	\end{align*}
	by applying (\ref{eqn:sec_ineq}) $m$ times.
	
	Now \begin{align*}
		\lVert \iota(h_t)-\iota(h_{t'})\rVert &=\lVert v_t-v_{t'}\rVert\leq A\lVert \Delta(t)(v_t)-\Delta(t)(v_{t'})\rVert\\
		&\leq A\lVert \Delta(t)(v_t)-\Delta(g)(v_{g})\rVert + A \lVert \Delta(t)(v_{t'})- \Delta(g)(v_{g}) \rVert \\
		& \leq AN+AA_1(KN+A)
	\end{align*}
	Since $\sigma_t=\sigma_{t'}$, we have  \[d_G(t,t')=d_G(\sigma_th_t,\sigma_{t'}h_{t'})\leq K\lVert \iota(h_t)-\iota(h_{t'})\rVert +A.\] As $d_G(g,t)\leq m+d_G(t,t')$ and $N= d_Y(f(g),f(t))$, we deduce $d_G(g,t)\leq K_2  d_Y(f(g),f(t))+A_2$ for some constants $K_2$ and $A_2$ depending on $K$, $A$ and $A_1$. This completes the proof that $f$ is a quasi-isometry.
	Finally, we note  that if $\pi:\bbR^n\times G/H\rightarrow G/H$ is the projection, then $p=\pi\circ f$ and so $f$ is a fibre-preserving quasi-isometry. Therefore $p:G\rightarrow G/H$ is quasi-isometrically trivial. This completes the proof of Lemma \ref{lem:qitrivial} and hence of Proposition \ref{prop:qitrivial}.
\end{proof}

We now show  a coarse Lipschitz section always exists when the quotient space $G/H$ is hyperbolic.
\begin{lem}\label{lem:lipsection}
	Let $G$ be a finitely generated group containing a commensurated subgroup $\bbZ^n\cong H\alnorm G$. If the quotient space $G/H$ is hyperbolic, then $p:G\rightarrow G/H$ has a coarse Lipschitz section. 
\end{lem}

We  prove Lemma \ref{lem:lipsection}  using $\ell_\infty$-cohomology;  see \cite{gersten98cohomol}. 
A simplicial complex is said to have \emph{bounded geometry} if for every $n$, there is a number $M_n$ such that every $n$-simplex is the face of at most $M_n$ $(n+1)$-simplices. Let $X$ be a bounded geometry contractible simplicial complex and let $(V,\lVert\cdot \rVert)$ be a normed real vector space. We define $C^n_{(\infty)}(X,V)$ to be the set of all simplicial $n$-cochains whose coefficients are  in $V$ and are bounded with respect to the norm on $V$. Since $X$ has bounded geometry,  $\delta(C^n_{(\infty)}(X,V))\subseteq C^{n+1}_{(\infty)}(X,V)$ for each $n$ and so  $(C^\bullet_{(\infty)}(X,V),\delta)$ forms a cochain complex. We   define $H^*_{(\infty)}(X,V)$ to be the cohomology of this cochain complex.

The key ingredient in Lemma \ref{lem:lipsection} is the following theorem of Mineyev:
\begin{thm}[\cite{mineyev00isoperimetric}]\label{thm:mineyev_iso}
	If $X$ is a hyperbolic bounded geometry contractible simplicial complex and $V$ is any real normed vector space,  then $H^2_{(\infty)}(X,V)=0$.
\end{thm}
\begin{rem}
	In the case where $X$ admits a geometric action of a finitely generated hyperbolic group $G$, this follows from \cite[Theorem 0]{mineyev00isoperimetric}. However, careful inspection of the proofs in \cite{mineyev00isoperimetric} demonstrates shows that Mineyev's argument automatically  generalises to the setting where $X^{(1)}$ is hyperbolic. Indeed, the proofs of \cite[Theorems 11 and 12]{mineyev00isoperimetric} are written in terms of a contractible cell complex $X$ that satisfies the following properties:
	\begin{enumerate}
		\item $X^{(1)}$ satisfies the thin triangle condition. This is the definition of  hyperbolicity of $X^{(1)}$.
		\item $X$ has a linear isoperimetric inequality, i.e.\ there is a constant $K$ such that any loop of length $L$ in $X^{(1)}$ can be filled in by a disc diagram of area at most  $KL$ in $X^{(2)}$. This is also equivalent to hyperbolicity of $X^{(1)}$.
		\item $X^{(1)}$ has bounded geometry. (This property is not explicitly mentioned in \cite{mineyev00isoperimetric}, but is used in the definition of the constant $T'$ in the proof of Theorem 12 of  \cite{mineyev00isoperimetric}.)
	\end{enumerate}
\end{rem}
\begin{proof}[Proof of Lemma \ref{lem:lipsection}]
		If $Y$ is a metric space and $r\geq 0$, the Rips complex $P_r(Y)$ is defined  to be the simplicial complex with vertex set $Y$, where $n+1$ distinct elements $\{y_0,\dots,y_n\}$ span an $n$-simplex precisely when $d(y_i,y_j)\leq r$ for all $r$. If $Y$ is quasi-geodesic, then the inclusion $Y\to P_r(Y)$ is a quasi-isometry for $r$ sufficiently large; see e.g. \cite[Exercise 9.27]{drutu2018geometric}. Moreover, if $Y$ is a net in a hyperbolic metric space, then $P_r(Y)$ is contractible for $r$ sufficiently large \cite[Proposition III.$\Gamma$.3.23]{bridson1999metric}.
	
	We retain the notation used in the proof of Lemma \ref{lem:qitrivial}, making the additional assumption $G/H$ is hyperbolic, i.e. the associated relative Cayley graph $\Gamma_{G,H}$ is hyperbolic. In particular this means that, $G/H$ equipped with the subspace metric, is a net in the hyperbolic metric space $\Gamma_{G,H}$. Thus  there exists an $r\geq 1$ such that the  Rips complex $X\coloneqq P_r(G/H)$ is contractible and the inclusion $G/H\rightarrow X$ is a quasi-isometry. Moreover,  $X$ has bounded geometry and is hyperbolic. 
	Let $R=Kr+A$. Fix an arbitrary section $\sigma:G/H\rightarrow G$ of $p:G\rightarrow G/H$ and let $\sigma_g\coloneqq \sigma(gH)$. 
	
	If $d_{G/H}(gH,tH)\leq r$, then as $p:G\rightarrow G/H$ is a $(K,A)$-coarse bundle, we have $\sigma_t\in N_R(gH)$ and so  there exists $k_{g,t}\in G$ and $h_{g,t}\in H$  such that $\sigma_t=\sigma_{g}h_{g,t}k_{g,t}$ with $\lvert k_{g,t}\rvert_S\leq R$. Note that $\Delta(t)=\Delta(\sigma_t)=\Delta(\sigma_{g}h_{g,t}k_{g,t})=\Delta(g)\Delta(k_{g,t})$, so that $\Delta(k_{g,t})=\Delta(g^{-1}t)$.  Let $v_{g,t}\coloneqq \iota(h_{g,t})$. 
	We define a cochain $\alpha\in C^1(X,\bbR^n)$ by $\alpha([gH,tH])=\Delta(g)v_{g,t}$ for each oriented 1-simplex $[gH,tH]$ of $X$. We  first show $\delta \alpha$ is contained in $C^2_{(\infty)}(X,\bbR^n)$.
	
	Note that if $d(gH,tH)\leq r$, $\sigma_g=\sigma_{g}h_{g,t}k_{g,t}h_{t,g}k_{t,g}$. Therefore $\lvert h_{g,t}k_{g,t}h_{t,g}\rvert_S\leq R$.
	By Lemma \ref{lem:projlemma}, \[	\lVert v_{g,t}+\Delta(g^{-1}t) v_{t,g}\rVert = \lVert \iota(h_{g,t})+\Delta(k_{g,t}) \iota(h_{t,g})\rVert\leq B_R\]
	with $B_R$ as in Lemma \ref{lem:projlemma}. Moreover, if $[gH,tH,kH]$ is a simplex of $X$,   then  $\sigma_{g}=\sigma_gh_{g,t}k_{g,t}h_{t,k}k_{t,k}h_{k,g}k_{k,g}$, so $\lvert h_{g,t}k_{g,t}h_{t,k}k_{t,k}h_{k,g}\rvert_S\leq R$. Applying Lemma \ref{lem:projlemma} again, we see 
	\begin{align*}
		\lVert v_{g,t}+\Delta(g^{-1}t) v_{t,k}+\Delta(g^{-1}k) v_{k,g}\rVert & = \lVert v_{g,t}+\Delta(k_{g,t}) v_{t,k}+\Delta(k_{g,t}k_{t,k}) v_{k,g}\rVert\leq B_R.
	\end{align*}
	These two inequalities combine to show \begin{align*}
		\lVert \delta \alpha[gH,tH,kH]\rVert&=\lVert \alpha[tH,kH]-\alpha[gH,kH]+\alpha[gH,tH]\rVert\\
		&=\lVert \Delta(t)v_{t,k}-\Delta(g)v_{g,k}+\Delta(g)v_{g,t}\rVert\\
		&\leq A\lVert \Delta(g^{-1}t)v_{t,k}-v_{g,k}+v_{g,t}\rVert\\
		&\leq A\lVert \Delta(g^{-1}t)v_{t,k}+\Delta(g^{-1}k)v_{k,g}+v_{g,t}\rVert+AB_R\\
		&\leq 2AB_R,
	\end{align*}
	implying that  $\delta\alpha$ is indeed in $C^2_{(\infty)}(X,\bbR^n)$. By Theorem \ref{thm:mineyev_iso}, there exists a  cochain $\beta\in C^1_{(\infty)}(X,\bbR^n)$ such that $\delta\alpha=\delta\beta$. Since $X$ is contractible, there exists $\gamma\in C^0(X,\bbR^n)$ such that $\delta\gamma=\beta-\alpha$. As $N_A(\iota(H))=\bbR^n$, for each $gH\in G/H$ we pick $\gamma_{gH}\in H$ such that $\lVert \iota(\gamma_{gH})-\Delta(g^{-1})\gamma(gH)\rVert\leq A$. We define a section $\hat \sigma:G/H\rightarrow G$ via the formula $\hat \sigma(gH)\coloneqq \sigma_g\gamma_{gH}$.

	We claim that $\hat\sigma$ is coarse Lipschitz. Since $\beta\in C^1_{(\infty)}(X,\bbR^n)$, there exists a constant $C$ such that $\lVert \beta([gH,gsH])\rVert\leq C$ for all $g\in G$ and $s\in S$.  If $g\in G$ and $s\in S$, then \begin{align*}
		\lVert\iota(\gamma_{gH}^{-1}h_{g,gs})+\Delta(k_{g,gs})(\iota(\gamma_{gsH}))\rVert& = \lVert-\iota(\gamma_{gH})+\iota(h_{g,gs}) +\Delta(s)(\iota(\gamma_{gsH}))\rVert\\
		&\leq \lVert -\Delta(g^{-1})\gamma(gH)+v_{g,gs} +\Delta(g^{-1})\gamma(gsH)\rVert+2A^2\\
		& \leq A\lVert -\gamma(gH)+\Delta(g)v_{g,gs} +\gamma(gsH)\rVert+2A^2\\
		& = A\lVert  (\delta\gamma+\alpha)[gH,gsH]\rVert +2A^2\\
		& = A\lVert \beta[gH,gsH]\rVert +2A^2\\&
		\leq AC+2A^2\eqqcolon A_1.
	\end{align*} 
	By Lemma \ref{lem:projlemma},  $\lvert \gamma_{gH}^{-1}h_{g,gs}k_{g,gs}\gamma_{gsH}\rvert_S\leq B_{A_1}$.
	It follows that for all $g\in G$ and $s\in S$, \begin{align*}
		d_G(\hat\sigma(gH),\hat\sigma(gsH))&=d_G(\sigma_g\gamma_{gH},\sigma_{gs}\gamma_{gsH})\\
		& =d_G(\sigma_g\gamma_{gH},\sigma_{g}h_{g,gs}k_{g,gs}\gamma_{gsH})\\
		&= \lvert \gamma_{gH}^{-1}h_{g,gs}k_{g,gs}\gamma_{gsH}\rvert_S\\
		&\leq B_{A_1}.	
	\end{align*}
	Thus $\hat \sigma$ is coarse Lipschitz by Lemma \ref{lem:coarse_lip}.\end{proof}

\subsection{Proof of Theorem \ref{thm:qitocent_main}}
Our starting point in the proof of Theorem \ref{thm:qitocent_main} is the following consequence of Propositions \ref{prop:inducedqa} and \ref{prop:coarsebundle_fibrepres}.
\begin{lem}\label{lem:prod_induced}
		Let $G$ be a finitely generated group quasi-isometric to $\bbR^n\times Q$, where $Q$ is a locally finite vertex-transitive non-elementary  hyperbolic graph. Then the natural quasi-action of $G$ on $\bbR^n\times Q$ induces a quasi-action $G\qa Q$ such that the projection $\bbR^n\times Q\rightarrow Q$ is coarsely equivariant.
\end{lem}
We now want to analyse the induced quasi-action $G\qa Q$. We first show the following:
\begin{prop}\label{prop:qitocenthyp}
	The induced quasi-action $G\qa Q$ as in Lemma \ref{lem:prod_induced} is discretisable and does not fix a point of $\partial Q$.
\end{prop}
We first show how Theorem \ref{thm:qitocent_main} follows from  Proposition \ref{prop:qitocenthyp}:
\begin{proof}[Proof of Theorem \ref{thm:qitocent_main}]
	Let $Q$ be a locally finite vertex-transitive non-elementary  hyperbolic graph and let $G$ be a finitely generated group quasi-isometric to $X=\bbR^n\times Q$. By Lemma \ref{lem:prod_induced} and Proposition \ref{prop:qitocenthyp}, the cobounded quasi-action $G\qa X$ induces a discretisable  quasi-action $G\qa Q$ with coarse stabiliser $H$. Proposition \ref{prop:quasi-conj_coarse bundle} implies there is a fibre-preserving quasi-isometry between the  coarse bundles $G\rightarrow G/H$ and $X\rightarrow Q$.  In particular, $H$ is a finitely generated commensurated subgroup quasi-isometric to $\bbR^n$ and $G/H$ is quasi-isometric to $Q$. Since $H$ is quasi-isometric to $\bbR^n$, it has a finite index subgroup isomorphic to $\bbZ^n$; see for instance \cite[Theorem I.8.40]{bridson1999metric}. By Proposition \ref{prop:coarsestabsarecomm}, such a finite index subgroup is still a coarse stabiliser of  $G\qa Q$, so without loss of generality we may assume $H$ is isomorphic to $\bbZ^n$. Since $G\rightarrow G/H$ is quasi-isometrically trivial, it follows from  Proposition \ref{prop:qitrivial} that the image of the modular homomorphism $G\rightarrow \GL_n(\bbQ)$ is conjugate to $\Orth_n(\bbR)$. 
	
	Conversely, suppose $G$ contains a commensurated subgroup $H\cong \bbZ^n$ such that $G/H$ is quasi-isometric to $Q$, and  the image of the modular homomorphism $G\rightarrow \GL_n(\bbQ)$ is conjugate to $\Orth_n(\bbR)$. Since $Q$ is hyperbolic, it follows from Lemma \ref{lem:lipsection} that $p:G\rightarrow G/H$ has a coarse Lipschitz section. Proposition \ref{prop:qitrivial} implies that $p:G\rightarrow G/H$ is quasi-isometrically trivial. In particular, $G$ is quasi-isometric to $H\times G/H$, hence quasi-isometric to $\bbR^n\times Q$ as required.
\end{proof}
We now turn to the proof of Proposition \ref{prop:qitocenthyp}. We use the following lemma, whose  proof closely follows the proof of \cite[Proposition 5.3]{kleinerleeb2001symmetric}.
\begin{lem}\label{lem:growth_homspace}
	Let $X$ be a  negatively curved Riemannian homogeneous manifold $X$ of dimension at least two. Suppose $\Gamma\leq \Isom(X)$ is a finitely generated cocompact subgroup fixing a point of $\partial X$. Let $U\subseteq \Isom(X)$ be an identity neighbourhood. If \[f_U(k)\coloneqq \left|\{g\in \Gamma\mid g\in U, \lvert g\rvert_\Gamma<k \} \right|,\]
	where $\lvert \cdot \rvert _\Gamma$ is the word norm on $\Gamma$ with respect to a finite generating set, then for every $d>0$  $\limsup_{k\rightarrow \infty} \frac{f_U(k)}{k^d}=\infty$, i.e.\ $f_U$ grows faster than any polynomial function.
\end{lem}
We recall some terminology from \cite{caprace2015amenable} that will be used in the proof of Lemma \ref{lem:growth_homspace}.
If $H$ is a locally compact group, then $\alpha\in \Aut(H)$ is:
\begin{itemize}
	\item \emph{compacting} if there exists a compact subset $K\subseteq H$ such that $\alpha(K)\subseteq K$ and $\cup_{n=0}^\infty{\alpha^{-n}(K)}=H$;
	\item \emph{contracting} if for all $g\in H$, $\lim_{n\rightarrow\infty}\alpha^n(g)=1$.
\end{itemize}
\begin{proof}[Proof of Lemma \ref{lem:growth_homspace}]
	Let $G$ be the closure of  $\Gamma\subseteq \Isom(X)$. It follows from  \cite[Theorem 7.3]{caprace2015amenable} --- see also \cite[Proposition 19.8]{cornulier2018quasi} --- that $G$ has a maximal compact subgroup $W$ and $G/W$ is isomorphic to a semidirect product $H\rtimes_\alpha \bbZ$ or $H\rtimes_\alpha \bbR$, where $H$ is a virtually connected non-compact Lie
	group and $\alpha(1)$ is a compacting automorphism of $H$. Replacing $G$ with $G/W$ if necessary, we may assume $W=1$.
	
	 Let $G^\circ$ be the connected component of $G$.  Since $G^\circ$ is connected and possesses no compact normal subgroup (as $G^\circ$ does not), $G^\circ$ is a connected Lie group.   There are two cases to consider, as in \cite[Proposition 5.3]{kleinerleeb2001symmetric}.\\
	\underline{Case 1: $G^\circ$ is nilpotent}\\
	First note  $G$ cannot be isomorphic to $H\rtimes_\alpha \bbR$, since then $G^\circ$ would be a cocompact subgroup of $G$, hence $G^\circ$ would be a focal hyperbolic group. This cannot occur, since the nilpotent Lie group  $G^\circ$ is unimodular, contradicting \cite[Theorem 7.3]{caprace2015amenable}. Therefore $G$ is isomorphic to $H\rtimes_\alpha \bbZ$ and so $G^\circ=H^\circ$ is a finite index open subgroup of $H$. 
	As $\alpha^{-1}(1)$ is an open subgroup of $G$, it intersects the dense subgroup $\Gamma$. We can thus pick some $g\in \Gamma \cap \alpha^{-1}(1)$ that acts on  $H$ via conjugation  as a compacting automorphism. Since $G^\circ\leq H$ is characteristic, $g$ acts on $G^\circ$ as a compacting automorphism. It follows from \cite[Proposition 6.9]{caprace2015amenable} that $G^\circ$ is simply-connected and the action of $g$  on the nilpotent group  $G^\circ$ is contracting. 
	
	As $G^\circ$ is open,  $\Gamma\cap G^\circ$ is dense in $G^\circ$. Pick the largest $k$ such that the $k$th derived subgroup $A\coloneqq (G^\circ)^{(k)}$ is nontrivial. Since $G^\circ$ is a simply connected nilpotent Lie group,  $A\cong \bbR^s$ for some $s$.  Moreover, since the $k$th derived subgroup $(\Gamma\cap G^\circ)^{(k)}$ is dense in $A$,  $\Gamma\cap A$ is dense in $A$. Let $\lVert \cdot \rVert_A$ be a Euclidean norm on $A$. 
	As $g$ acts on $G^\circ$ as a contracting automorphism, we can thus pick $m$ sufficiently large such that $\lVert g^mag^{-m} \rVert_A\leq \frac{1}{2}\lVert a \rVert_A$ for all $a\in A$. Set $\gamma=g^m$.
	
	  Let $U\subseteq G$ be an identity neighbourhood and pick $r>0$ such that if $a\in A$ and  $\lVert a \rVert_A<r$, then $a\in U$. As $\Gamma\cap A$ is dense in $A$, we can pick a nontrivial $h\in \Gamma$ such that $h\in A$ and  $\lVert h \rVert_A<\frac{r}{2}$.
	  Then for $n\in \bbN$,  the elements 
	 		\[\gamma_{\epsilon_0\dots \epsilon_{n-1}}=h^{\epsilon_0}(\gamma h\gamma^{-1})^{\epsilon_1}(\gamma^2 h\gamma^{-2})^{\epsilon_2}\dots (\gamma^{n-1} h\gamma^{1-n})^{\epsilon_{n-1}}\]
	 with $\epsilon_0,\dots,\epsilon_{n-1}\in \{0,1\}$ are $2^n$ distinct elements of $\Gamma\cap U\cap A$  with word norm in $\Gamma$ at most $n^2(\lvert \gamma\rvert_\Gamma+\lvert h\rvert_\Gamma)$. Therefore, $f_U$ grows faster  than any polynomial function. 
	 
	 \underline{Case 2: $G^\circ$ is not nilpotent}\\
	 We argue as in case 2 of the proof of \cite[Proposition 5.3]{kleinerleeb2001symmetric}. Namely, we define $Z_k$ to be the  term in the upper central series of $G^\circ$ at which $\dim(Z_k)$ achieves its maximum value. Then $G/Z_k$ has discrete centre and $\dim(Z_k)<\dim(G)$ as $G^\circ$ is not nilpotent. We can therefore apply \cite[Lemma 5.5]{kleinerleeb2001symmetric} to $\Gamma\leq G$.
\end{proof}

We also need to following lemma for the case in which $Q$ is a tree. 
\begin{lem}\label{lem:tree_fixedpt}
	Let $T$ be a locally finite infinite-ended tree and let $G$ be a finitely generated group quasi-isometric to $\bbR^n\times T$. Then the induced quasi-action $G\qa T$ as in Lemma \ref{lem:prod_induced} does not fix a point of $\partial T$.
\end{lem}
\begin{proof}
	First note that  $\bbR^n\times T$ is quasi-isometric to $\bbZ^n\times F_2$, where $F_2$ is  a free group of rank two. Thus $G$ is quasi-isometric to a non-amenable finitely generated group, hence is non-amenable. We assume for contradiction that $G$  fixes a point of $\partial T$.

	By \cite[Theorems 1 and 2]{mosher2003quasi}, there exists a locally finite bushy tree $T'$ such that the  induced action $G\qa T$ can be quasi-conjugated to an isometric action $G\qa T'$, and the stabiliser of every vertex and edge of $T'$ is virtually $\bbZ^n$. As $G$ fixes a point of $\partial T$, $G$ also fixes a point of $\partial T'$. Replacing $T'$ with a minimal subtree and considering the quotient graph $G\backslash T'$, we deduce  $G$ must be a strictly ascending HNN extension of a group $H$ that is virtually $\bbZ^n$ and defined by an injective endomorphism $\alpha:H\rightarrow H$. Since $H$ is amenable, so is the normal subgroup $H_\infty\coloneqq \cup_{i\in \bbN}\alpha^{-i}(H)$. Since  $G/H_\infty\cong \bbZ$ is an infinite cyclic subgroup generated by the image of the stable letter,  $G$ is amenable. This is the desired contradiction.  
\end{proof}
\begin{proof}[Proof of Proposition \ref{prop:qitocenthyp}]
	We suppose for contradiction $G\qa Q$ fixes a point of $\partial Q$. By Corollary \ref{cor:trichot_hyp_fixpt},  $G\qa Q$ can be quasi-conjugated to an isometric action $\rho:G\rightarrow \Isom(Z)$, where $Z$ is either a negatively curved homogeneous space, a pure millefeuille space, or a regular tree of finite valency greater than two. Since $Q$ is a locally finite vertex-transitive graph, Proposition \ref{prop:millefeuille} says  $Z$ is not a pure millefeuille space. Moreover, Lemma \ref{lem:tree_fixedpt} precludes the case $Z$ is a tree. Therefore,  $Z$  must be a negatively curved homogeneous space. 
	
	To apply Lemma \ref{lem:growth_homspace}, we need to fix some constants. Let $X\coloneqq \bbR^n\times Q$ and  suppose $G\qa X$ is an $A_1$-cobounded  proper quasi-action on $X$. We also assume the projection $\pi:X\rightarrow Q$ is $A_2$-coarsely equivariant and that  $f:Q\rightarrow Z$ is a $(K_3,A_3)$-quasi-conjugacy from the quasi-action $G\qa Q$ to the isometric action $G\curvearrowright Z$. 
	
	Pick a basepoint $x_0\in X$, set $q_0\coloneqq\pi(x_0)$ and $z_0\coloneqq f(\pi(x_0))$. Set $A_4\coloneqq \max(K_3(2A_3+1)+A_2,A_1)$ and let \[V\coloneqq \{g\in G\mid d_Q(\pi(g\cdot x_0),q_0)\leq A_4\}.\] We claim $V\subseteq G$, equipped with the subspace metric, is quasi-isometric to $\bbR^n$. Since the map  $h:G\rightarrow X$ given by $g\mapsto g\cdot x_0$ is a quasi-isometry, it is sufficient to show that $h(V)$ has finite Hausdorff distance from $D_{q_0}\coloneqq \pi^{-1}(q_0)$, which is isometric to $\bbR^n$. By definition $h(V)\subseteq N_{A_4}(D_{q_0})$. Since $G\qa X$ is $A_1$-cobounded, it follows for each $x\in D_{q_0}$, there is some $g\in G$ such that $d(g\cdot x_0,x)\leq A_1$.  Therefore, $d_Q(\pi(g\cdot x_0),q_0)\leq A_1\leq A_4$, so $g\in V$ and hence $x\in N_{A_1}(h(V))$. As $D_{q_0}\subseteq N_{A_1}(h(V))$, $h(V)$ and $D_{q_0}$ are at finite Hausdorff distance. It follows that $r(k)\coloneqq \left| \{g\in V\mid \lvert g\rvert_G\leq k\}\right|$ grows polynomially with $k$.
	
	Let $U\coloneqq \{\phi\in \Isom(Z)\mid d_Z(\phi(z_0),z_0)< 1\}$, which is an identity neighbourhood in $\Isom(Z)$. We claim $\rho^{-1}(U)\subseteq V$. Indeed, suppose $g\in \rho^{-1}(U)$. Then \begin{align*}
		d_Q(\pi(g\cdot x_0),q_0)&\leq d_Q(g\cdot q_0,q_0)+A_2\\
		&\leq K_3(d_Z(f(g\cdot q_0),z_0)+A_3)+A_2\\
		& \leq K_3(d_Z(\rho(g)(z_0),z_0)+2A_3)+A_2\\
		& \leq K_3(1+2A_3)+A_2\\
		&\leq A_4
	\end{align*}
	so that $g\in V$. Lemma \ref{lem:growth_homspace} then guarantees that $r(k)$ grows faster than any polynomial, contradicting the previous paragraph.

	We have thus shown that  $G\qa Q$ cannot fix a point of $\partial Q$. By Corollary \ref{cor:dichothyp_main},  either $G\qa Q$ is discretisable or $G\qa Q$ is quasi-conjugate to a non-discrete isometric action on a rank one symmetric space $Z$. The latter cannot occur, as shown in the proof of \cite[Theorem 1.1]{kleinerleeb2001symmetric} (see also \cite[Proposition 5.3]{kleinerleeb2001symmetric}).
\end{proof}

\subsection{Consequences of Theorem \ref{thm:qitocent_main}}

\begin{thm}\label{thm:zbyhyp_main}
	Let $Q$ be a locally finite vertex-transitive non-elementary  hyperbolic graph. A finitely generated group $G$ is quasi-isometric to $\bbR\times Q$ if and only if $G$ contains an infinite cyclic normal subgroup $\bbZ\cong H\vartriangleleft G$  such that $G/H$ is quasi-isometric to $Q$.  
\end{thm}
\begin{proof}
	$\implies$: Suppose $G$ is quasi-isometric to $\bbR\times Q$. Then Theorem \ref{thm:qitocent_main} ensures $G$ contains an infinite cyclic commensurated subgroup $\langle a\rangle=H$ such that the quotient space $G/H$ is quasi-isometric to $Q$, and the image of  $\Delta:G\rightarrow \GL_1(\bbR)$ is conjugate to a subgroup of $\Orth_1(\bbR)$.  Since $\Orth_1(\bbR)=\{\pm 1\}$ is a normal subgroup of $\GL_1(\bbR)$, it follows that $\Delta(G)$ is contained in $\{\pm 1\}$. Thus for each $g\in G$, there is some  sufficiently large $N_g$ such that $ga^{N_g}g^{-1}=a^{\pm N_g}$. Let $S=\{s_1,\dots,s_n\}$ be a generating set of $G$, and let $N=\lcm(N_{s_1},\dots,N_{s_n})$. Therefore $s_i a^N s_i^{-1}=a^{\pm N}$ for  $i=1,\dots, n$, whence $g a^N g^{-1}=a^{\pm N}$ for all $g\in G$. Therefore $H_0\coloneqq \langle a^N\rangle$ is a normal subgroup of $G$. As $H_0$ is a finite index subgroup of $H$, the quotient group $G/H_0$ is quasi-isometric to the quotient space $G/H$, hence to $Q$. 
	
	$\impliedby$: Suppose $G$ contains an infinite cyclic normal subgroup $H\vartriangleleft G$ with $G/H$ quasi-isometric to $Q$. Since  $\Aut(H)\cong \bbZ/2\bbZ$, the image of the map $G\to \Aut(H)\cong \{\pm 1\}$ is conjugate to a subgroup of $\Orth_1(\bbR)$. Therefore Theorem \ref{thm:qitocent_main} implies $G$ is quasi-isometric to $\bbR\times Q$.
\end{proof}

Theorem \ref{thm:zbyhyp_main} implies:
\begin{cor}\label{cor:zbyhyp}
	A finitely generated group that is quasi-isometric to a $\bbZ$-by-hyperbolic group is also $\bbZ$-by-hyperbolic.
\end{cor}

The groups constructed by Leary--Minasyan demonstrate that  Theorem \ref{thm:zbyhyp_main} cannot in general be strengthened to deduce that groups quasi-isometric to $\bbR^n\times Q$ contain a normal free abelian subgroup of rank $n$. The following proposition gives necessary and sufficient criteria under which the commensurated subgroup  in Theorem \ref{thm:qitocent_main} contains a finite index normal subgroup. The author would like to thank Sam Hughes for pointing out a result of Valiunas used in the proof of the following proposition.
\begin{prop}\label{prop:qitocent_main}
	Let $Q$ be a locally finite vertex-transitive non-elementary  hyperbolic graph and let $G$ be a finitely generated group quasi-isometric to $\bbR^n\times Q$. Let $H\cong \bbZ^n$ be a commensurated subgroup as in Theorem \ref{thm:qitocent_main}. The following are equivalent:
	\begin{enumerate}
		\item $H$  is weakly separable;\label{item:qitocent1}
		\item $H$ has a finite index subgroup  that is normal in $G$;\label{item:qitocent2}
		\item a finite index subgroup of $G$ centralises a finite index subgroup of $H$;\label{item:qitocent3}
		\item $G$ is  biautomatic.\label{item:qitocent4}
	\end{enumerate}
\end{prop}
\begin{proof}
	(\ref{item:qitocent1})$\iff$(\ref{item:qitocent2}) and (\ref{item:qitocent3})$\implies$(\ref{item:qitocent2}) follow from   Lemma \ref{lem:weaksep}.
	
	(\ref{item:qitocent2})$\implies$(\ref{item:qitocent3}): Suppose  $H'\leq H$ is a finite index subgroup of $H$ that is normal in $G$. The image of the modular homomorphism $\Delta:G\to \Aut(H)\cong \GL_n(\bbZ)$ is both a discrete subgroup of $GL_n(\bbR)$ and conjugate to a subgroup of $\Orth_n(\bbR)$ by Theorem \ref{thm:qitocent_main}. Thus $\im(\Delta) $ is finite, ensuring $H'$ is centralised by the finite index subgroup $\ker(\Delta)\leq G$.
	
	(\ref{item:qitocent3})$\implies$(\ref{item:qitocent4}): This follows from a theorem of Neumann--Reeves, which shows a virtually central extension of a hyperbolic group is biautomatic   \cite[Theorem 1.3]{neumannreeves97central}.
	
	(\ref{item:qitocent4})$\implies$(\ref{item:qitocent3}):  Valiunas showed if $G$ is biautomatic and $H\alnorm G$ is a  finitely generated commensurated abelian subgroup, then a finite index subgroup of $G$ centralises a finite index subgroup of $H$ \cite[Theorem 1.2]{valiunas2021leary}.
\end{proof}
Finally, we show the conclusion of Theorem \ref{thm:qitocent_main} can be strengthened under the additional hypothesis that $G$ is residually finite:
\begin{thm}\label{thm:resfinite_cent}
	Let $Q$ be a locally finite vertex-transitive non-elementary  hyperbolic graph. Let $G$ be a finitely generated residually finite group quasi-isometric to $\bbR^n\times Q$. Then $G$ contains a normal subgroup $\bbZ^n\cong H\vartriangleleft G$ such that $G/H$ is quasi-isometric to $Q$. Moreover, $H$ is centralised by a finite index subgroup of $G$.
\end{thm}
\begin{proof}
	By Theorem \ref{thm:qitocent_main}, $G$ contains a commensurated subgroup $\bbZ^n\cong H\alnorm G$. It follows from  the proof of Theorem \ref{thm:qitocent_main} that the induced quasi-action $G\qa Q$ as in Lemma \ref{lem:prod_induced} has coarse stabiliser $H$ and does not fix a point of $\partial Q$. By Proposition \ref{prop:cstab homspace},  $G\qa Q$ is quasi-conjugate to $G\curvearrowright G/H$ and the action of $G$ on $G/H$ does not fix a point of $\partial (G/H)$. It follows that any abelian normal subgroup of $G$ has bounded orbits in $G/H$, hence is contained in finitely many left $H$-cosets.
	
	Let $\overline{H}$ be the profinite closure of $H$, i.e.\ the smallest separable subgroup of $G$ containing $H$. Since $G$ is residually finite, the profinite topology is Hausdorff. Therefore,  $\overline{H}$ is the closure of an abelian subgroup of a Hausdorff topological group, hence is also abelian; see e.g. \cite[Corollary 5.3]{hewittross1979abstract}. Moreover, \cite[Corollary 6]{caprace_kropholler_reid_wesolek_2020} ensures $\overline{H}$ contains a finite index subgroup $N$ that is normal in $G$. Thus $N$ is a normal abelian subgroup of $G$, so by the previous paragraph, $N$ is contained in finitely many left $H$-cosets. Since $H\leq \overline{H}$,  $N\cap H$ is a finite index subgroup of $H$, hence is commensurable to $H$. Since $H$ is commensurable to a normal subgroup of $G$, the result follows from Proposition \ref{prop:qitocent_main}.
\end{proof}

\bibliography{bibliography/bibtex} 
\bibliographystyle{amsalpha}
\end{document}